\documentclass[a4paper]{amsart}                

\usepackage[latin1]{inputenc}                  
\usepackage[T1]{fontenc}                       
\usepackage[spanish,english]{babel}            
\usepackage[pdftex]{graphicx}                  
\usepackage{amsmath,amssymb,amsthm}            
\usepackage{latexsym}                          
\usepackage{delarray}                          
\usepackage{bbm}                               
\usepackage{hyperref}                          
\usepackage[pdftex,usenames,dvipsnames]{color} 
\usepackage{calrsfs,eufrak,mathrsfs,upgreek}   
\usepackage{bbding,trfsigns}                   
\usepackage{wasysym}                           
\usepackage{datetime}                          

\DeclareMathAlphabet{\mathpzc}{OT1}{pzc}{m}{it} 

\newtheorem{Th}{Theorem}[section]              
\newtheorem{Cor}{Corollary}[section]
\newtheorem{Rm}{Remark}[section]
\newtheorem{Cl}{Claim}
\newtheorem{Prop}{Proposition}[section]

\title[Harmonic analysis and Bessel operators]{Harmonic analysis operators associated with multidimensional Bessel operators}

\author[J.J. Betancor]{J.J. Betancor}
\address{Departamento de Análisis Matemático\\
Universidad de la Laguna\\
Campus de Anchieta, Avda. Astrofísico Francisco Sánchez, s/n\\
38271 La Laguna (Sta. Cruz de Tenerife), Spain}
\email{jbetanco@ull.es}

\author[A.J. Castro]{A.J. Castro}
\address{Departamento de Análisis Matemático\\
Universidad de la Laguna\\
Campus de Anchieta, Avda. Astrofísico Francisco Sánchez, s/n\\
38271 La Laguna (Sta. Cruz de Tenerife), Spain}
\email{alecas000@gmail.com}

\author[J. Curbelo]{J. Curbelo}
\address{Instituto de Ciencias Matemáticas\\
Consejo Superior de Investigaciones Científicas\\
Serrano 121\\
28006 Madrid, Spain} \email{jezabel.curbelo@icmat.es}

\thanks{The first author is partially supported by MTM2007/65609. The second author is supported by a grant for Master studies of "la Caixa". The third author is supported by a grant JAE-Predoc of the CSIC (Spain).}

\begin{document}

  \maketitle                                  

  \begin{abstract}
    In this paper we establish that the maximal operator and the Littlewood-Paley g-function associated with the heat semigroup defined by multidimensional Bessel
    operators are of weak type (1,1). Also, we prove that Riesz transforms in the multidimensional Bessel setting are of strong type (p,p), for every $1<p<\infty$,
    and of weak type (1,1).
  \end{abstract}

  \section{Introduction}

    Harmonic analysis operators (maximal operators, Riesz transforms, Littlewood-Paley g-func\-tions, multipliers, \dots) associated with orthogonal expansions in discrete
    and continuous settings have been studied since the 60's of the last century, but in the last years it has been a very active area and many authors have studied this
    topic (see, for instance, \cite{AFS}, \cite{CC}, \cite{CH}, \cite{CNS}, \cite{CR}, \cite{GMMST}, \cite{GHMSTV}, \cite{MST}, \cite{M}, \cite{MS},  \cite{NS}, \cite{NSt},
    \cite{Sa}, \cite{Sj}, \cite{ST} and \cite{Th}). Muckenhoupt and Stein \cite{MS} made a deep study about harmonic analysis in the one dimensional ultraspherical and Bessel
    settings. We are in\-te\-res\-ted in this paper in the $L^p$-boundedness properties for the maximal operator and Littlewood-Paley g-function associated with the heat
    semigroup for the multidimensional Bessel operators and the Riesz transforms in the multidimensional Bessel context. As far as we know harmonic analysis operators in
    the Bessel settings had been studied always in the one dimensional case. After \cite{MS}, Andersen \cite{An}, Kerman (\cite{K1} and \cite{K2}) and Andersen and
    Kerman \cite{AK} established $L^p$-weighted inequalities for the Bessel-Riesz transforms. Stempak \cite{Stem} studied Littlewood-Paley g-functions and Mihlin-Hörmander
    multipliers in the Bessel setting. Recently, in \cite{BFMR}, \cite{BFMT2}, \cite{BFMT1}, \cite{BFS}, \cite{BHNV}, and \cite{BS} the one-dimensional harmonic
    analysis in the Bessel context has been completed investigating properties of g-functions and Riesz transforms of every order and maximal operators associated with
    Poisson and heat semigroups.\\

    We consider the n-dimensional Bessel operator $\Delta_{\lambda_1, \dots, \lambda_n}$ defined by
    $$\Delta_{\lambda_1, \dots, \lambda_n}=\sum_{j=1}^n \Delta_{\lambda_j}^{(j)},$$
    where $\Delta_{\lambda_j}^{(j)}=-x_j^{-2\lambda_j}\frac{\partial}{\partial x_j}\left( x_j^{2\lambda_j} \frac{\partial}{\partial x_j}\right)$,
    and $\lambda_j > -1/2$, for every $j=1, \dots, n$, $n \geq 2$. The heat semigroup generated by $\Delta_{\lambda_1, \dots, \lambda_n}$ is represented by
    $\{W_t^{\lambda_1, \dots, \lambda_n}\}_{t>0}$. This semigroup is a symmetric diffusion semigroup in the sense of \cite{St1} with respect to the measure
    $m_{\lambda_1, \dots, \lambda_n}$ defined by $\overset{n}{\underset{j=1}{\prod}}x_j^{2\lambda_j}dx$ on $(0,\infty)^n$. Then, according to
    \cite[p. 73]{St1} the maximal operator
    $$W_*^{\lambda_1, \dots, \lambda_n}(f)=\sup_{t>0} \left \lvert W_t^{\lambda_1, \dots, \lambda_n}(f) \right \rvert$$
    is bounded from $L^p((0,\infty)^n,\overset{n}{\underset{j=1}{\prod}} x_j^{2\lambda_j}dx)$ into itself, for every $1<p<\infty$. Inspired in the ideas developed
    by Nowak and Sjögren in \cite{NS1} we establish that $W_*^{\lambda_1, \dots, \lambda_n}$ is of weak type (1,1) with respect to the measure
    $m_{\lambda_1, \dots, \lambda_n}$.

    \begin{Th}\label{maximal}
      Let $\lambda_j>-1/2$, $j=1, \dots, n$. The maximal operator $W_*^{\lambda_1,\dots, \lambda_n}$ is bounded from
      $L^1((0,\infty)^n,\overset{n}{\underset{j=1}{\prod}} x_j^{2\lambda_j}dx)$ into $L^{1,\infty}((0,\infty)^n,\overset{n}{\underset{j=1}{\prod}} x_j^{2\lambda_j}dx)$.
    \end{Th}

    Since for every $f \in C_c^\infty((0,\infty)^n)$, the space of $C^\infty$-functions on $(0,\infty)^n$ that have compact support,
    $$\lim_{t \rightarrow 0^+} W_t^{\lambda_1, \dots, \lambda_n}(f)(x)=f(x), \quad x \in (0,\infty)^n, $$
    and $C_c^\infty((0,\infty)^n)$ is a dense subspace of $L^p((0,\infty)^n,\overset{n}{\underset{j=1}{\prod}} x_j^{2\lambda_j}dx)$, for every $1 \leq p < \infty$,
    standard arguments allow us to deduce from the $L^p$-boundedness properties of the maximal operator $W_*^{\lambda_1, \dots, \lambda_n}$ the following result.

    \begin{Cor}
      Let $\lambda_j >-1/2$, $j=1, \dots, n$, and $1 \leq p < \infty$. For every $f \in L^p((0,\infty)^n,$ $\overset{n}{\underset{j=1}{\prod}} x_j^{2\lambda_j}dx)$,
      $$\lim_{t \rightarrow 0^+} W_t^{\lambda_1, \dots, \lambda_n}(f)(x)=f(x), \quad \text{a.e. } x \in (0,\infty)^n. $$
    \end{Cor}

    According to  \cite[p. 67]{St1} the semigroup $\{W_t^{\lambda_1, \dots, \lambda_n}\}_{t>0}$ admits an analytic extension to
    $\Omega = \{t \in \mathbb{C} : \lvert Arg(t) \rvert < \pi/4 \}$. The Littlewood-Paley g-function of the first order associated with
    $\{W_t^{\lambda_1, \dots, \lambda_n}\}_{t>0}$ is defined by
    $$g^{\lambda_1, \dots, \lambda_n}(f)(x)
      =\left \{ \int_0^\infty \left\lvert t \frac{\partial}{\partial t} W_t^{\lambda_1, \dots, \lambda_n}(f)(x)  \right \rvert^2 \frac{dt}{t} \right \}^{1/2},
       \quad x \in (0, \infty)^n.$$
    By \cite[p. 111]{St1}, $g^{\lambda_1, \dots, \lambda_n}$ defines a bounded operator from
    $L^p((0,\infty)^n,\overset{n}{\underset{j=1}{\prod}} x_j^{2\lambda_j}dx)$ into itself, for every $1<p<\infty$. We complete this result analyzing the behavior
    of $g^{\lambda_1, \dots, \lambda_n}$ on $L^1((0,\infty)^n,$ $\overset{n}{\underset{j=1}{\prod}} x_j^{2\lambda_j}dx)$.

    \begin{Th}\label{gfunction}
      Let $\lambda_j>-1/2$, $j=1, \dots, n$. Then, the operator $g^{\lambda_1, \dots, \lambda_n}$ is bounded from
      $L^1((0,\infty)^n,\overset{n}{\underset{j=1}{\prod}} x_j^{2\lambda_j}dx)$ into $L^{1,\infty}((0,\infty)^n,\overset{n}{\underset{j=1}{\prod}} x_j^{2\lambda_j}dx)$.
    \end{Th}

    The Bessel operator $\Delta_\lambda$ can be factorized as follows
    $$\Delta_\lambda = -x^{-2\lambda}\frac{d}{dx}\left(x^{2\lambda}\frac{d}{dx}\right)= \left(\frac{d}{dx}\right)^* \frac{d}{dx},$$
    where $(\dfrac{d}{dx})^*$ denotes the formal adjoint of $\dfrac{d}{dx}$ in $L^2((0,\infty)^n,\overset{n}{\underset{j=1}{\prod}} x_j^{2\lambda_j}dx)$.
    According to \cite[Section 16]{MS} and \cite{St1} we define, for every $i=1, \dots, n$, the i-th Riesz transform $R_i^{\lambda_1,\dots, \lambda_n}$
    associated with $\Delta_{\lambda_1, \dots, \lambda_n}$ by
    $$R_i^{\lambda_1,\dots, \lambda_n} f = \frac{\partial}{\partial x_i} \Delta_{\lambda_1, \dots, \lambda_n}^{-1/2} f, \quad f \in C_c^\infty ((0,\infty)^n).$$
    Here $\Delta_{\lambda_1, \dots, \lambda_n}^{-1/2}$ represents the negative square root of $ \Delta_{\lambda_1, \dots, \lambda_n}$ whose definition will be specified
    later (see Section~4). $L^p$-boundedness properties of the Riesz transforms $R_i^{\lambda_1, \dots, \lambda_n}$, $i=1, \dots, n$, are established in the following.

    \begin{Th}\label{Riesz}
      Let $\lambda_j>-1/2$, $j=1, \dots, n$, and $i=1, \dots, n$. For every $f \in C_c^\infty((0,\infty)^n)$, $\Delta_{\lambda_1, \dots, \lambda_n}^{-1/2}f$ admits
      derivative with respect to $x_i$ on almost all $(0,\infty)^n$ and
      $$\frac{\partial}{\partial x_i} \Delta_{\lambda_1, \dots, \lambda_n}^{-1/2}f(x)
        = \lim_{\varepsilon \rightarrow 0^+} \int_{\lvert x - y \rvert > \varepsilon}
          R_i^{\lambda_1, \dots, \lambda_n}(x,y) f(y) \prod_{j=1}^n y_j^{2\lambda_j} dy, \ \text{a.e. } x \in (0,\infty)^n, $$
      where
      $$ R_i^{\lambda_1, \dots, \lambda_n}(x,y)
         = \frac{1}{\sqrt{\pi}} \int_0^\infty \frac{\partial}{\partial x_i} W_t^{\lambda_1, \dots, \lambda_n}(x,y) \frac{dt}{\sqrt{t}}, \quad x, y \in (0,\infty)^n,$$
      and $W_t^{\lambda_1, \dots, \lambda_n}(x,y)$ represents the kernel of the operator $W_t^{\lambda_1, \dots, \lambda_n}$, for every $t>0$ (see \eqref{D1} for
      definitions).\\

      Moreover, the maximal operator $R_{i,*}^{\lambda_1, \dots, \lambda_n}$ defined by
      $$R_{i,*}^{\lambda_1, \dots, \lambda_n}(f)(x)
        = \sup_{\varepsilon > 0} \left \lvert
          \int_{\lvert x - y \rvert > \varepsilon} R_i^{\lambda_1, \dots, \lambda_n}(x,y) f(y) \prod_{j=1}^n y_j^{2\lambda_j} dy \right \rvert, \quad x \in (0,\infty)^n,$$
      is bounded from $L^p((0,\infty)^n,\overset{n}{\underset{j=1}{\prod}} x_j^{2\lambda_j}dx)$ into itself, for every $1<p<\infty$, and from
      $L^1((0,\infty)^n,$ $\overset{n}{\underset{j=1}{\prod}} x_j^{2\lambda_j}dx)$ into $L^{1,\infty}((0,\infty)^n,\overset{n}{\underset{j=1}{\prod}} x_j^{2\lambda_j}dx)$.
      The Riesz transform $R_i^{\lambda_1, \dots, \lambda_n}$ can be extended to $L^p((0,\infty)^n,\overset{n}{\underset{j=1}{\prod}} x_j^{2\lambda_j}dx)$ by
      $$R_i^{\lambda_1, \dots, \lambda_n}(f)(x)
        = \lim_{\varepsilon \rightarrow 0^+} \int_{\lvert x - y \rvert > \varepsilon}
          R_i^{\lambda_1, \dots, \lambda_n}(x,y) f(y) \prod_{j=1}^n y_j^{2\lambda_j} dy, \ \text{a.e. } x \in (0,\infty)^n,$$
      for every $f \in L^p((0,\infty)^n,\overset{n}{\underset{j=1}{\prod}} x_j^{2\lambda_j}dx)$, $1 \leq p < \infty$, that is a bounded operator
      from $L^p((0,\infty)^n,$ $\overset{n}{\underset{j=1}{\prod}} x_j^{2\lambda_j}dx)$ into itself, $1<p<\infty$, and from
      $L^1((0,\infty)^n,\overset{n}{\underset{j=1}{\prod}} x_j^{2\lambda_j}dx)$ into $L^{1,\infty}((0,\infty)^n,\overset{n}{\underset{j=1}{\prod}} x_j^{2\lambda_j}dx)$.
    \end{Th}

    We now recall some definitions and properties that will be useful in the sequel. If $J_\nu$ denotes the Bessel function of the first kind and order $\nu >-1$,
    we have that, when $\lambda>-1/2$,
    $$\Delta_{\lambda,x} \left( (xz)^{-\lambda+1/2} J_{\lambda - 1/2}(xz) \right)
      = z^2 (xz)^{-\lambda+1/2} J_{\lambda - 1/2}(xz), \quad x,z \in (0,\infty) .$$
    Then, the heat semigroup $\{W_t^\lambda\}_{t>0}$ generated by $\Delta_{\lambda,x}$ is given by
    $$W_t^\lambda(f)(x) = \int_0^\infty W_t^\lambda(x,y)f(y)  y^{2\lambda} dy,$$
    where
    $$W_t^\lambda(x,y)
      = \int_0^\infty e^{-tz^2} (xz)^{-\lambda+1/2} J_{\lambda - 1/2}(xz) (yz)^{-\lambda+1/2} J_{\lambda - 1/2}(yz) z^{2\lambda} dz, \ t,x,y \in (0,\infty).$$
    Moreover, according to \cite[p. 195]{W}, we can write
    $$W_t^\lambda(x,y)=\frac{(xy)^{-\lambda+1/2}}{2t}I_{\lambda-1/2}\left( \frac{x y}{2t}\right) e^{-(x^2+y^2)/4t}, \quad t,x,y \in (0,\infty),$$
    where $I_\nu$ represents the modified Bessel function of the first kind and order $\nu > -1$.\\

    The heat semigroup $\{W_t^{\lambda_1, \dots, \lambda_n}\}_{t>0}$ associated with the multidimensional Bessel operator $\Delta_{\lambda_1, \dots, \lambda_n}$
    is defined by
    $$W_t^{\lambda_1, \dots, \lambda_n}(f)(x) = \int_{(0,\infty)^n} W_t^{\lambda_1, \dots, \lambda_n}(x,y)f(y) \prod_{j=1}^n y_j^{2\lambda_j} dy,$$
    being
    \begin{equation}\label{D1}
      W_t^{\lambda_1, \dots, \lambda_n}(x,y) = \prod_{j=1}^n W_t^{\lambda_j}(x_j,y_j), \quad x=(x_1,\dots, x_n),\,\,\, y=(y_1, \dots, y_n) \in (0,\infty)^n.
    \end{equation}

    The following properties of the modified Bessel function $I_\nu$, $\nu>-1$, will be used (see \cite{L} and \cite{W}). $I_\nu$ admits the following series representation
    \begin{equation}\label{S1}
      I_\nu (z) = \sum_{k=0}^\infty \frac{\left( z/2 \right)^{2k+\nu}}{k! \ \Gamma(\nu + k +1)}, \quad z \in (0,\infty).
    \end{equation}
    Then,
    \begin{equation}\label{E1}
      I_\nu (z) \sim \frac{1}{2^\nu \Gamma(\nu +1)} z^\nu, \quad \text{as } z \rightarrow 0^+.
    \end{equation}
    Also,
    \begin{equation}\label{E2}
      I_\nu (z) = \frac{e^z}{\sqrt{2 \pi z}}\left( \sum_{k=0}^m (-1)^k [\nu,k] (2z)^{-k} + O(\frac{1}{z^{m+1}}) \right), \quad m \in \mathbb{N}, \ z \in (0,\infty),
    \end{equation}
    where $[\nu,0]=1$ and
    $$[\nu,k]=\frac{(4\nu^2-1)(4\nu^2-3^2) \cdots (4\nu^2-(2k-1)^2)}{2^{2k}\Gamma(k+1)}, \quad k=1,2, \dots \ .$$
    From \eqref{S1} it is easy to deduce that
    \begin{equation}\label{E3}
      \frac{d}{dz}\left( z^{-\nu} I_\nu(z) \right) = z^{-\nu} I_{\nu+1}(z), \quad z \in (0,\infty).
    \end{equation}
    Asymptotic expansion \eqref{E2} allows us to connect the operators (maximal operators, g-functions and Riesz transforms) associated with
    $\Delta_{\lambda_1, \dots, \lambda_n}$ with the corresponding operators in the classical Euclidean setting in a local region that is close to the diagonal
    in $(0,\infty)^n$. This is a crucial point in the proof of our theorems. The proof of Theorems~\ref{gfunction} and \ref{Riesz} is more involved that the one of
    the Theorem~\ref{maximal} because the maximal operator $W_*^{\lambda_1, \dots, \lambda_n}$ is positive while Littlewood-Paley g-function and Riesz transforms are
    not positive operators.\\

    This paper is organized as follows. In Section~2 we prove Theorem~\ref{maximal}, Theorem~\ref{gfunction} is showed in Section~3 and a proof of Theorem~\ref{Riesz}
    is presented in Section~4.\\

    Throughout this paper we will use repeatedly without saying it that, for every $k \in \mathbb{N}$, there exists $C_k>0$ such that $z^k e^{-z} \leq C_k$, $z>0$.
    Always by $C$ we denote a suitable positive constant that can change from one line to the other one.

  \section{Proof of Theorem~\ref{maximal}}

    In this section we prove that the maximal operator
    $$W_*^{\lambda_1, \dots, \lambda_n}(f)(x)
      = \sup_{t>0} \left \lvert \int_{(0,\infty)^n} \prod_{j=1}^n W_t^{\lambda_j}(x_j,y_j)f(y) \prod_{j=1}^n  y_j^{2\lambda_j} dy  \right \rvert,
      \quad x \in (0,\infty)^n, $$
    where $\lambda_j > -1/2$, $j=1, \dots, n$, is bounded from $ L^1((0,\infty)^n,\overset{n}{\underset{j=1}{\prod}} x_j^{2\lambda_j}dx)$
    into $ L^{1,\infty}((0,\infty)^n,$ $\overset{n}{\underset{j=1}{\prod}} x_j^{2\lambda_j}dx)$.\\

    By $\mathbb{W}_t$ we denote the classical heat kernel in one
    dimension, that is,
    $$\mathbb{W}_t(x,y)=\frac{1}{2\sqrt{\pi t}} e^{-(x-y)^2/4t}, \quad t \in (0,\infty), \ x,y \in \mathbb{R}.$$
    We will split the region $(0,\infty)^n$ of integration in
    several parts and study each of the operators that appear
    associated with every part.\\

    Firstly we analyze some auxiliary operators that will be
    useful in the sequel. The Hardy type operator $H_\infty$
    defined by
    $$H_\infty(g)(x)=\int_x^\infty \frac{g(y)}{y} dy, \quad x \in (0,\infty),$$
    is bounded from $L^1((0,\infty),x^{2\alpha}dx)$ into itself
    when $\alpha>-1/2$. Then, for every $k \in \mathbb{N}$, the
    operator
    $$H_\infty^k(g)(x)
      =\int_{x_1}^\infty \cdots  \int_{x_k}^\infty \frac{g(y_1,y_2, \dots, y_k)}{y_1 y_2  \dots  y_k} dy_k \cdots dy_1,
      \quad x=(x_1, \dots, x_k) \in (0,\infty)^k,$$
    is bounded from $\displaystyle L^1((0,\infty)^k, \overset{k}{\underset{j=1}{\prod}} x_j^{2\alpha_j} dx)$ into
    itself provided that  $\alpha_j>-1/2$, $j=1, \dots, k$.\\

    Suppose that $k \in \mathbb{N}$ and $\alpha_j>-1/2$, $j=1, \dots, k$.
    We define the operator $L_{\alpha_1, \dots, \alpha_k}$ by
    $$L_{\alpha_1, \dots, \alpha_k}(g)(x)
      =\frac{1}{\left( \overset{k}{\underset{j=1}{\sum}} x_j \right)^{2 \overset{k}{\underset{j=1}{\sum}} (\alpha_j + 1/2)}}
      \int_0^{x_1} \cdots \int_0^{x_k} g(y) \prod_{j=1}^k y_j^{2\alpha_j} dy,
      \quad x \in (0,\infty)^k.$$
    $L_{\alpha_1, \dots, \alpha_k}$ is bounded from $ L^1((0,\infty)^k,\overset{k}{\underset{j=1}{\prod}} x_j^{2\alpha_j}dx)$
    into $ L^{1,\infty}((0,\infty)^k,\overset{k}{\underset{j=1}{\prod}} x_j^{2\alpha_j}dx)$. Indeed, let $g \in
    L^1((0,\infty)^k,\overset{k}{\underset{j=1}{\prod}}
    x_j^{2\alpha_j}dx)$
    and $\gamma >0$. By denoting $m_{\alpha_1, \dots, \alpha_k}$
    the measure $\overset{k}{\underset{j=1}{\prod}}x_j^{2\alpha_j}dx$ we have that
    \begin{align*}
      m&_{\alpha_1, \dots, \alpha_k}  (\{x \in (0,\infty)^k : \lvert L_{\alpha_1, \dots, \alpha_k}(g)(x)  \rvert > \gamma\})\\
       & \leq  m_{\alpha_1, \dots, \alpha_k} \Big(\Big\{x \in (0,\infty)^k :
          \left( \overset{k}{\underset{j=1}{\sum}} x_j \right)^{-2 \overset{k}{\underset{j=1}{\sum}} (\alpha_j + 1/2)}
          \lVert g \rVert_{L^1((0,\infty)^k,\overset{k}{\underset{j=1}{\prod}} x_j^{2\alpha_j}dx)} > \gamma\Big\}\Big)\\
       & =  m_{\alpha_1, \dots,\alpha_k}\Big(\Big\{x \in (0,\infty)^k : \sum_{j=1}^k x_j
         < \left( \frac{1}{\gamma}\lVert g \rVert_{L^1((0,\infty)^k,\overset{k}{\underset{j=1}{\prod}} x_j^{2\alpha_j}dx)}  \right)
         ^{1/(2 \overset{k}{\underset{j=1}{\sum}} (\alpha_j + 1/2))}\Big\}\Big)\\
       & \leq  m_{\alpha_1, \dots,\alpha_k}(Q),
    \end{align*}
    where $Q=\overset{k}{\underset{j=1}{\prod}} [0,( \frac{1}{\gamma} \lVert g \rVert_{L^1((0,\infty)^k,\overset{k}{\underset{j=1}{\prod}}
    x_j^{2\alpha_j}dx)} )^{1/(2 \overset{k}{\underset{j=1}{\sum}} (\alpha_j + 1/2))}]$.\\

    Since,
    $$m_{\alpha_1, \dots,\alpha_k}(Q) \leq \frac{C}{\gamma} \lVert g \rVert_{L^1((0,\infty)^k,\overset{k}{\underset{j=1}{\prod}} x_j^{2\alpha_j}dx)}, $$
    it concludes that
    $$m_{\alpha_1, \dots, \alpha_k}(\{x \in (0,\infty)^k : \lvert L_{\alpha_1, \dots, \alpha_k}(g)(x)  \rvert > \gamma\})
      \leq \frac{C}{\gamma} \lVert g \rVert_{L^1((0,\infty)^k,\overset{k}{\underset{j=1}{\prod}} x_j^{2\alpha_j}dx)}.$$

    Let $\alpha_j \in \mathbb{R}$, $j=1, \dots,k$, with $k \in \mathbb{N}$, we define the local Hardy type operator $H_{loc}^{\alpha_1, \dots, \alpha_k}$ by
    $$H_{loc}^{\alpha_1, \dots, \alpha_k}(g)(x)= \frac{1}{\overset{k}{\underset{j=1}{\prod}}x_j^{2 \alpha_j +1}}
      \int_{x_1/2}^{2x_1} \cdots \int_{x_k/2}^{2x_k} g(y) \prod_{j=1}^k y_j^{2\alpha_j} dy,
      \quad x=(x_1, \dots, x_k) \in (0,\infty)^k.$$
    It is not hard to see that $H_{loc}^{\alpha_1, \dots,\alpha_k}$ is
    bounded from $L^1((0,\infty)^k,\overset{k}{\underset{j=1}{\prod}} x_j^{2\alpha_j}dx)$
    into itself.\\

    Also we consider the local maximal operator $\mathbb{W}_{*,loc}^{\alpha_1, \dots,\alpha_k}$ defined by
    $$\mathbb{W}_{*,loc}^{\alpha_1, \dots,\alpha_k}(g)(x)
       = \sup_{t>0} \left \lvert \int_{x_1/2}^{2x_1} \cdots \int_{x_k/2}^{2x_k}
         \prod_{j=1}^k (x_jy_j)^{-\alpha_j} \mathbb{W}_t(x_j,y_j)g(y) \prod_{j=1}^k  y_j^{2\alpha_j} dy  \right \rvert,
         \quad x \in (0,\infty)^k.$$
    $\mathbb{W}_{*,loc}^{\alpha_1, \dots,\alpha_k}$ is
    bounded from $L^1((0,\infty)^k,\overset{k}{\underset{j=1}{\prod}} x_j^{2\alpha_j}dx)$ into
    $L^{1,\infty}((0,\infty)^k,\overset{k}{\underset{j=1}{\prod}} x_j^{2\alpha_j}dx)$. Indeed, for every $j=(j_1, \dots, j_k) \in \mathbb{Z}^k$, we
    denote by $Q_j$ the dyadic cube
    $Q_j=\overset{k}{\underset{i=1}{\prod}}[2^{j_i},2^{j_i+1}]$.
    Also, for every $x \in (0,\infty)^k$, $D(x)$ represents the
    following set
    $$D(x)=\{y=(y_1, \dots, y_k) \in (0,\infty)^k : x_j/2 < y_j < 2x_j, \ j=1, \dots, k\}.$$
    Note that, if $x \in Q_j$ and $y \in D(x)$, then $2^{j_i-1} \leq y_i \leq 2^{j_i+2}$, $i=1, \dots,k$.\\

    Hence, it has
    \begin{align*}
      \mathbb{W}_{*,loc}^{\alpha_1, \dots,\alpha_k}(g)(x)
        & \leq \sup_{t>0} \int_{2^{j_1-1}}^{2^{j_1+2}} \cdots \int_{2^{j_k-1}}^{2^{j_k+2}}
          \prod_{i=1}^k (x_iy_i)^{-\alpha_i} \mathbb{W}_t(x_i,y_i) \lvert g(y) \rvert \prod_{i=1}^k  y_i^{2\alpha_i} dy\\
        & \leq C \sup_{t>0} \int_{2^{j_1-1}}^{2^{j_1+2}} \cdots \int_{2^{j_k-1}}^{2^{j_k+2}}
          \prod_{i=1}^k \mathbb{W}_t(x_i,y_i) \lvert g(y) \rvert dy, \quad x \in Q_j,
    \end{align*}
    where $C>0$ does not depend on $j \in \mathbb{Z}^k$.\\

    If we denote, for every $j=(j_1, \dots, j_k) \in \mathbb{Z}^k$, by
    $D_j$ the set
    $$D_j=\{y=(y_1, \dots, y_k) \in (0,\infty)^k : 2^{j_i-1} < y_j < 2^{j_i+2}, \quad i=1, \dots, k\},$$
    since, as it is well known, the maximal operator
    $$\mathbb{W}_*(f)(x)= \sup_{t>0} \left \lvert \int_{\mathbb{R}^k} \prod_{i=1}^k \mathbb{W}_t(x_i,y_i) f(y) dy \right \rvert, \quad x \in \mathbb{R}^k,$$
    is bounded from $L^1(\mathbb{R}^k,dx)$ into $L^{1,\infty}(\mathbb{R}^k,dx)$, it follows, for every $\gamma >0$, that
    \begin{align*}
      m_{\alpha_1, \dots, \alpha_k}&  (\{ x \in (0,\infty)^k : \mathbb{W}_{*,loc}^{\alpha_1, \dots,\alpha_k}(g)(x) > \gamma\})\\
       & =    \sum_{j \in \mathbb{Z}^k} m_{\alpha_1, \dots, \alpha_k}
              (\{ x \in Q_j : \mathbb{W}_{*,loc}^{\alpha_1, \dots,\alpha_k}(g)(x) > \gamma\})\\
       & \leq C \sum_{j \in \mathbb{Z}^k} 2^{\overset{k}{\underset{i=1}{\sum}}2j_i\alpha_i}m_0^{(k)}
              (\{ x \in Q_j : \mathbb{W}_{*,loc}^{\alpha_1, \dots,\alpha_k}(g)(x) > \gamma\})\\
       & \leq C \sum_{j \in \mathbb{Z}^k} 2^{\overset{k}{\underset{i=1}{\sum}}2j_i\alpha_i}m_0^{(k)}
              (\{ x \in \mathbb{R}^k : \mathbb{W}_*(\chi_{D_j}\lvert g \rvert)(x) > \gamma M\})\\
       & \leq \frac{C}{\gamma} \sum_{j \in \mathbb{Z}^k} 2^{\overset{k}{\underset{i=1}{\sum}}2j_i\alpha_i} \int_{D_j} \lvert g(y) \rvert dy\\
       & \leq \frac{C}{\gamma} \sum_{j \in \mathbb{Z}^k}
              \int_{D_j} \lvert g(y) \rvert \prod_{i=1}^k y_i^{2\alpha_i} dy\\
       & \leq \frac{C}{\gamma} \lVert g(y) \rVert_{L^1((0,\infty)^k,\overset{k}{\underset{i=1}{\prod}} x_i^{2\alpha_i}dx)},
         \quad g \in L^1((0,\infty)^k,\overset{k}{\underset{i=1}{\prod}} x_i^{2\alpha_i}dx).
    \end{align*}
    Here $m_0^{(k)}$ denotes the usual Lebesgue measure on $\mathbb{R}^k$ and $M$ is a suitable positive constant that does not depend on $j \in \mathbb{Z}^k$.\\

    We now define, for every $1 \leq l \leq k$, $l,k \in \mathbb{N}$, and $\alpha_j > -1/2$, $j=1, \dots, k$,
    the operator
    \begin{align*}
      \mathcal{H}_{l,k}^{\alpha_1, \dots, \alpha_k}(g)(x)
         = & \int_0^{x_1/2} \cdots \int_0^{x_l/2}\int_{x_{l+1}/2}^{2x_{l+1}} \cdots \int_{x_k/2}^{2x_k}  \prod_{j=l+1}^k(x_jy_j)^{-\alpha_j}\\
           & \times\frac{g(y)}{\left( \overset{k}{\underset{j=1}{\sum}} (x_j-y_j)^2 \right)^\varepsilon} \prod_{j=1}^k y_j^{2\alpha_j} dy,
             \quad x=(x_1, \dots, x_k) \in (0,\infty)^k,
    \end{align*}
    where $\varepsilon = \overset{l}{\underset{j=1}{\sum}}(\alpha_j+1/2)+(k-l)/2$.
    By proceeding as in \cite[Case 3]{NS1}, we can prove that $\mathcal{H}_{l,k}^{\alpha_1, \dots,\alpha_k}$
    is bounded from $L^1((0,\infty)^k,\overset{k}{\underset{j=1}{\prod}} x_j^{2\alpha_j}dx)$
    into $L^{1,\infty}((0,\infty)^k,\overset{k}{\underset{j=1}{\prod}} x_j^{2\alpha_j}dx)$.\\

    In order to study the operator $W_*^{\lambda_1, \dots, \lambda_n}$ we use the estimates \eqref{E1}
    and \eqref{E2} and divide, for every $j=1, \dots, n$, the region $(0,\infty)\times(0,\infty)$
    in three parts: one of them is close to the diagonal $\{(x_j,y_j) : x_j=y_j\}$ and the other two ones are
    far from the diagonal.\\

    Our result will be proved when we see that, for every $0 \leq l \leq m \leq n$, $l,m,n \in \mathbb{N}$ the operator
    $S_{l,m}^{\lambda_1, \dots,\lambda_n}$ defined by
    \begin{align*}
      S_{l,m}^{\lambda_1, \dots,\lambda_n}(f)(x)
        = &  \sup_{t>0} \int_0^{x_1/2} \cdots \int_0^{x_l/2}\int_{x_{l+1}/2}^{2x_{l+1}} \cdots \int_{x_m/2}^{2x_m}\int_{2x_{m+1}}^\infty \cdots \int_{2x_n}^\infty \\
          &  \prod_{j=1}^n W_t^{\lambda_j}(x_j,y_j) \lvert f(y) \rvert \prod_{j=1}^n y_j^{2\lambda_j}  dy,
             \quad x=(x_1, \dots, x_n) \in (0,\infty)^n,
    \end{align*}
    is bounded from $L^1((0,\infty)^n,\overset{n}{\underset{j=1}{\prod}} x_j^{2\alpha_j}dx)$
    into $L^{1,\infty}((0,\infty)^n,\overset{n}{\underset{j=1}{\prod}} x_j^{2\alpha_j}dx)$.\\

    According to \cite[Lemma 4]{BHNV}, we have that
    \begin{equation}\label{A0}
      0 \leq W_t^\lambda(x,y) \leq C y^{-2\lambda-1}, \quad  2x<y<\infty \text{ and } t>0,
    \end{equation}
    provided that $\lambda>-1/2$. Then, for every $0 \leq l \leq m \leq n$, $l,m,n \in \mathbb{N}$,
    \begin{align*}
      S_{l,m}^{\lambda_1, \dots,\lambda_n}(f)(x)
        \leq & C \sup_{t>0}  \int_0^{x_1/2} \cdots \int_0^{x_l/2}\int_{x_{l+1}/2}^{2x_{l+1}} \cdots \int_{x_m/2}^{2x_m}  \prod_{j=1}^m W_t^{\lambda_j}(x_j,y_j)\\
             & \times \left( \int_{2x_{m+1}}^\infty \cdots \int_{2x_n}^\infty \frac{\lvert f(y) \rvert}{y_{m+1} \cdots y_n} dy_n \cdots y_{m+1}\right)
               \prod_{j=1}^m y_j^{2\lambda_j} dy_m  \cdots  dy_1, \ x \in (0,\infty)^n.
    \end{align*}

    Hence, since the operator $H_\infty^{n-m}$ is bounded from $L^1((0,\infty)^{n-m},\overset{n-m}{\underset{j=1}{\prod}} x_j^{2\alpha_j}dx)$
    into itself, for every $0 \leq m <n$, $m \in \mathbb{N}$, by \cite[Proposition 1]{Din},
    we must show that, for every $0 \leq l \leq m \leq n$, $l,m \in \mathbb{N}$, and $\alpha_j>-1/2$, $j=1, \dots, n$, the
    operator $\mathcal{S}_{l,m}^{\alpha_1, \dots, \alpha_m}$ defined by
    \begin{align*}
      \mathcal{S}_{l,m}^{\alpha_1, \dots, \alpha_m}(g)(x)
        = & \sup_{t>0} \int_0^{\frac{x_1}{2}} \cdots \int_0^{\frac{x_l}{2}}\int_{\frac{x_{l+1}}{2}}^{2x_{l+1}} \cdots \int_{\frac{x_m}{2}}^{2x_m} \prod_{j=1}^m W_t^{\alpha_j}(x_j,y_j)
            \lvert g(y) \rvert \prod_{j=1}^m y_j^{2\alpha_j} dy,  x \in (0,\infty)^m,
    \end{align*}
    is bounded from $L^1((0,\infty)^m,\overset{m}{\underset{j=1}{\prod}} x_j^{2\alpha_j}dx)$
    into $L^{1,\infty}((0,\infty)^m,\overset{m}{\underset{j=1}{\prod}} x_j^{2\alpha_j}dx)$.\\

    Let $l,m \in \mathbb{N}$ such that $0 \leq l \leq m \leq n$ and let $\alpha_j>-1/2$, $j=1, \dots, m$.
    According to \eqref{E2} we get
    \begin{align}\label{A1}
      0 \leq W_t^\lambda(x,y) & \leq C \frac{(xy)^{-\lambda}}{\sqrt{t}}e^{-\frac{(x-y)^2}{4t}}
                                \leq C \frac{(xy)^{-\lambda}}{\sqrt{t}}e^{-\frac{x^2}{16t}}
                                \leq C \frac{(xy)^{-\lambda-1/2}}{\sqrt{t}}xe^{-\frac{x^2}{16t}}
                                \leq C \frac{e^{-\frac{x^2}{20t}}}{t^{\lambda + 1/2}},
    \end{align}
    provided that $0<y<x/2<\infty$, $t>0$, $xy/t \geq 1$ and $\lambda > -1/2$.\\

    Also by \eqref{E1} it obtains
    \begin{equation}\label{A2}
      0 \leq W_t^\lambda(x,y) \leq C \frac{1}{t^{\lambda + 1/2}}e^{-(x^2+y^2)/4t}
                              \leq C \frac{1}{t^{\lambda + 1/2}}e^{-x^2/4t},
    \end{equation}
    when $t,x,y \in (0,\infty)$, $xy/t \leq 1$ and $\lambda >-1/2$.\\

    By combining \eqref{A1} and \eqref{A2} it follows that, for every $\lambda > -1/2$,
    \begin{equation}\label{A3}
      0 \leq W_t^\lambda(x,y) \leq C \frac{1}{t^{\lambda + 1/2}}e^{-x^2/20t}, \quad 0<y<x/2\,\,\, and \,\,\,t>0.
    \end{equation}

    Then,
    \begin{align*}
      \mathcal{S}_{m,m}^{\alpha_1, \dots, \alpha_m}(g)(x)
        & \leq C \int_0^{x_1/2} \cdots \int_0^{x_m/2} \sup_{t>0} \frac{e^{-\overset{m}{\underset{j=1}{\sum}} x_j^2/20t}}{t^{\overset{m}{\underset{j=1}{\sum}}(\alpha_j + 1/2)}}
          \lvert g(y) \rvert \prod_{j=1}^m y_j^{2\alpha_j}dy\\
        & \leq C \frac{1}{\left( \overset{m}{\underset{j=1}{\sum}} x_j^2 \right)^{\overset{m}{\underset{j=1}{\sum}}(\alpha_j + 1/2)}} \int_0^{x_1/2} \cdots \int_0^{x_m/2}
          \lvert g(y) \rvert \prod_{j=1}^m y_j^{2\alpha_j}dy\\
        & \leq C L_{\alpha_1, \dots, \alpha_m}(\lvert g \rvert)(x), \quad x \in (0, \infty)^m.
    \end{align*}
    Hence, $\mathcal{S}_{m,m}^{\alpha_1, \dots, \alpha_m}$ is a bounded operator from $L^1((0,\infty)^m,\overset{m}{\underset{j=1}{\prod}} x_j^{2\alpha_j}dx)$
    into $L^{1,\infty}((0,\infty)^m,\overset{m}{\underset{j=1}{\prod}} x_j^{2\alpha_j}$ $dx)$.\\

    We now analyze the operator $\mathcal{S}_{0,m}^{\alpha_1, \dots, \alpha_m}$. Again, according to \eqref{E1} and \eqref{E2}, it gets,
    if $\lambda>-1/2$,
    \begin{equation}\label{A4}
      0 \leq W_t^\lambda(x,y) \leq C \frac{(xy)^{-\lambda}}{\sqrt{t}}e^{-(x-y)^2/4t}, \quad t,x,y \in (0,\infty) \text{ and } \frac{xy}{t}\geq 1,
    \end{equation}
    and
    \begin{equation}\label{A5}
      0 \leq W_t^\lambda(x,y) \leq C \frac{e^{-(x^2+y^2)/4t}}{t^{\lambda + 1/2}}
                              \leq C \frac{1}{x^{2\lambda + 1}}, \quad t,x,y \in (0,\infty) \text{ and } \frac{xy}{t}\leq 1.
    \end{equation}
    From \eqref{A4} and \eqref{A5} it infers that, if $\lambda >-1/2$,
    \begin{equation}\label{A6}
      0 \leq W_t^\lambda(x,y) \leq C \left( \frac{(xy)^{-\lambda}}{\sqrt{t}}e^{-(x-y)^2/4t} + \frac{1}{x^{2\lambda + 1}} \right), \quad t,x,y \in (0,\infty).
    \end{equation}
    Then \eqref{A6} implies that the operator $\mathcal{S}_{0,m}^{\alpha_1, \dots, \alpha_m}$ is bounded from
    $L^1((0,\infty)^m,\overset{m}{\underset{j=1}{\prod}} x_j^{2\alpha_j}dx)$ into
    $L^{1,\infty}((0,\infty)^m,\overset{m}{\underset{j=1}{\prod}} x_j^{2\alpha_j}dx)$
    when this property is established for the operator $\mathcal{D}_{l,m}^{\alpha_1, \dots, \alpha_m}$ defined, for every $0 \leq l \leq m$, $l \in \mathbb{N}$, by
    \begin{align*}
      \mathcal{D}_{l,m}^{\alpha_1, \dots, \alpha_m}(g)(x)
          & = \sup_{t>0}  \int_{x_1/2}^{2x_1} \cdots \int_{x_l/2}^{2x_l} \prod_{j=1}^l \frac{(x_jy_j)^{-\alpha_j}}{\sqrt{t}}e^{-(x_j-y_j)^2/4t}\\
                      & \times \left( \frac{1}{\overset{m}{\underset{j=l+1}{\prod}} x_j^{2\alpha_j+1}}\int_{x_{l+1}/2}^{2x_{l+1}} \cdots \int_{x_m/2}^{2x_m}
                        \lvert g(y) \rvert \prod_{j=l+1}^m y_j^{2\alpha_j} dy_j\right) \prod_{j=1}^l y_j^{2\alpha_j} dy_j , \ x \in (0, \infty)^m.
    \end{align*}
    Since the operator $\mathbb{W}_{*,loc}^{\alpha_1, \dots, \alpha_l}$, $0< l \leq m$, is bounded from
    $L^1((0,\infty)^l,\overset{l}{\underset{j=1}{\prod}} x_j^{2\alpha_j}dx)$ into $L^{1,\infty}($ $(0,\infty)^l,\overset{l}{\underset{j=1}{\prod}} x_j^{2\alpha_j}dx)$
    and the operator $H_{loc}^{\alpha_{l+1}, \dots, \alpha_m}$,
    $0<l \leq m$, is bounded from $L^1((0,\infty)^{m-l},$ $\overset{m}{\underset{j=l+1}{\prod}} x_j^{2\alpha_j}dx)$
    into itself, from \cite[Proposition 1]{Din}, it follows that $\mathcal{D}_{l,m}^{\alpha_1, \dots, \alpha_m}$
    is bounded from $L^1((0,\infty)^m,\overset{m}{\underset{j=1}{\prod}} x_j^{2\alpha_j}dx)$
    into $L^{1,\infty}((0,\infty)^m,\overset{m}{\underset{j=1}{\prod}} x_j^{2\alpha_j}dx)$, for every $0 \leq l \leq m$, $l \in \mathbb{N}$.\\

    Assume that $0<l<m$. Since the operator $L_{\alpha_1, \dots, \alpha_k}$, $k \in \mathbb{N}$, es bounded from
    $L^1((0,\infty)^k,$ $\overset{k}{\underset{j=1}{\prod}} x_j^{2\alpha_j}dx)$ into $L^{1,\infty}((0,\infty)^k,\overset{k}{\underset{j=1}{\prod}}  x_j^{2\alpha_j}dx)$
    and the operator $H_{loc}^{\alpha_1, \dots, \alpha_k}$,
    $k \in \mathbb{N}$, is bounded from $L^1((0,\infty)^k,\overset{k}{\underset{j=1}{\prod}} x_j^{2\alpha_j}dx)$ into
    itself, by taking into account \cite[Proposition 1]{Din}, \eqref{A3} and \eqref{A6}, in order to prove that the operator
    $\mathcal{S}_{l,m}^{\alpha_1, \dots, \alpha_m}$ is bounded from $L^1((0,\infty)^m,\overset{m}{\underset{j=1}{\prod}} x_j^{2\alpha_j}dx)$
    into $L^{1,\infty}((0,\infty)^m,\overset{m}{\underset{j=1}{\prod}} x_j^{2\alpha_j}dx)$, it is sufficient to see this property for the operator
    \begin{align*}
      T_{l,m}^{\alpha_1, \dots, \alpha_m}(g)(x)
         = \sup_{t>0} & \int_0^{x_1/2} \cdots \int_{0}^{x_l/2}\int_{x_{l+1}/2}^{2x_{l+1}} \cdots \int_{x_m/2}^{2x_m}
                        \prod_{j=1}^l \frac{1}{t^{\alpha_j +1/2}}e^{-x_j^2/20t}\\
                      & \times \prod_{j=l+1}^m \frac{(x_jy_j)^{-\alpha_j}}{\sqrt{t}}e^{-(x_j-y_j)^2/4t} \lvert g(y) \rvert \prod_{j=1}^m y_j^{2\alpha_j} dy,
                        \ x \in (0, \infty)^m.
    \end{align*}
    We have that
    \begin{align*}
      T_{l,m}^{\alpha_1, \dots, \alpha_m}(g)(x)
         \leq &  \sup_{t>0} \int_0^{x_1/2} \cdots \int_{0}^{x_l/2}\int_{x_{l+1}/2}^{2x_{l+1}} \cdots \int_{x_m/2}^{2x_m}
                        \prod_{j=l+1}^m (x_jy_j)^{-\alpha_j} \\
         &  \times \frac{e^{-(\overset{l}{\underset{j=1}{\sum}}x_j^2 + \overset{m}{\underset{j=l+1}{\sum}}(x_j-y_j)^2)/20t}}
           {t^{\overset{l}{\underset{j=1}{\sum}}(\alpha_j +1/2)+(m-l)/2}}\lvert g(y) \rvert \prod_{j=1}^m y_j^{2\alpha_j}  dy\\
         \leq &  C \int_0^{x_1/2} \cdots \int_{0}^{x_l/2}\int_{x_{l+1}/2}^{2x_{l+1}} \cdots \int_{x_m/2}^{2x_m}
           \prod_{j=l+1}^m (x_jy_j)^{-\alpha_j} \\
         &  \times \frac{1}{\left( \overset{l}{\underset{j=1}{\sum}}x_j^2 + \overset{m}{\underset{j=l+1}{\sum}}(x_j-y_j)^2 \right)
            ^{\overset{l}{\underset{j=1}{\sum}}(\alpha_j +1/2)+(m-l)/2}}\lvert g(y)\rvert \prod_{j=1}^m y_j^{2\alpha_j}dy\\
         \leq &  C \mathcal{H}_{l,m}^{\alpha_1, \dots,\alpha_m}(\lvert g \rvert)(x), \quad x=(x_1, \dots, x_m) \in (0,\infty)^m.
    \end{align*}
    Then, $T_{l,m}^{\alpha_1, \dots, \alpha_m}$ is bounded from $L^1((0,\infty)^m,\overset{m}{\underset{j=1}{\prod}} x_j^{2\alpha_j}dx)$
    into $L^{1,\infty}((0,\infty)^m,\overset{m}{\underset{j=1}{\prod}} x_j^{2\alpha_j}dx)$, because $\mathcal{H}_{l,m}^{\alpha_1, \dots,\alpha_m}$ has this property.\\

    Thus the proof of Theorem~\ref{maximal} is finished. \qed

  \section{Proof of Theorem~\ref{gfunction}}

    We are going to see that the Littlewood-Paley g-function associated with the heat semigroup $\{W_t^{\lambda_1, \dots, \lambda_n}\}_{t>0}$ and given by
    $$g^{\lambda_1, \dots, \lambda_n}(f)(x)
      =\left \{ \int_0^\infty \left\lvert t \frac{\partial}{\partial t} W_t^{\lambda_1, \dots, \lambda_n}(f)(x)  \right \rvert^2 \frac{dt}{t} \right \}^{1/2},
       \quad x \in (0, \infty)^n,$$
    defines a bounded operator from $L^1((0,\infty)^n,\overset{n}{\underset{j=1}{\prod}} x_j^{2\lambda_j}dx)$
    into $L^{1,\infty}((0,\infty)^n,\overset{n}{\underset{j=1}{\prod}} x_j^{2\lambda_j}dx)$, when $\lambda_j > -1/2$, $j=1, \dots,n$.\\

    For each $x \in (0, \infty)^n$, we consider the local operator $g_{loc}^{\lambda_1, \dots, \lambda_n}$ defined by
    \begin{align*}
      g_{loc}^{\lambda_1, \dots, \lambda_n}(f)(x)
         & = \left \{ \int_0^\infty t   \left \lvert  \frac{\partial}{\partial t} \int_{x_1/2}^{2x_1} \cdots \int_{x_n/2}^{2x_n}
            \prod_{j=1}^n (x_jy_j)^{-\lambda_j}\frac{e^{-(x_j-y_j)^2/4t}}{\sqrt{t}} f(y) \prod_{j=1}^n y_j^{2\lambda_j} dy \right\rvert^2 dt \right \}^{1/2}.
    \end{align*}

    Minkowski's inequality leads us to
    \begin{align*}
       \lvert  g^{\lambda_1, \dots, \lambda_n}&(f)(x) - g_{loc}^{\lambda_1, \dots, \lambda_n}(f)(x)  \rvert \\
      \leq &   \left \{ \int_0^\infty t \left \lvert  \frac{\partial}{\partial t} \int_{(0,\infty)^n}
        \left( \prod_{j=1}^n W_t^{\lambda_j}(x_j,y_j) -  \prod_{j=1}^n \chi_{\{x_j/2 < y_j < 2x_j\}}(y_j)
        (x_jy_j)^{-\lambda_j}  \right. \right. \right.
    \end{align*}
    \begin{align*}
      &  \left. \left. \left. \times \frac{e^{-(x_j-y_j)^2/4t}}{2\sqrt{\pi t}}\right) f(y) \prod_{j=1}^n y_j^{2\lambda_j}
        dy \right\rvert^2 dt \right \}^{1/2}, \ x=(x_1, \dots, x_n) \in (0, \infty)^n.
    \end{align*}

    Our proof is divided in two steps:

    \begin{Cl}\label{Claim1}
      The operator $g_{loc}^{\lambda_1, \dots, \lambda_n}$
      is bounded from $L^1((0,\infty)^n,\overset{n}{\underset{j=1}{\prod}} x_j^{2\lambda_j}dx)$
      into $L^{1,\infty}((0,\infty)^n,$ $\overset{n}{\underset{j=1}{\prod}}x_j^{2\lambda_j}dx)$.
    \end{Cl}

    \begin{Cl}\label{Claim2}
      The operator $G^{\lambda_1, \dots, \lambda_n}$ defined by
      \begin{align*}
        G^{\lambda_1, \dots, \lambda_n}(f)(x)
          = &  \left \{ \int_0^\infty t \left \lvert  \frac{\partial}{\partial t} \int_{(0,\infty)^n}
              \left( \prod_{j=1}^n W_t^{\lambda_j}(x_j,y_j) -  \prod_{j=1}^n \chi_{\{x_j/2 < y_j < 2x_j\}}(y_j) (x_jy_j)^{-\lambda_j} \right. \right. \right.\\
          & \left. \left. \left. \times  \frac{e^{-(x_j-y_j)^2/4t}}{2 \sqrt{t \pi}}\right) f(y)
            \prod_{j=1}^n y_j^{2\lambda_j} dy \right\rvert^2 dt \right \}^{1/2},  \ x=(x_1, \dots, x_n) \in (0, \infty)^n.
      \end{align*}
      is bounded from $L^1((0,\infty)^n,\overset{n}{\underset{j=1}{\prod}} x_j^{2\lambda_j}dx)$
      into $L^{1,\infty}((0,\infty)^n,\overset{n}{\underset{j=1}{\prod}}x_j^{2\lambda_j}dx)$.
    \end{Cl}

    We prove firstly Claim~\ref{Claim1}. It is sufficient to show that the operator
    \begin{align*}
      g_{1,loc}^{\lambda_1, \dots, \lambda_n}(f)(x)
        = &  \left\{ \int_0^\infty  t \left \lvert \int_{x_1/2}^{2x_1} \cdots \int_{x_n/2}^{2x_n}
            (x_1y_1)^{-\lambda_1} \frac{\partial}{\partial t}\left(\frac{e^{-(x_1-y_1)^2/4t}}{\sqrt{t}}\right) \prod_{j=2}^n (x_jy_j)^{-\lambda_j} \right. \right.\\
        & \left. \left. \times \frac{e^{-(x_j-y_j)^2/4t}}{\sqrt{t}} f(y)
          \prod_{j=1}^n y_j^{2\lambda_j} dy \right\rvert^2 dt \right \}^{1/2}, \ x=(x_1, \dots, x_n) \in (0,\infty)^n,
    \end{align*}
    is bounded from $L^1((0,\infty)^n,\overset{n}{\underset{j=1}{\prod}} x_j^{2\lambda_j}dx)$
    into $L^{1,\infty}((0,\infty)^n,\overset{n}{\underset{j=1}{\prod}}x_j^{2\lambda_j}dx)$.
    When the derivative acts on any other factor the operator that
    appears can be analyzed in a similar way. \\

    For every $m=(m_1, \dots, m_n) \in \mathbb{Z}^n$ the operator $Y_m^{\lambda_1, \dots,\lambda_n}$ is defined by
    \begin{align*}
       Y_m^{\lambda_1, \dots,\lambda_n}(f)(x)
        = & \chi_{\overset{n}{\underset{j=1}{\prod}}[2^{m_j},2^{m_j+1})}(x)
        \left\{ \int_0^\infty  t \left \lvert \int_{(0,\infty)^n}\chi_{\overset{n}{\underset{j=1}{\prod}}[2^{m_j-1},2^{m_j+2})}(y)
        (x_1y_1)^{-\lambda_1}  \right. \right.\\
      & \left. \left.  \times \frac{\partial}{\partial t}\left(\frac{e^{-(x_1-y_1)^2/4t}}{\sqrt{t}}\right)
        \prod_{j=2}^n (x_jy_j)^{-\lambda_j} \frac{e^{-(x_j-y_j)^2/4t}}{\sqrt{t}} f(y) \prod_{j=1}^n y_j^{2\lambda_j} dy \right\rvert^2 dt \right \}^{1/2}\\
       = &  \chi_{\overset{n}{\underset{j=1}{\prod}}[2^{m_j},2^{m_j+1})}(x) \prod_{j=1}^n x_j^{-\lambda_j}
        \left\{ \int_0^\infty  t \left \lvert \int_{(0,\infty)^n} \frac{\partial}{\partial t}\left(\frac{e^{-(x_1-y_1)^2/4t}}{\sqrt{t}}\right) \right. \right.\\
      &  \left. \left.  \times \prod_{j=2}^n \frac{e^{-(x_j-y_j)^2/4t}}{\sqrt{t}}  f_m(y) dy \right\rvert^2 dt \right \}^{1/2},
      \ x=(x_1, \dots, x_n) \in (0,\infty)^n,
    \end{align*}
    where
    $\displaystyle f_m(y)=\chi_{\overset{n}{\underset{j=1}{\prod}}[2^{m_j-1},2^{m_j+2})}(y) f(y) \prod_{j=1}^n y_j^{\lambda_j}, \quad y=(y_1, \dots, y_n) \in (0,\infty)^n.$

    Let $\gamma>0$. We can write, for each $f \in L^1((0,\infty)^n,\overset{n}{\underset{j=1}{\prod}}x_j^{2\lambda_j}dx)$,
    \begin{align}\label{B1}
      m_{\lambda_1, \dots, \lambda_n}&(\{ x \in (0, \infty)^n : g_{1,loc}^{\lambda_1, \dots, \lambda_n}(f)(x) > \gamma\})\nonumber\\
       \leq &  m_{\lambda_1, \dots, \lambda_n}\left(\{ x \in (0, \infty)^n :
         \lvert g_{1,loc}^{\lambda_1, \dots, \lambda_n}(f)(x) - \sum_{m \in \mathbb{Z}^n} Y_m^{\lambda_1, \dots,\lambda_n}(f)(x) \rvert \nonumber \right. \\
         & + \left. \sum_{m \in \mathbb{Z}^n} Y_m^{\lambda_1, \dots,\lambda_n}(f)(x)  > \gamma\}\right)\nonumber\\
       = &  m_{\lambda_1, \dots, \lambda_n}(\{ x \in (0, \infty)^n :
         \lvert g_{1,loc}^{\lambda_1, \dots, \lambda_n}(f)(x) - \sum_{m \in \mathbb{Z}^n} Y_m^{\lambda_1, \dots,\lambda_n}(f)(x) \rvert  > \gamma/2\})\nonumber\\
       &  + m_{\lambda_1, \dots, \lambda_n}(\{ x \in (0, \infty)^n :
         \sum_{m \in \mathbb{Z}^n} Y_m^{\lambda_1, \dots,\lambda_n}(f)(x)  > \gamma/2\})\nonumber\\
      = &  m_{\lambda_1, \dots, \lambda_n}(\{ x \in (0, \infty)^n :
         \lvert g_{1,loc}^{\lambda_1, \dots, \lambda_n}(f)(x) - \sum_{m \in \mathbb{Z}^n} Y_m^{\lambda_1, \dots,\lambda_n}(f)(x) \rvert  > \gamma/2\})\\
       &  + \sum_{m \in \mathbb{Z}^n} m_{\lambda_1, \dots, \lambda_n}(\{ x \in \prod_{j=1}^n [2^{m_j},2^{m_j+1}) :
         Y_m^{\lambda_1, \dots,\lambda_n}(f)(x)  > \gamma/2\}). \nonumber
    \end{align}
    For every $m \in \mathbb{Z}^n$, it has
    \begin{align*}
      m_{\lambda_1, \dots, \lambda_n}&(\{ x \in \prod_{j=1}^n [2^{m_j},2^{m_j+1}) :  Y_m^{\lambda_1, \dots,\lambda_n}(f)(x)  > \gamma/2\})\\
       \leq &  C 2^{2\overset{n}{\underset{j=1}{\sum}}m_j \lambda_j}
          m_0^{(n)}(\{  x \in \prod_{j=1}^n [2^{m_j},2^{m_j+1}) : g_1(f_m)(x) > M \gamma 2^{\overset{n}{\underset{j=1}{\sum}}m_j \lambda_j} \}),
    \end{align*}
    where $C,M >0$ are suitable constants that do not depend on $m \in \mathbb{Z}^n$, and $g_1$ represents the classical Littlewood-Paley g-function defined by
    $$g_1(F)(x)
      = \left \{ \int_0^\infty t \left\lvert \int_{\mathbb{R}^n} \frac{\partial}{\partial t} \left( \frac{e^{-(x_1-y_1)^2/4t}}{2\sqrt{\pi t}} \right)
        \prod_{j=2}^n \frac{e^{-(x_j-y_j)^2/4t}}{2\sqrt{\pi t}} F(y) dy \right  \rvert^2 dt \right \}^{1/2},
        \ x \in \mathbb{R}^n.$$
    Since, as it is well known, $g_1$ is a bounded operator from $L^1(\mathbb{R}^n,dx)$ into $L^{1,\infty}(\mathbb{R}^n,dx)$, we deduce that,
    \begin{align*}
      m_{\lambda_1, \dots, \lambda_n}&(\{ x \in \prod_{j=1}^n [2^{m_j},2^{m_j+1}) :  Y_m^{\lambda_1, \dots,\lambda_n}(f)(x)  > \gamma/2\})\\
        \leq & \frac{C}{\gamma} 2^{\overset{n}{\underset{j=1}{\sum}}m_j\lambda_j} \lVert f_m \rVert_{L^1(\mathbb{R}^n,dx)}\\
        = & \frac{C}{\gamma} 2^{\overset{n}{\underset{j=1}{\sum}}m_j\lambda_j} \int_{\overset{n}{\underset{j=1}{\prod}} [2^{m_j-1},2^{m_j+2})}
           \lvert f(y) \rvert \prod_{j=1}^n y_j^{\lambda_j} dy\\
        \leq & \frac{C}{\gamma} \int_{\overset{n}{\underset{j=1}{\prod}} [2^{m_j-1},2^{m_j+2})}
           \lvert f(y) \rvert \prod_{j=1}^n y_j^{2\lambda_j} dy.
    \end{align*}
    Hence, we obtain
    \begin{align}\label{B2}
      \sum_{m \in \mathbb{Z}^n} & m_{\lambda_1, \dots, \lambda_n}(\{ x \in \prod_{j=1}^n [2^{m_j},2^{m_j+1})
                                   :  Y_m^{\lambda_1, \dots,\lambda_n}(f)(x)  > \frac{\gamma}{2}\})
                                \leq \frac{C}{\gamma} \lVert f \rVert_{L^1((0,\infty)^n,\overset{n}{\underset{j=1}{\prod}}x_j^{2\lambda_j}dx)}.
    \end{align}
    On the other hand, it has that
    \begin{align}\label{B3}
      m&_{\lambda_1, \dots, \lambda_n}(\{ x \in (0, \infty)^n :
         \lvert g_{1,loc}^{\lambda_1, \dots, \lambda_n}(f)(x) - \sum_{m \in \mathbb{Z}^n} Y_m^{\lambda_1, \dots,\lambda_n}(f)(x) \rvert  > \gamma/2\})\\
        & =  \sum_{m \in \mathbb{Z}^n}  m_{\lambda_1, \dots, \lambda_n}(\{ x \in \prod_{j=1}^n [2^{m_j},2^{m_j+1}) :
        \lvert g_{1,loc}^{\lambda_1, \dots, \lambda_n}(f)(x) - Y_m^{\lambda_1, \dots,\lambda_n}(f)(x) \rvert  > \gamma/2\}). \nonumber
    \end{align}
    Fix $m=(m_1, \dots, m_n) \in \mathbb{Z}^n$. We consider, for every $ x \in \overset{n}{\underset{j=1}{\prod}} [2^{m_j},2^{m_j+1})$, the set
    $$A(x)= \prod_{j=1}^n [2^{m_j-1},2^{m_j+2}) \backslash \prod_{j=1}^n [x_j/2,2x_j).$$
    It is clear that
    $$A(x)=\bigcup_{l=1}^n \left( B_l(x) \cup C_l(x) \right), \quad x \in \prod_{j=1}^n [2^{m_j},2^{m_j+1}),$$
    where, for every $x \in \overset{n}{\underset{j=1}{\prod}} [2^{m_j},2^{m_j+1})$ and $l=1, \dots, n$,
    \begin{align*}
      B_l(x)& = \{ y \in (0,\infty)^n : 2^{m_j-1} \leq y_j \leq 2^{m_j+2}, \ j=1, \dots, n, \ j\neq l; \ 2^{m_l-1} \leq y_l < x_l/2\}
    \end{align*}
    and
    \begin{align*}
      C_l(x)& = \{ y \in (0,\infty)^n : 2^{m_j-1} \leq y_j \leq 2^{m_j+2}, \ j=1, \dots, n, \ j\neq l; \ 2x_l \leq y_l <2^{m_l+2}\}.
    \end{align*}

    We get by using Minkowski's inequality, for every $ x \in \overset{n}{\underset{j=1}{\prod}}[2^{m_j},2^{m_j+1})$,
    \begin{align*}
       \lvert g&_{1,loc}^{\lambda_1, \dots, \lambda_n}(f)(x) - Y_m^{\lambda_1, \dots, \lambda_n}(f)(x) \rvert \\
      \leq &   \prod_{j=1}^n x_j^{-\lambda_j} \left\{ \int_0^\infty t \left \lvert \sum_{l=1}^n \left( \int_{B_l(x)} + \int_{C_l(x)} \right)
          \frac{\partial}{\partial t} \left( \frac{e^{-\frac{(x_1-y_1)^2}{4t}}}{\sqrt{t}} \right) \prod_{j=2}^n \frac{e^{-\frac{(x_j-y_j)^2}{4t}}}{\sqrt{t}} f(y) \prod_{j=1}^n y_j^{\lambda_j}  dy \right \rvert^2 dt \right\}^\frac{1}{2}\\
      \leq &   \sum_{l=1}^n \left(
        \prod_{j=1}^n x_j^{-\lambda_j} \left\{ \int_0^\infty t \left \lvert \int_{B_l(x)}
          \frac{\partial}{\partial t} \left( \frac{e^{-\frac{(x_1-y_1)^2}{4t}}}{\sqrt{t}} \right) \prod_{j=2}^n \frac{e^{-\frac{(x_j-y_j)^2}{4t}}}{\sqrt{t}} f(y) \prod_{j=1}^n y_j^{\lambda_j}  dy \right \rvert^2 dt \right\}^\frac{1}{2} \right. \\
      &  \left. +  \prod_{j=1}^n x_j^{-\lambda_j} \left\{ \int_0^\infty t \left \lvert \int_{C_l(x)}
          \frac{\partial}{\partial t} \left( \frac{e^{-\frac{(x_1-y_1)^2}{4t}}}{\sqrt{t}} \right) \prod_{j=2}^n \frac{e^{-\frac{(x_j-y_j)^2}{4t}}}{\sqrt{t}} f(y) \prod_{j=1}^n y_j^{\lambda_j}  dy \right \rvert^2 dt \right\}^\frac{1}{2} \right)\\
      \leq &   \sum_{l=1}^n \left( \prod_{j=1}^n x_j^{-\lambda_j} \int_{B_l(x)}|f(y)| \left\{ \int_0^\infty t
        \left| \frac{ \partial}{ \partial t}\left(\frac{e^{-\frac{(x_1-y_1)^2}{4t}}}{\sqrt{t}}\right) \prod_{j=2}^n \frac{e^{-\frac{(x_j-y_j)^2}{4t}}}{\sqrt{t}}\right|^2 dt \right\}^\frac{1}{2} \prod_{j=1}^n y_j^{\lambda_j} dy \right. \\
      &  \left. + \prod_{j=1}^n x_j^{-\lambda_j} \int_{C_l(x)}|f(y)| \left\{ \int_0^\infty t
        \left| \frac{ \partial}{ \partial t}\left(\frac{e^{-\frac{(x_1-y_1)^2}{4t}}}{\sqrt{t}}\right) \prod_{j=2}^n \frac{e^{-\frac{(x_j-y_j)^2}{4t}}}{\sqrt{t}}\right|^2 dt \right\}^\frac{1}{2} \prod_{j=1}^n y_j^{\lambda_j} dy \right).
    \end{align*}
    Since
    \begin{equation}\label{B3,5}
        \left|\frac{\partial}{\partial t}\left( \frac{e^{-(x_1-y_1)^2/4t}}{\sqrt{t}}\right)\right| \leq C \frac{e^{-(x_1-y_1)^2/8t}}{t^{3/2}}, \quad t,x_1,y_1 \in (0,\infty),
    \end{equation}
    it follows that, for every $x\in \overset{n}{\underset{j=1}{\prod}}\left[ 2^{m_j}, 2^{m_j+1}\right),$
    \begin{align}\label{B4}
      \lvert g_{1,loc}^{\lambda_1, \dots, \lambda_n}&(f)(x) - Y_m^{\lambda_1, \dots, \lambda_n}(f)(x) \rvert \\
     \leq &  C \sum_{l=1}^n \left(\prod_{j=1}^n x_j^{-\lambda_j} \int_{B_l(x)}|f(y)| \left\{ \int_0^\infty
       \frac{e^{-\overset{n}{\underset{j=1}{\sum}}(x_j-y_j)^2/8t}}{t^{n+1}}  dt \right\}^{1/2}\prod_{j=1}^n y_j^{\lambda_j} dy \right.\nonumber \\
     &  + \left. \prod_{j=1}^n x_j^{-\lambda_j} \int_{C_l(x)}|f(y)| \left\{ \int_0^\infty
       \frac{e^{-\overset{n}{\underset{j=1}{\sum}}(x_j-y_j)^2/8t}}{t^{n+1}}  dt \right\}^{1/2}\prod_{j=1}^n y_j^{\lambda_j} dy \right). \nonumber
    \end{align}
   By taking into account symmetries we only study  the first term, $l=1$, in the last sum. Moreover the analysis of the integrals on $B_1(x)$ and $C_1(x)$ can be made
   similarly. We estimates only the integral on $B_1(x).$ Let $x\in \overset{n}{\underset{j=1}{\prod}}\left[ 2^{m_j}, 2^{m_j+1}\right).$ Since
   $|x_1-y_1|= x_1-y_1 > x_1/2 > 2^{m_1-1},$ when $2^{m_1-1}\leq y_1 < x_1/2, $ we have,
   \begin{align*}
    \prod_{j=1}^n x_j^{-\lambda_j} \int_{B_1(x)}&|f(y)| \left\{ \int_0^\infty
       \frac{e^{-\overset{n}{\underset{j=1}{\sum}}(x_j-y_j)^2/4t}}{t^{n+1}}  dt \right\}^{1/2}\prod_{j=1}^n y_j^{\lambda_j} dy\\
    \leq &  C \prod_{j=1}^n x_j^{-\lambda_j} \int_{B_1(x)} \frac{|f(y)|}{(2^{2m_1-2} + \overset{n}{\underset{j=2}{\sum}} (x_j-y_j)^2 )^{n/2}} \overset{n}{\underset{j=1}{\prod}} y_j^{\lambda_j} dy\\
    \leq&  C  \int_{\overset{n}{\underset{j=1}{\prod}}\left[ 2^{m_j-1}, 2^{m_j+2}\right)} \frac{|f(y)|}{(2^{2m_1-2}+
      \overset{n}{\underset{j=2}{\sum}} (x_j-y_j)^2 )^{n/2}}dy\\
    = &  C \int_{\mathbb{R}^{n-1}}\frac{\tilde{f}_m(\bar{y})}{(2^{2m_1-2}+ \overset{n}{\underset{j=2}{\sum}} (x_j-y_j)^2 )^{n/2}}d\bar{y},
   \end{align*}
   where $\bar{y}=(y_2,\ldots,y_n) \in \mathbb{R}^{n-1}$ and
   $\displaystyle \tilde{f}_m(\bar{y})= \chi_{\overset{n}{\underset{j=2}{\prod}}\left[ 2^{m_j-1}, 2^{m_j+2}\right)}(\bar{y}) \int_{2^{m_1-1}}^{2^{m_1+2}} |f(y_1,\bar{y})|dy_1.$  Then, by using standard argument it gets
   \begin{align*}
    \prod_{j=1}^n& x_j^{-\lambda_j} \int_{B_1(x)}|f(y)|\left\{ \int_0^\infty
      \frac{e^{-\overset{n}{\underset{j=1}{\sum}}(x_j-y_j)^2/4t}}{t^{n+1}}  dt \right\}^{1/2}\prod_{j=1}^n y_j^{\lambda_j} dy\\
    \leq &  C \left( \sum_{k=0}^\infty \int_{2^{k+m_1}\leq |\bar{x}- \bar{y}|
      <2^{k+1+m_1}}\frac{\tilde{f}_m(\bar{y})}{(2^{2m_1-2}+|\bar{x}-\bar{y}|^2)^{\frac{n}{2}}}d\bar{y}
    + \int_{|\bar{x}-\bar{y}|<2^{m_1}}\frac{\tilde{f}_m(\bar{y})}{(2^{2m_1-2}+|\bar{x}-\bar{y}|^2)^{\frac{n}{2}}}d\bar{y}\right)\\
    \leq &  C \sum_{k=0}^\infty \frac{1}{2^{(m_1+k)n}} \int_{|\bar{x}-\bar{y}|<2^{k+m_1}} \tilde{f}_m(\bar{y}) d\bar{y}\\
   \leq &  \frac{C}{2^{m_1}}M_{n-1}(\tilde{f}_m)(\bar{x}),
   \end{align*}
   where $\bar{x}=(x_2, \ldots, x_n)$ and $M_{n-1}$ denotes the Hardy-Littlewood maximal function on $\mathbb{R}^{n-1}.$\\

   For every $\beta>0,$ the maximal function theorem leads to
   \begin{align*}
    m_{\lambda_1, \ldots, \lambda_n}&(\{x\in \overset{n}{\underset{j=1}{\prod}}\left[ 2^{m_j}, 2^{m_j+1}\right):
      \overset{n}{\underset{j=1}{\prod}} x_j^{-\lambda_j}\int_{B_1(x)}|f(y)|\left\{ \int_0^\infty
      \frac{e^{-\overset{n}{\underset{j=1}{\sum}}\frac{(x_j-y_j)^2}{4t}}}{t^{n+1}}  dt \right\}^\frac{1}{2}\prod_{j=1}^n y_j^{\lambda_j} dy>\beta\})\\
  \leq &  m_{\lambda_1, \ldots, \lambda_n}(\{x\in \overset{n}{\underset{j=1}{\prod}}\left[ 2^{m_j}, 2^{m_j+1}\right):
      M_{n-1}(\tilde{f}_m)(\bar{x})>\beta2^{m_1}M\})\\
   \leq & 2^{2\overset{n}{\underset{j=1}{\sum}}\lambda_j m_j } m_0^{(n)} (\{ x\in \overset{n}{\underset{j=1}{\prod}}\left[ 2^{m_j}, 2^{m_j+1}\right):
      M_{n-1}(\tilde{f}_m)(\bar{x})>\beta2^{m_1}M\})\\
   \leq & 2^{2\overset{n}{\underset{j=1}{\sum}}\lambda_j m_j } 2^{m_1}m_0^{(n-1)} (\{ \bar{x}\in \mathbb{R}^{n-1}:
      M_{n-1}(\tilde{f}_m)(\bar{x})>\beta2^{m_1}M\})\\
   \leq & 2^{2\overset{n}{\underset{j=1}{\sum}}\lambda_j m_j }\frac{C}{\beta}\|\tilde{f}_m\|_{L^1(\mathbb{R}^{n-1},d\bar{x})}\\
   \leq & 2^{2\overset{n}{\underset{j=1}{\sum}}\lambda_j m_j }\frac{C}{\beta}\int_{\overset{n}{\underset{j=1}{\prod}}
      \left[ 2^{m_j-1},2^{m_j+2}\right)} |f(y)|dy\\
   \leq & \frac{C}{\beta} \int_{\overset{n}{\underset{j=1}{\prod}}
      \left[ 2^{m_j-1},2^{m_j+2}\right)} |f(y)| \overset{n}{\underset{j=1}{\prod}} y_j^{2\lambda_j}dy.
   \end{align*}
   Here, $C$ and $M$ denote positive constants that do not depend on $m\in \mathbb{Z}^n,$ and $m_0^{(k)}$ represents the Lebesgue measure on $\mathbb{R}^k.$\\

   From \eqref{B3} and \eqref{B4} we deduce that
   \begin{align}\label{B5}
      m_{\lambda_1, \dots, \lambda_n}&(\{ x \in (0, \infty)^n :
         \lvert g_{1,loc}^{\lambda_1, \dots, \lambda_n}(f)(x) - \sum_{m \in \mathbb{Z}^n} Y_m^{\lambda_1, \dots,\lambda_n}(f)(x) \rvert  > \gamma/2\})\nonumber\\
       & \leq \frac{C}{\gamma}\sum_{m\in \mathbb{Z}^n} \int_{\overset{n}{\underset{j=1}{\prod}}
         \left[ 2^{m_j-1},2^{m_j+2}\right)} |f(y)| \overset{n}{\underset{j=1}{\prod}} y_j^{2\lambda_j}dy\nonumber\\
       & \leq \frac{C}{\gamma} \|f\|_{L^1((0,\infty)^n,\overset{n}{\underset{j=1}{\prod}} x_j^{2\lambda_j}dx )}.
   \end{align}
   Finally, putting together \eqref{B1}, \eqref{B2} and \eqref{B5} we conclude that
   $$m_{\lambda_1, \dots, \lambda_n}(\{ x \in (0, \infty)^n :\lvert g_{1,loc}^{\lambda_1, \dots, \lambda_n}(f)(x) \rvert >\gamma \}) \leq \frac{C}{\gamma}\|f\|_{L^1((0,\infty)^n,\overset{n}{\underset{j=1}{\prod}} x_j^{2\lambda_j}dx )}. $$
   Thus Claim \ref{Claim1} is established. \qed\\

   We now prove Claim \ref{Claim2}. Minkowski inequality allow us to write
   \begin{align*}
      G^{\lambda_1, \ldots, \lambda_n}(f)(x) \leq &  \sum_{i=1}^n \left\{ \int_0^\infty t \left| \int_{(0, \infty)^n}
         \left( \frac{\partial}{\partial t}W_t^{\lambda_i} (x_i,y_i) \prod_{j=1;j\neq i}^n W_t^{\lambda_j}(x_j,y_j)
         \right. \right. \right.\\
      & - \prod_{j=1}^n \chi_{\{x_j/2<y_j<2x_j\}}(y_j)(x_iy_i)^{-\lambda_i}\frac{\partial}{\partial t}
         \left( \frac{e^{-(x_i-y_i)^2/4t}}{2\sqrt{\pi t}}\right)  \\
       & \times \left.\left.\left.   \prod_{j=1;j\neq i}^n \frac{e^{-(x_j-y_j)^2/4t}}{2\sqrt{\pi t}}(x_jy_j)^{-\lambda_j} \right)
         f(y) \prod_{j=1}^n y_j^{2\lambda_j}dy\right|^2dt\right\}^{1/2}\\
      = &  \sum_{i=1}^n G_i^{\lambda_1, \ldots, \lambda_n}(f)(x), \quad \quad x=(x_1, \ldots, x_n) \in (0,\infty)^n.
   \end{align*}
   By taking into account symmetries to show Claim \ref{Claim2}, it is sufficient to prove that $G_1^{\lambda_1,\ldots, \lambda_n}$ defines \;a bounded \; operator \;from \;$L^1((0,\infty)^n, \overset{n}{\underset{j=1}{\prod}} x_j^{2\lambda_j} dx)$
   into $L^{1,\infty}((0,\infty)^n, \overset{n}{\underset{j=1}{\prod}} x_j^{2\lambda_j} dx).$ \\

   Again by using Minkoswski inequality we have that
   \begin{align}\label{B6}
    G_1^{\lambda_1, \ldots, \lambda_n}&(f)(x)  \leq \left\{ \int_0^\infty t \left| \int_{(0,\infty)^n} \left( \frac{\partial}{\partial t}
        W_t^{\lambda_1}(x_1,y_1) - \chi_{\{x_1/2 < y_1 < 2x_1 \}}(y_1) (x_1y_1)^{-\lambda_1} \right.\right. \right. \nonumber\\
     & \times \left. \left. \left. \frac{\partial}{\partial t}
       \left( \frac{e^{-(x_1-y_1)^2/4t}}{2\sqrt{\pi t}}\right)\right)\prod_{j=2}^n W_t^{\lambda_j}(x_j,y_j)f(y)\prod_{j=1}^n y_j^{2\lambda_j}dy \right|^2 dt
       \right\}^{1/2}\nonumber\\
     &  + \sum_{i=2}^n \left\{ \int_0^\infty t \left| \int_{(0,\infty)^n} \chi_{\{x_1/2 < y_1 < 2x_1\}} (y_1)(x_1y_1)^{-\lambda_1}\frac{\partial}{\partial t}\left(
       \frac{e^{-(x_1-y_1)^2/4t}}{2\sqrt{\pi t}}\right)  \right. \right.\nonumber\\
     & \times \prod_{j=2}^{i-1}\chi_{\{x_j/2 < y_j < 2x_j\}} (y_j)(x_jy_j)^{-\lambda_j} \frac{e^{-(x_j-y_j)^2/4t}}{2\sqrt{\pi t}}
       \prod_{j=i+1}^n W_t^{\lambda_j}(x_j,y_j) \nonumber\\
     & \times \left. \left.\left( W_t^{\lambda_i}(x_i,y_i)-\chi_{\{x_i/2 < y_i < 2x_i \}}(y_i)(x_iy_i)^{-\lambda_i}\frac{e^{-(x_i-y_i)^2/4t}}{2\sqrt{\pi t}}\right) f(y) \prod_{j=1}^n y_j^{2\lambda_j}dy\right|^2 dt \right\}^{1/2} \nonumber  \\
   = & \sum_{i=1}^n G_{1,i}^{\lambda_1, \ldots, \lambda_n}(f)(x), \quad \quad x=(x_1,\ldots, x_n) \in (0, \infty)^n.
   \end{align}
   We are going to study $G_{1,i}^{\lambda_1, \ldots, \lambda_n},$ $i=1,\ldots, n.$ Firstly we considerer $G_{1,1}^{\lambda_1, \ldots, \lambda_n}.$ By splitting the integral on $y_1 \in (0, \infty)$ we get
   \begin{align}\label{B7}
     G_{1,1}^{\lambda_1, \ldots, \lambda_n}&(f)(x)\leq \left\{ \int_0^\infty t\left|\int_0^{x_1/2}\int_{(0,\infty)^{n-1}}
              \frac{\partial}{\partial t} W_t^{\lambda_1}(x_1,y_1)\prod_{j=2}^n W_t^{\lambda_j} (x_j,y_j) f(y) \prod_{j=1}^n y_j^{2\lambda_j}dy \right|^2 dt \right\}^{\frac{1}{2}} \nonumber\\
            & + \left\{ \int_0^\infty t \left| \int_{2x_1}^\infty \int_{(0,\infty)^{n-1}}  \frac{\partial}{\partial t}
              W_t^{\lambda_1}(x_1,y_1)\prod_{j=2}^n W_t^{\lambda_j} (x_j,y_j) f(y) \prod_{j=1}^n y_j^{2\lambda_j}dy \right|^2 dt \right\}^{\frac{1}{2}} \nonumber \\
            & + \left\{ \int_0^\infty t \left| \int_{x_1/2}^{2x_1} \int_{(0,\infty)^{n-1}} \left( \frac{\partial}{\partial t}W_t^{\lambda_1}(x_1,y_1)
              -(x_1y_1)^{-\lambda_1}\frac{\partial}{\partial t}\left( \frac{e^{-(x_1-y_1)^2/4t}}{2\sqrt{\pi t}}\right)\right) \right. \right.\nonumber\\
            &  \times\left. \left. \prod_{j=2}^n W_t^{\lambda_j} (x_j,y_j) f(y) \prod_{j=1}^n y_j^{2\lambda_j}dy \right|^2 dt
              \right\}^{\frac{1}{2}} \nonumber\\
            =&  \sum_{i=1}^3 G_{1,1;i}^{\lambda_1, \ldots, \lambda_n}(f)(x), \quad x=(x_1,\ldots, x_n) \in (0, \infty)^n.
   \end{align}
   According to \cite[proof of Lemma~8]{BHNV} we have, for every $\lambda > -1/2,$ $t,x,y \in (0,\infty)$ and $y<x/2,$
   \begin{equation}\label{B8}
      \left| \frac{\partial}{\partial t}W_t^\lambda(x,y)\right|\leq C \frac{(xy)^{-\lambda}x^2}{t^{5/2}}e^{-x^2/16t}, \quad xy>t,
   \end{equation}
   and
   \begin{equation}\label{B9}
      \left| \frac{\partial}{\partial t}W_t^\lambda(x,y)\right|\leq C \frac{1}{t^{\lambda + 3/2}}e^{-(x^2+y^2)/8t}, \quad xy\leq t.
   \end{equation}
   From \eqref{B8} we deduce that, when $0<y<x/2<\infty$,
   \begin{equation}
      \left| \frac{\partial}{\partial t}W_t^\lambda(x,y)\right|\leq C \frac{(xy)^{-\lambda-1/2}x^3}{t^{5/2}}e^{-x^2/16t}\leq
      \frac{C}{t^{\lambda + 3/2}}e^{-x^2/20t}, \quad xy>t>0,
   \end{equation}
   together  with \eqref{B9}, implies that, for every $\lambda>-1/2,$ $t,x,y\in (0,\infty),$ and $y<x/2,$
   \begin{equation}\label{B10}
      \left| \frac{\partial}{\partial t}W_t^\lambda(x,y)\right|\leq \frac{C}{t^{\lambda + 3/2}}e^{-x^2/20t}.
   \end{equation}

   It will be proved that the operator defined by $G_{1,1;1}^{\lambda_1, \ldots, \lambda_n}(f)(x)$ is bounded from $L^1((0,\infty)^n, $ $\overset{n}{\underset{j=1}{\prod}} x_j^{2\lambda_j} dx)$ into $L^{1,\infty}((0,\infty)^n, \overset{n}{\underset{j=1}{\prod}} x_j^{2\lambda_j} dx)$ when we show that, for every $1\leq l \leq m \leq n,$ the operator $S_{l,m}^{\lambda_1, \ldots, \lambda_n}$ defined by
   \begin{align*}
      S_{l,m}^{\lambda_1, \ldots, \lambda_n}(f)(x)=& \left\{ \int_0^\infty t \left| \int_0^{x_1/2} \ldots
         \int_0^{x_l/2}\int_{x_{l+1}/2}^{2x_{l+1}}\ldots \int_{x_{m}/2}^{2x_{m}}\int_{2x_{m+1}}^\infty \ldots \int_{2x_{n}}^\infty \right. \right.\\
      & \left. \left. \left| \frac{\partial}{\partial t}W_t^{\lambda_1}(x_1,y_1)\right| \prod_{j=2}^n W_t^{\lambda_j} (x_j,y_j) |f(y)| \prod_{j=1}^n y_j^{2\lambda_j}dy \right|^2 dt \right\}^{1/2},\quad x\in(0,\infty)^n,
   \end{align*}
   is bounded from $L^1((0,\infty)^n, \overset{n}{\underset{j=1}{\prod}} x_j^{2\lambda_j} dx)$
   into $L^{1,\infty}((0,\infty)^n, \overset{n}{\underset{j=1}{\prod}} x_j^{2\lambda_j} dx).$\\

   By using \eqref{A0} since the operator $H^k_\infty$ is bounded from $L^1((0,\infty)^k, \overset{k}{\underset{j=1}{\prod}} x_j^{2\alpha_j} dx)$ into itself, for every $k\in \mathbb{N}$ and
   $\alpha_j>-1/2,$ $j=1,\ldots,k,$ in order to prove that the operator $S_{l,m}^{\lambda_1,\ldots, \lambda_n},$ $1\leq l\leq m\leq n,$
   is bounded from $L^1((0,\infty)^n, \overset{n}{\underset{j=1}{\prod}} x_j^{2\lambda_j} dx)$ into $L^{1,\infty}((0,\infty)^n, \overset{n}{\underset{j=1}{\prod}} x_j^{2\lambda_j} dx)$ it is sufficient to show that the operator
   \begin{align*}
      \mathcal{S}_{l,m}^{\alpha_1, \ldots, \alpha_m}(f)(x) = &\left\{ \int_0^\infty t \left| \int_0^{x_1/2} \ldots
         \int_0^{x_l/2}\int_{x_{l+1}/2}^{2x_{l+1}}\ldots \int_{x_{m}/2}^{2x_{m}}\left| \frac{\partial}{\partial t}W_t^{\alpha_1}(x_1,y_1)\right|\right. \right. \\
      & \times\left. \left.  \prod_{j=2}^m W_t^{\alpha_j} (x_j,y_j) |f(y)| \prod_{j=1}^m y_j^{2\alpha_j}dy \right|^2 dt \right\}^{1/2}, \ x= (x_1, \ldots, x_m)\in (0,\infty)^m,
   \end{align*}
   is bounded from $L^1((0,\infty)^m, \overset{m}{\underset{j=1}{\prod}} x_j^{2\alpha_j} dx)$
   into $L^{1,\infty}((0,\infty)^m, \overset{m}{\underset{j=1}{\prod}} x_j^{2\alpha_j} dx), $ for every $\alpha_j>-1/2,$ $j=1, \ldots m,$ and $1\leq l \leq m.$\\

   Let $\alpha_j > -1/2,$ $j=1,\ldots m.$ According to \eqref{A3} and \eqref{B10}, we can write
   \begin{align*}
      \mathcal{S}_{m,m}^{\alpha_1, \ldots, \alpha_m}(f)(x)& \leq C \left\{ \int_0^\infty t \left| \int_0^{x_1/2} \ldots
         \int_0^{x_m/2} \frac{e^{-(\overset{m}{\underset{j=1}{\sum}} x_j^2)/20t}}{t^{ \overset{m}{\underset{j=1}{\sum}}(\alpha_j+1/2)+1 }} |f(y)| \prod_{j=1}^m y_j^{2\alpha_j} dy \right|^2dt \right\}^{1/2}\\
      & \leq C\left\{ \int_0^\infty \frac{e^{-(\overset{m}{\underset{j=1}{\sum}} x_j^2)/10t}}{t^{ 2\overset{m}{\underset{j=1}{\sum}}(\alpha_j+1/2)+1 }} dt\right\}^{1/2} \int_0^{x_1/2} \ldots \int_0^{x_m/2}|f(y)|\prod_{j=1}^m y_j^{2\alpha_j} dy \\
      & \leq C L_{\alpha_1, \ldots, \alpha_m}(|f|)(x), \quad x=(x_1, \ldots, x_m) \in (0, \infty)^m.
   \end{align*}
   Then, $\mathcal{S}_{m,m}^{\alpha_1, \ldots, \alpha_m}$ is bounded from $L^1((0,\infty)^m, \overset{m}{\underset{j=1}{\prod}} x_j^{2\alpha_j} dx)$ into $L^{1,\infty}((0,\infty)^m, \overset{m}{\underset{j=1}{\prod}} x_j^{2\alpha_j} dx),$
   because $L_{\alpha_1, \ldots, \alpha_m}$ has this property.\\

   Assume that $1\leq l < m.$ By \eqref{A6}, since $H_{loc}^{\alpha_1, \ldots, \alpha_k},$ $1\leq k \leq m,$ is bounded from $L^1((0,\infty)^k,$ $ \overset{k}{\underset{j=1}{\prod}} x_j^{2\alpha_j} dx)$ into itself and $\mathcal{S}_{m,m}^{\alpha_1, \ldots, \alpha_m}$
   from $L^1((0,\infty)^k, \overset{k}{\underset{j=1}{\prod}} x_j^{2\alpha_j} dx)$ into $L^{1,\infty}((0,\infty)^k, \overset{k}{\underset{j=1}{\prod}} x_j^{2\alpha_j} dx),$ \cite[Proposition 1]{Din} implies that $\mathcal{S}_{l,m}^{\alpha_1, \ldots, \alpha_m}$
   is bounded from $L^1((0,\infty)^m, \overset{m}{\underset{j=1}{\prod}} x_j^{2\alpha_j} dx)$ into $L^{1,\infty}((0,\infty)^m,$ $\overset{m}{\underset{j=1}{\prod}} x_j^{2\alpha_j} dx),$ provided that, for every $k\in \mathbb{N}$, $m<k\le n$, the operator
   \begin{align*}
      T_{l,k}^{\alpha_1, \ldots, \alpha_k}(f)(x) =& \left\{ \int_0^\infty t \left| \int_0^{x_1/2} \ldots
         \int_0^{x_l/2}\int_{x_{l+1}/2}^{2x_{l+1}}\ldots \int_{x_{k}/2}^{2x_{k}} \frac{e^{-(\overset{l}{\underset{j=1}{\sum}}x_j^2 + \overset{k}{\underset{j=l+1}{\sum}} (x_j-y_j)^2) /20t}}{t^{\overset{l}{\underset{j=1}{\sum}} (\alpha_j + \frac{1}{2})+1 + \frac{k-l}{2}}}  \right. \right.\\
      & \times\left. \left.  \prod_{j=l+1}^k (x_jy_j)^{-\alpha_j} |f(y)| \prod_{j=1}^k y_j^{2\alpha_j}dy \right|^2 dt \right\}^{1/2}, \quad x=(x_1, \ldots, x_k) \in (0, \infty)^k,
   \end{align*}
   has this property. Let $k\in \mathbb{N}$, $m<k\le n$. Minkowski's inequality leads to
   \begin{align*}
      T_{l,k}^{\alpha_1, \ldots, \alpha_k}(f)(x) \leq & C \int_0^{x_1/2} \ldots
         \int_0^{x_l/2}\int_{x_{l+1}/2}^{2x_{l+1}}\ldots \int_{x_{k}/2}^{2x_{k}} \prod_{j=l+1}^{k}(x_jy_j)^{-\alpha_j}|f(y)|\\
      & \times \left\{\int_0^\infty \frac{e^{-(\overset{l}{\underset{j=1}{\sum}}x_j^2 + \overset{k}{\underset{j=l+1}{\sum}}
         (x_j-y_j)^2) /10t}}{t^{2\overset{l}{\underset{j=1}{\sum}} (\alpha_j + \frac{1}{2})+1 + {k-l}}} dt\right\}^{1/2} \prod_{j=1}^k y_j^{2 \alpha_j}dy\\
      \leq & C \mathcal{H}_{l,k}^{ \alpha_1, \ldots, \alpha_k}(|f|)(x), \quad x=(x_1, \ldots, x_k) \in (0, \infty)^k.
   \end{align*}
   Hence, $T_{l,k}^{\alpha_1, \ldots, \alpha_k}$ is bounded from $L^1((0,\infty)^k, \overset{k}{\underset{j=1}{\prod}} x_j^{2\alpha_j} dx)$ into
   $L^{1,\infty}((0,\infty)^k, \overset{k}{\underset{j=1}{\prod}} x_j^{2\alpha_j} dx), $ because $\mathcal{H}_{l,k}^{ \alpha_1, \ldots, \alpha_k}$ has this property.\\

   We have proved that $G_{1,1;1}^{\lambda_1, \ldots, \lambda_n}$ is bounded from $L^1((0,\infty)^n, \overset{n}{\underset{j=1}{\prod}} x_j^{2\lambda_j} dx)$ into
    $L^{1,\infty}((0,\infty)^n, $ $\overset{n}{\underset{j=1}{\prod}} x_j^{2\lambda_j} dx). $\\

  In order to see that $G_{1,1;2}^{\lambda_1, \ldots, \lambda_n}$ is bounded from $L^1((0,\infty)^n, \overset{n}{\underset{j=1}{\prod}} x_j^{2\lambda_j} dx)$
   into $L^{1,\infty}((0,\infty)^n,$ $ \overset{n}{\underset{j=1}{\prod}} x_j^{2\lambda_j} dx), $ we must prove that, for every $1<l\leq m \leq n,$ the operator
  \begin{align*}
      {S}_{l,m}^{\alpha_1, \ldots, \alpha_m}(f)(x)& = \left\{ \int_0^\infty t \left| \int_{2x_1}^\infty \ldots \int_{2x_l}^\infty
         \int_0^{x_{l+1}/2} \ldots\int_0^{x_m/2}\int_{x_{m+1}/2}^{2x_{m+1}}\ldots \int_{x_{n}/2}^{2x_{n}}\right. \right. \\
      & \left. \left. \left| \frac{\partial}{\partial t}W_t^{\lambda_1}(x_1,y_1)\right| \prod_{j=2}^n W_t^{\lambda_j} (x_j,y_j)
         |f(y)| \prod_{j=1}^n y_j^{2\lambda_j}dy \right|^2 dt \right\}^{1/2}, \quad x\in (0,\infty)^m,
  \end{align*}
  is bounded from $L^1((0,\infty)^n, \overset{n}{\underset{j=1}{\prod}} x_j^{2\lambda_j} dx)$ into
  $L^{1,\infty}((0,\infty)^n, \overset{n}{\underset{j=1}{\prod}} x_j^{2\lambda_j} dx). $\\

  Note firstly that if $l<m,$ we can proceed as in the previous cases (in the analysis of $G_{1,1;1}^{\lambda_1, \ldots, \lambda_n}$) to see that $S_{l,m}^{\lambda_1, \ldots, \lambda_n}$
  is bounded from $L^1((0,\infty)^n, \overset{n}{\underset{j=1}{\prod}} x_j^{2\lambda_j} dx)$ into $L^{1,\infty}((0,\infty)^n, \overset{n}{\underset{j=1}{\prod}} x_j^{2\lambda_j} dx). $\\

  Suppose now that $l=m.$ Since the operator $H_\infty^k$ is bounded from $L^1((0,\infty)^k, \overset{k}{\underset{j=1}{\prod}} x_j^{2\alpha_j} dx)$
  into itself, for every $\alpha_j > -1/2,$ $j=1, \ldots, k,$ in order to prove that $S_{m,m}^{\lambda_1,\ldots ,\lambda_n}$ is bounded from
  $L^1((0,\infty)^n, \overset{n}{\underset{j=1}{\prod}} x_j^{2\lambda_j} dx)$ into $L^{1,\infty}((0,\infty)^n, \overset{n}{\underset{j=1}{\prod}} x_j^{2\lambda_j} dx), $
  it is sufficient to show that the operator
  \begin{align*}
     S_k^{\alpha_1, \ldots, \alpha_k}(f)(x) = &\left\{ \int_0^\infty t \left| \int_{2x_1}^\infty \int_{\frac{x_2}{2}}^{2x_2} \ldots \int_{\frac{x_k}{2}}^{2x_k}
        \left| \frac{\partial}{\partial t}W_t^{\alpha_1}(x_1,y_1)\right|\prod_{j=2}^k W_t^{\alpha_j} (x_j,y_j)
         |f(y)| \prod_{j=1}^k y_j^{2\alpha_j}dy \right|^2 dt \right\}^{\frac{1}{2}}
  \end{align*}
  is bounded from $L^1((0,\infty)^k, \overset{k}{\underset{j=1}{\prod}} x_j^{2\alpha_j} dx)$ into $L^{1,\infty}((0,\infty)^k, \overset{k}{\underset{j=1}{\prod}} x_j^{2\alpha_j} dx), $
  for every $\alpha_j>-1/2,$ $j=1, \ldots, k,$ and for each $x=(x_1,\ldots,x_k) \in (0,\infty)^k.$\\

  If $\alpha_j>-{1/2}, $ $j=1,\ldots, k,$ since $H_{loc}^{\beta_1, \ldots, \beta_l}$ is bounded  from $L^1((0,\infty)^l, \overset{l}{\underset{j=1}{\prod}} x_j^{2\beta_j} dx)$
  into itself when $\beta_j>-1/2,$ $j=1, \ldots, l,$ and $l \in \mathbb{N},$ by taking into account that $W_t^\lambda(x,y)=W_t^\lambda(y,x)$, $t,x,y\in (0,\infty)$,
  and \eqref{A6} and \eqref{B10}, the desired property for $S_k^{\alpha_1, \ldots, \alpha_k}$ follows when we see that the operator
  \begin{align*}
     T_k^{\beta_1, \ldots, \beta_l}(f)(x)=&\left\{ \int_0^\infty t \left| \int_{2x_1}^\infty \int_{x_2/2}^{2x_2} \ldots \int_{x_l/2}^{2x_l}
        \frac{e^{- \frac{y_1^2 + \overset{l}{\underset{j=2}{\sum}} (x_j-y_j)^2}{20t}}}{t^{\beta_1+ \frac{3}{2}+\frac{l-1}{2}}}\prod_{j=2}^l (x_jy_j)^{-\beta_j}
         f(y) \prod_{j=1}^l y_j^{2\beta_j}dy \right|^2 dt \right\}^{\frac{1}{2}}
  \end{align*}
  is bounded from $L^1((0,\infty)^l, \overset{l}{\underset{j=1}{\prod}} x_j^{2\beta_j} dx)$ into $L^{1, \infty}((0,\infty)^l, \overset{l}{\underset{j=1}{\prod}} x_j^{2\beta_j} dx),$
  for every $\beta_j >-1/2, $ $j=1, \ldots, l,$ $l \in \mathbb{N},$ and $x=(x_1,\ldots,x_l) \in (0,\infty)^l .$\\

  Let $\beta_j>-1/2,$ $j=1,\ldots, l$ and $l\in \mathbb{N}.$ Minkowski's inequality leads to
  \begin{align*}
     |T_k^{\beta_1, \ldots, \beta_l}(f)(x)|\leq & \int_{2x_1}^\infty \int_{x_2/2}^{2x_2} \ldots \int_{x_l/2}^{2x_l} \left\{ \int_0^\infty
        \frac{e^{- \frac{y_1^2 + \overset{l}{\underset{j=2}{\sum}} (x_j-y_j)^2}{10t}}}{t^{2\beta_1+l+1}} dt \right\}^\frac{1}{2} \prod_{j=2}^l (x_jy_j)^{-\beta_j} |f(y)| \prod_{j=1}^l y_j^{2\beta_j}dy \\
     \leq & C  \int_{2x_1}^\infty \int_{x_2/2}^{2x_2} \ldots \int_{x_l/2}^{2x_l} \frac{\overset{l}{\underset{j=2}{\prod}} (x_jy_j)^{-\beta_j} |f(y)|
            }{(y_1^2 + \overset{l}{\underset{j=2}{\sum}} (x_j-y_j)^2)^{\beta_1+ \frac{l}{2}}}  \prod_{j=1}^l y_j^{2\beta_j}dy,
            \quad x \in (0, \infty)^l.
  \end{align*}
  Then, by proceeding as in the proof of the main property (Case 3) in \cite{NS} we can see that the operator $T_k^{\beta_1, \ldots, \beta_l}$ is bounded from
  $L^1((0,\infty)^l, \overset{l}{\underset{j=1}{\prod}} x_j^{2\beta_j} dx)$ into $L^{1, \infty}((0,\infty)^l, \overset{l}{\underset{j=1}{\prod}} x_j^{2\beta_j} dx).$\\

  We conclude that $G_{1,1;2}^{\lambda_1, \ldots, \lambda_n}$ defines a bounded operator from $L^1((0,\infty)^n,$ $ \overset{n}{\underset{j=1}{\prod}} x_j^{2\lambda_j} dx)$
   into $L^{1, \infty}((0,\infty)^n, \overset{n}{\underset{j=1}{\prod}} x_j^{2\lambda_j} dx).$\\

  We are going to work with $G_{1,1;3}^{\lambda_1, \ldots, \lambda_n}.$ According to \cite[proof of Lemma~8]{BHNV} we have, for every
  $\lambda>-1/2,$ that
  \begin{align}\label{B11}
     \left| \frac{\partial}{\partial t} W_t^\lambda (x,y)- (xy)^{- \lambda} \frac{\partial}{\partial t} \mathbb{W}_t(x,y)\right| \leq
        C\frac{(xy)^{-\lambda-1}}{\sqrt{t}}& e^{-(x-y)^2/8t}, \ t,x,y \in (0, \infty),\ xy\geq t,
  \end{align}
  and
  \begin{equation}\label{B12}
     \left|\frac{\partial}{\partial t} W_t^\lambda (x,y)\right| \leq C \frac{e^{-(x-y)^2/8t}}{t^{\lambda + 3/2}}, \quad t,x,y \in (0,\infty), \; xy<t.
  \end{equation}
  By combining \eqref{B3,5}, \eqref{B11} and \eqref{B12} in [\cite{BHNV}, Lemma 8] it was established that
  \begin{equation}\label{B13}
     \left(\int_0^\infty t\left| \frac{\partial}{\partial t} W_t^\lambda (x,y)- (xy)^{- \lambda} \frac{\partial}{\partial t} \mathbb{W}_t(x,y) \right|^2dt \right)^{\frac{1}{2}}\leq \frac{C}{x^{2\lambda+1}},
      \quad 0<\frac{x}{2}< y < 2x,
  \end{equation}
  provided that $\lambda>-1/2.$\\

  The operator defined by $G_{1,1;3}^{\lambda_1, \ldots, \lambda_n}$ is bounded from $L^1((0,\infty)^n, \overset{n}{\underset{j=1}{\prod}} x_j^{2\lambda_j} dx)$ into
  $L^{1,\infty}((0,\infty)^n,$ $ \overset{n}{\underset{j=1}{\prod}} x_j^{2\lambda_j} dx)$ provided that, for every $1\leq l \leq m \leq n$ and
  $\alpha_j>-1/2, $ $j=1, \ldots, n,$ the operator
  \begin{align*}
     S&_{l,m}^{\alpha_1, \ldots, \alpha_n}(f)(x) = \left\{ \int_0^\infty t \left| \int_{x_{1/2}}^{2x_1}\left| \frac{\partial}{\partial t}
        W_t^{\alpha_1} (x_1,y_1)- (x_1y_1)^{- \alpha_1} \frac{\partial}{\partial t} \mathbb{W}_t(x_1,y_1)\right|\int_{\frac{x_2}{2}}^{2x_2} \ldots \int_{\frac{x_l}{2}}^{2x_l} \right. \right. \\
        &  \left. \left.\int_0^{\frac{x_{l+1}}{2}} \ldots \int_0^{{\frac{x_m}{2}}}
        \int_{2x_{m+1}}^\infty \ldots \int_{2x_n}^\infty \prod_{j=2}^n W_t^{\alpha_j}(x_j,y_j) |f(y)|  \prod_{j=1}^n y_j^{2\alpha_j} dy \right|^2 dt \right\}^{1/2}, \ x \in (0, \infty)^n,
  \end{align*}
  is bounded from $L^1((0,\infty)^n, \overset{n}{\underset{j=1}{\prod}} x_j^{2\alpha_j} dx)$ into
  $L^{1,\infty}((0,\infty)^n, \overset{n}{\underset{j=1}{\prod}} x_j^{2\alpha_j} dx).$\\

  Suppose that $l<m$ and $\alpha_j>-1/2, $ $ j=1, \ldots, n.$ According to \eqref{B13} and by using Minkowski inequality we obtain that
  \begin{align*}
     S_{l,m}^{\alpha_1, \ldots, \alpha_n}(f)(x) \leq &C\int_{\frac{x_1}{2}}^{2x_1}\frac{1}{x_1^{2\alpha_1+1}}\int_{\frac{x_2}{2}}^{2x_2} \ldots \int_{\frac{x_l}{2}}^{2x_l}\int_0^{\frac{x_{l+1}}{2}} \ldots \int_0^{{\frac{x_m}{2}}} \int_{2x_{m+1}}^\infty \ldots \int_{2x_n}^\infty\\
     & \sup_{t>0} \prod_{j=2}^n W_t^{\alpha_j}(x_j,y_j) |f(y)|\prod_{j=1}^n y_j^{2\alpha_j} dy, \quad x=(x_1, \ldots, x_n) \in (0, \infty)^n.
  \end{align*}
  In Section 2 it was proved that the operator
  \begin{align*}
     T_{l,m}^{\alpha_2, \ldots, \alpha_n}(g)(x)=&\int_{\frac{x_2}{2}}^{2x_2} \ldots \int_{\frac{x_l}{2}}^{2x_l}\int_0^{\frac{x_{l+1}}{2}} \ldots \int_0^{{\frac{x_m}{2}}} \int_{2x_{m+1}}^\infty \ldots \int_{2x_n}^\infty\\
     & \sup_{t>0} \prod_{j=2}^n W_t^{\alpha_j}(x_j,y_j) |g(y)|\prod_{j=2}^n y_j^{2\alpha_j} dy, \quad x=(x_2, \ldots, x_n) \in (0, \infty)^{n-1},
  \end{align*}
  is bounded from $L^1((0,\infty)^{n-1}, \overset{n}{\underset{j=2}{\prod}} x_j^{2\alpha_j} dx)$ into
  $L^{1,\infty}((0,\infty)^{n-1}, \overset{n}{\underset{j=2}{\prod}} x_j^{2\alpha_j} dx).$ Moreover, $H_{loc}^{\alpha_1}$ is bounded from
  $L^1((0,\infty), x^{2\alpha_1}dx)$ into itself. Then, \cite[Proposition 1]{Din} implies that $S_{l,m}^{\alpha_1, \ldots, \alpha_n}$
  is a bounded operator from $L^1((0,\infty)^n, \overset{n}{\underset{j=1}{\prod}} x_j^{2\alpha_j} dx)$ into
  $L^{1,\infty}((0,\infty)^n, \overset{n}{\underset{j=1}{\prod}} x_j^{2\alpha_j} dx).$\\

  Since the operator $H_{\infty}^k$ is bounded from $L^1((0, \infty)^k, \overset{k}{\underset{j=1}{\prod}} x_j^{2\beta_j}dx)$ into itself, for every $\beta_j>-1/2,$ $j=1, \ldots, k,$ we need only to see that the operator
  \begin{align*}
     \mathcal{S}_{k}^{\beta_1, \ldots, \beta_k}(g)(x) =&\left\{ \int_0^\infty t \left| \int_{\frac{x_1}{2}}^{2x_1}\int_{\frac{x_2}{2}}^{2x_2} \ldots
        \int_{\frac{x_k}{2}}^{2x_k}\left( \frac{\partial}{\partial t}
        W_t^{\beta_1} (x_1,y_1)- (x_1y_1)^{- \beta_1} \frac{\partial}{\partial t} \mathbb{W}_t(x_1,y_1)\right) \right. \right. \\
     &  \times  \left. \left. \prod_{j=2}^k W_t^{\beta_j}(x_j,y_j) g(y)\prod_{j=1}^k y_j^{2\beta_j} dy \right|^2 dt \right\}^{1/2}, \quad x=(x_1, \ldots, x_k) \in (0, \infty)^k,
  \end{align*}
  is bounded from $L^1((0, \infty)^k, \overset{k}{\underset{j=1}{\prod}} x_j^{2\beta_j}dx)$ into
  $L^{1, \infty}((0, \infty)^k, \overset{k}{\underset{j=1}{\prod}} x_j^{2\beta_j}dx),$ when $\beta_j >-1/2,$ $j=1, \ldots, k,$ to show that the operator
  $S_{m,m}^{\alpha_1, \ldots, \alpha_n}, $ $ m=1, \ldots, n,$ has desired boundedness property.\\

  Let $\beta_j>-1/2,$ $j=1,\ldots,k.$ By \eqref{B3,5}, \eqref{B11} and \eqref{B12} we have that
  \begin{align*}
     \mathcal{S}_{k}^{\beta_1, \ldots, \beta_k}(g)(x) = &\left\{ \left( \int_0^{x_1^2}+ \int_{x_1^2}^\infty\right) t
        \left| \int_{{\frac{x_1}{2}}}^{2x_1}\int_{\frac{x_2}{2}}^{2x_2} \ldots\int_{\frac{x_k}{2}}^{2x_k}\left( \frac{\partial}{\partial t}
        W_t^{\beta_1} (x_1,y_1) \right. \right. \right.\\
     & -\left. \left.\left.  (x_1y_1)^{- \beta_1} \frac{\partial}{\partial t} \mathbb{W}_t(x_1,y_1)\right) \prod_{j=2}^k W_t^{\beta_j}(x_j,y_j) g(y)
     \prod_{j=1}^k y_j^{2\beta_j} dy \right|^2 dt \right\}^{1/2}
  \end{align*}
  \begin{align*}
     \leq& C \left[ \left\{\int_0^{x_1^2} t \left|\int_{{\frac{x_1}{2}}}^{2x_1}\ldots\int_{\frac{x_k}{2}}^{2x_k}(x_1y_1)^{-\beta_1-1}
        \frac{e^{-(x_1-y_1)^2/8t}}{\sqrt{t}}\prod_{j=2}^k W_t^{\beta_j}(x_j,y_j) g(y)\prod_{j=1}^k y_j^{2\beta_j} dy \right|^2 dt \right\}^{1/2}\right.\\
     & + \left\{\int_{x_1^2}^\infty t \left|\int_{{\frac{x_1}{2}}}^{2x_1}\ldots\int_{\frac{x_k}{2}}^{2x_k}
        \frac{e^{-(x_1-y_1)^2/8t}}{t^{\beta_1+3/2}}\prod_{j=2}^k W_t^{\beta_j}(x_j,y_j) g(y)\prod_{j=1}^k y_j^{2\beta_j} dy \right|^2 dt \right\}^{1/2}\\
     & + \left.\left\{\int_{x_1^2}^\infty t \left|\int_{{\frac{x_1}{2}}}^{2x_1}\ldots\int_{\frac{x_k}{2}}^{2x_k}
         (x_1y_1)^{- \beta_1} \frac{\partial}{\partial t} \mathbb{W}_t(x_1,y_1)\prod_{j=2}^k W_t^{\beta_j}(x_j,y_j) g(y)\prod_{j=1}^k y_j^{2\beta_j} dy \right|^2 dt \right\}^{1/2}\right]\\
     \leq & C \left[\left\{ \int_0^{x_1^2}tdt\right\}^{1/2}\sup_{t>0} \int_{{\frac{x_1}{2}}}^{2x_1}\ldots\int_{\frac{x_k}{2}}^{2x_k}(x_1y_1)^{-\beta_1-1}
        \frac{e^{-(x_1-y_1)^2/8t}}{\sqrt{t}}\prod_{j=2}^k W_t^{\beta_j}(x_j,y_j) |g(y)|\prod_{j=1}^k y_j^{2\beta_j} dy \right.\\
     & + \left\{ \int_{x_1^2}^\infty \frac{1}{t^{2\beta_1+2}}dt\right\}^{1/2}
        \sup_{t>0}\int_{{\frac{x_1}{2}}}^{2x_1}\ldots\int_{\frac{x_k}{2}}^{2x_k}
        \prod_{j=2}^k W_t^{\beta_j}(x_j,y_j) |g(y)|\prod_{j=1}^k y_j^{2\beta_j} dy \\
     & + \left.\left\{ \int_{x_1^2}^\infty \frac{dt}{t^2}\right\}^{1/2}x_1^{-2\beta_1}
        \sup_{t>0}\int_{{\frac{x_1}{2}}}^{2x_1}\ldots\int_{\frac{x_k}{2}}^{2x_k}
        \prod_{j=2}^k W_t^{\beta_j}(x_j,y_j) |g(y)|\prod_{j=1}^k y_j^{2\beta_j} dy \right]\\
     \leq & C \left[ \sup_{t>0} \int_{{\frac{x_1}{2}}}^{2x_1}\ldots\int_{\frac{x_k}{2}}^{2x_k}(x_1y_1)^{-\beta_1}
        \frac{e^{-(x_1-y_1)^2/8t}}{\sqrt{t}}\prod_{j=2}^k W_t^{\beta_j}(x_j,y_j) |g(y)|\prod_{j=1}^k y_j^{2\beta_j} dy\right.\\
     & + \left.\sup_{t>0}\int_{{\frac{x_1}{2}}}^{2x_1}\ldots\int_{\frac{x_k}{2}}^{2x_k}
        \prod_{j=2}^k W_t^{\beta_j}(x_j,y_j)\frac{1}{x_1^{2\beta_1+1}} \int_{\frac{x_1}{2}}^{2x_1} g(y)y_1^{2\beta_1}dy_1 \prod_{j=2}^k y_j^{2\beta_j} dy_k \ldots dy_2
        \right], \ x \in (0, \infty)^k.
  \end{align*}
  Since the operator $H_{loc}^{\beta_1}$ is bounded from $L^1((0,\infty), y^{2\beta_1}dy)$ into itself, by taking into account the results established in the proof of Theorem \ref{maximal} (see Section 2), we conclude that the operator $\mathcal{S}_k^{\beta_1, \ldots, \beta_k}$ is bounded from
  $L^1((0, \infty)^k, \overset{k}{\underset{j=1}{\prod}} x_j^{2\beta_j}dx)$ into $L^{1, \infty}((0, \infty)^k, \overset{k}{\underset{j=1}{\prod}} x_j^{2\beta_j}dx).$\\

  Thus, we prove that $G_{1,1;3}^{\lambda_1, \ldots, \lambda_n}$ defines a bounded operator from
  $L^1((0, \infty)^n, \overset{n}{\underset{j=1}{\prod}} x_j^{2\lambda_j}dx)$ into $L^{1, \infty}((0, \infty)^n, \overset{n}{\underset{j=1}{\prod}} x_j^{2\lambda_j}dx).$\\

  According to \eqref{B7} we establish that the operator defined by $G_{1,1}^{\lambda_1, \ldots, \lambda_n}$ is bounded from
  $L^1((0, \infty)^n,$ $ \overset{n}{\underset{j=1}{\prod}} x_j^{2\lambda_j}dx)$ into $L^{1, \infty}((0, \infty)^n, \overset{n}{\underset{j=1}{\prod}} x_j^{2\lambda_j}dx).$\\

  Assume now that $i=2,\ldots, n.$ We analyze the operator $G_{1,i}^{\lambda_1, \ldots, \lambda_n}.$ In order to do this we write
  \begin{align*}
    G&_{1,i}^{\lambda_1, \ldots, \lambda_n}(f)(x) \leq \left\{ \int_0^\infty t
       \left|\int_{{\frac{x_1}{2}}}^{2x_1}\ldots\int_{\frac{x_{i-1}}{2}}^{2x_{i-1}} (x_1y_1)^{- \lambda_1}\frac{\partial}{\partial t}\left(
       \frac{e^{\frac{-(x_1-y_1)^2}{4t}}}{2\sqrt{\pi t}}\right) \prod_{ j=2}^{i-1}(x_jy_j)^{-\lambda_j}\frac{e^{\frac{-(x_j-y_j)^2}{4t}}}{ 2\sqrt{\pi t}}
        \right. \right.\\
    & \times \left. \left. \int_0^{\frac{x_i}{2}}\int_{(0, \infty)^{n-i}}
       \prod_{j=i}^n W_t^{\lambda_j}(x_j,y_j) f(y)\prod_{j=1}^n y_j^{2\lambda_j} dy\right|^2dt\right\}^{1/2}+ \left\{ \int_0^\infty t\left|\int_{{\frac{x_1}{2}}}^{2x_1}
       \ldots\int_{\frac{x_{i-1}}{2}}^{2x_{i-1}}(x_1y_1)^{- \lambda_1}  \right. \right.\\
    & \times \left. \left.  \frac{\partial}{\partial t}\left( \frac{e^{\frac{-(x_1-y_1)^2}{4t}}}{2\sqrt{\pi t}}\right)
      \prod_{ j=2}^{i-1}(x_jy_j)^{-\lambda_j}\frac{e^{\frac{-(x_j-y_j)^2}{4t}}}{ 2 \sqrt{\pi t}} \int_{2x_i}^{\infty}
       \int_{(0, \infty)^{n-i}} \prod_{j=i}^n W_t^{\lambda_j}(x_j,y_j) f(y)\prod_{j=1}^n y_j^{2\lambda_j} dy\right|^2dt\right\}^{1/2}
  \end{align*}
  \begin{align*}
      &+ \left\{ \int_0^\infty t\left|\int_{{\frac{x_1}{2}}}^{2x_1}\ldots\int_{\frac{x_{i-1}}{2}}^{2x_{i-1}}(x_1y_1)^{- \lambda_1} \frac{\partial}{\partial t}\left(
       \frac{e^{\frac{-(x_1-y_1)^2}{4t}}}{2\sqrt{\pi t}}\right)\prod_{ j=2}^{i-1}(x_jy_j)^{-\lambda_j}\frac{e^{\frac{-(x_j-y_j)^2}{4t}}}{ 2\sqrt{\pi t}}
        \right. \right.\\
    & \times \left. \left.\int_{\frac{x_i}{2}}^{2x_i}\left( W_t^{\lambda_i}(x_i,y_i) -(x_iy_i)^{- \lambda_i} \frac{e^{\frac{-(x_i-y_i)^2}{4t}}}
       {2\sqrt{\pi t}}\right) \int_{(0,\infty)^{n-i}}
       \prod_{j=i+1}^n W_t^{\lambda_j}(x_j,y_j) f(y)\prod_{j=1}^n y_j^{2\lambda_j} dy\right|^2dt\right\}^{1/2}\\
    =&\sum_{l=1}^3 G_{1,i;l}(f)(x), \quad x=(x_1,\ldots, x_n) \in (0,\infty)^n.
  \end{align*}
  In order to analyze the boundedness property for $G_{1,i;1}$ and $G_{1,i;2}$ we can follow a similar way  to the one employed to study $G_{1,1;1}$ and
  $G_{1,1;2},$ respectively.\\

  On the other hand, by \eqref{E1} we get, for every $\lambda>-1/2,$ $t,x,y\in (0, \infty) ,$
  \begin{align}\label{B14}
     \left|  W_t^\lambda (x,y)- (xy)^{- \lambda} \mathbb{W}_t(x,y) \right|
        \leq C &\left( \frac{1}{t^{\lambda+1/2}} + \frac{(xy)^{-\lambda}}{\sqrt{t}}\right)e^{- \frac{x^2+ y^2}{4t}}, \quad xy \leq t.
  \end{align}
  Also, \eqref{E2} implies that
  \begin{align}\label{B15}
     \left|  W_t^\lambda (x,y)- (xy)^{- \lambda}  \mathbb{W}_t(x,y) \right|
        \leq C & (xy)^{-\lambda-1}\sqrt{t}e^{- \frac{(x-y)^2}{4t}}, \quad xy > t,
  \end{align}
  when $\lambda>-1/2,$ $t,x,y\in (0, \infty).$\\

  By using \eqref{B14} and \eqref{B15}, and proceeding as in the study of $G_{1,1;3}$ we can prove that $G_{1,i;3}$ defines a bounded operator from $L^1((0, \infty)^n, \overset{n}{\underset{j=1}{\prod}} x_j^{2\lambda_j}dx)$
  into $L^{1, \infty}((0, \infty)^n, \overset{n}{\underset{j=1}{\prod}} x_j^{2\lambda_j}dx).$\\

  Thus, we conclude that $G_{1,i}$ is bounded from $L^1((0, \infty)^n, \overset{n}{\underset{j=1}{\prod}} x_j^{2\lambda_j}dx)$
  into $L^{1, \infty}((0, \infty)^n,$ $ \overset{n}{\underset{j=1}{\prod}} x_j^{2\lambda_j}dx).$\\

  Hence, the operator $G_1$ is bounded from $L^1((0, \infty)^n, \overset{n}{\underset{j=1}{\prod}} x_j^{2\lambda_j}dx)$
  into $L^{1, \infty}((0, \infty)^n,\overset{n}{\underset{j=1}{\prod}} x_j^{2\lambda_j}$ $dx),$ and the proof of Claim \ref{Claim2} is finished. \qed

  \section{Proof of Theorem \ref{Riesz}}

    In this section we prove that the Riesz transforms $R_i^{\lambda_1, \dots, \lambda_n}$, $i=1, \dots, n$, associated with the Bessel operator
    $\Delta_{\lambda_1, \dots, \lambda_n}$, are bounded from $L^1((0,\infty)^n,\overset{n}{\underset{j=1}{\prod}} x_j^{2\lambda_j}dx)$
    into $L^{1,\infty}((0,\infty)^n,\overset{n}{\underset{j=1}{\prod}}x_j^{2\lambda_j}dx)$.\\

    For every $i=1, \dots, n$, and $\lambda_j > -1/2$, $j=1, \dots, n$, the Riesz transform $R_i^{\lambda_1, \dots, \lambda_n}$ is formally defined  by
    \begin{equation}\label{C1}
      R_i^{\lambda_1, \dots, \lambda_n} = \frac{\partial}{\partial x_i} \Delta_{\lambda_1, \dots , \lambda_n}^{-1/2},
    \end{equation}
    where $\Delta_{\lambda_1, \dots , \lambda_n}^{-1/2}$ denotes the negative square root of the operator $ \Delta_{\lambda_1, \dots , \lambda_n}$. We are going
    to precise the definition \eqref{C1}.\\

    Assume that $0 < \beta < \overset{n}{\underset{j=1}{\sum}} (\lambda_j + 1/2)+1 $ and $\lambda_j > -1/2$, $j=1, \dots, n$. We define the negative power
    $ \Delta_{\lambda_1, \dots , \lambda_n}^{-\beta}$ on $C_c^\infty((0, \infty)^n)$ as follows
    \begin{align*}
      \Delta_{\lambda_1, \dots , \lambda_n}^{-\beta}f(x)
         = &\frac{1}{\Gamma(\beta)} \int_0^\infty \left( W_t^{\lambda_1, \dots, \lambda_n} (f)(x)
            - \chi_{(1,\infty)}(t) \frac{t^{-\overset{n}{\underset{j=1}{\sum}}(\lambda_j + 1/2)}}{\overset{n}{\underset{j=1}{\prod}} 2^{2\lambda_j}\Gamma(\lambda_j + 1/2)}
             \right.
    \end{align*}
    \begin{align*}
        & \times \left.  \int_{(0,\infty)^n} f(y) \prod_{j=1}^n y_j^{2\lambda_j} dy \right)t^{\beta-1} dt, \quad x \in (0,\infty)^n,
    \end{align*}
    where $f \in C_c^\infty((0,\infty)^n)$.\\

    Let $f \in C_c^\infty((0,\infty)^n)$. We can write
    \begin{equation}\label{C2}
      \Delta_{\lambda_1, \dots , \lambda_n}^{-\beta}f(x) = \int_{(0,\infty)^n} f(y) K_\beta^{\lambda_1, \dots, \lambda_n}(x,y) \prod_{j=1}^n y_j^{2\lambda_j} dy,
      \quad x \in (0,\infty)^n,
    \end{equation}
    being, for each $x,y \in (0,\infty)^n,$ $x\neq y$,
    \begin{align*}
      K_\beta^{\lambda_1, \dots, \lambda_n}(x,y)
        & = \frac{1}{\Gamma(\beta)} \int_0^\infty \left( \prod_{j=1}^n W_t^{\lambda_j}(x_j,y_j) - \chi_{(1,\infty)}(t) \frac{t^{-\overset{n}{\underset{j=1}{\sum}}(\lambda_j + 1/2)}}{\overset{n}{\underset{j=1}{\prod}} 2^{2\lambda_j}\Gamma(\lambda_j + 1/2)}
           \right)t^{\beta-1} dt.
    \end{align*}

    In order to prove the integral representation \eqref{C2} for $\Delta_{\lambda_1, \dots , \lambda_n}^{-\beta}$ it is sufficient to show that, for every $x \in (0,\infty)^n,$
    \begin{align*}
      \int_{(0,\infty)^n} \lvert f(y) \rvert  \int_0^\infty & \left \lvert \prod_{j=1}^n W_t^{\lambda_j}(x_j,y_j) - \chi_{(1,\infty)}(t)
         \frac{t^{-\overset{n}{\underset{j=1}{\sum}}(\lambda_j + 1/2)}}{\overset{n}{\underset{j=1}
         {\prod}} 2^{2\lambda_j}\Gamma(\lambda_j + 1/2)} \right \rvert t^{\beta-1} dt
         \prod_{j=1}^n y_j^{2\lambda_j} dy < \infty.
    \end{align*}
    Indeed, let $x=(x_1, \dots, x_n) \in (0,\infty)^n$. According to \eqref{E1} and \eqref{E2} and denoting $K = supp f$, we have that
    \begin{align*}
       \int_{(0,\infty)^n} \lvert f(y) \rvert \int_0^1 \prod_{j=1}^n & W_t^{\lambda_j}(x_j,y_j) t^{\beta-1}  dt  \prod_{j=1}^n y_j^{2\lambda_j} dy \\
        \leq &  C \int_K \int_0^1 t^{\beta - 1 -n/2}e^{- \overset{n}{\underset{j=1}{\sum}}(x_j-y_j)^2/4t} dt dy\\
        \leq  &   C \int_K \frac{1}{\lvert x-y \rvert^{n-2\beta} } dy < \infty,
    \end{align*}
    because $\beta >0$. In the last inequality we have used \cite[Lemma 1.1]{ST}.\\

    On the other hand, by using again \eqref{E1}, \eqref{E2} and \eqref{E3} it follows
    \begin{align}\label{Z}
       \Big \lvert W_t^\lambda(x,y) - &\frac{t^{-\lambda-1/2}}{2^{2\lambda}\Gamma(\lambda+1/2)} \Big \rvert
         = \Big \lvert \frac{1}{(2t)^{\lambda+1/2}} \left( \frac{xy}{2t}\right)^{-\lambda+1/2} I_{\lambda-1/2}\left(\frac{xy}{2t}\right)e^{-\frac{x^2+y^2}{4t}} - \frac{t^{-\lambda-1/2}}{2^{2\lambda}\Gamma(\lambda+1/2)} \Big  \rvert\nonumber\\
         \leq & \frac{1}{(2t)^{\lambda+1/2}} e^{-\frac{x^2+y^2}{4t}} \Big \lvert \left( \frac{xy}{2t}\right)^{-\lambda+1/2} I_{\lambda-1/2}\left(\frac{xy}{2t}\right) - \frac{1}{2^{\lambda-1/2}\Gamma(\lambda+1/2)} \Big  \rvert\nonumber\\
       & + \frac{t^{-\lambda-1/2}}{2^{2\lambda}\Gamma(\lambda+1/2)} \Big \lvert e^{-\frac{x^2+y^2}{4t}} -1 \Big \rvert\nonumber\\
       \leq & \frac{xy}{(2t)^{\lambda+3/2}} e^{-\frac{x^2+y^2}{4t}}  \sup_{z \in (0,\frac{xy}{2t})} \Big \lvert \frac{d}{dz} \left( z^{-\lambda+1/2} I_{\lambda-1/2}(z) \right) \Big \rvert
         + C \frac{x^2+y^2}{t^{\lambda+3/2}}\nonumber\\
         \leq & C \left( \frac{(xy)^2}{t^{\lambda+5/2}} e^{-\frac{(x-y)^2}{4t}}  + \frac{x^2+y^2}{t^{\lambda+3/2}} \right)\nonumber\\
         \leq & C \frac{x^2+y^2}{t^{\lambda+3/2}} \left( \frac{xy}{t} +1 \right), \quad t,x,y \in (0,\infty).
     \end{align}
     Then, by \eqref{E1} we get
    \begin{align*}
       &\int_{(0,\infty)^n} \lvert f(y) \rvert  \int_1^\infty   \left \lvert \prod_{j=1}^n W_t^{\lambda_j}(x_j,y_j)
        -\frac{t^{-\overset{n}{\underset{j=1}{\sum}}(\lambda_j + 1/2)}}{\overset{n}{\underset{j=1} {\prod}} 2^{2\lambda_j}\Gamma(\lambda_j + 1/2)} \right
        \rvert t^{\beta-1} dt  \prod_{j=1}^n y_j^{2\lambda_j} dy\\
      & \leq C \sum_{i=1}^n \int_K \int_1^\infty \prod_{j=1}^{i-1} \frac{t^{-(\lambda_j+1/2)}}{2^{2\lambda_j}\Gamma(\lambda_j +1/2)}
        \left| W_t^{\lambda_i}(x_i,y_i) - \frac{t^{-(\lambda_i+1/2)}}{2^{2\lambda_i}\Gamma(\lambda_i +1/2)} \right|\prod_{j=i+1}^n W_t^{\lambda_j}(x_j,y_j)t^{\beta-1} dt dy\\
      &\leq C \int_K \int_1^\infty t^{- \overset{n}{\underset{j=1}{\sum}}(\lambda_j +1/2)-2+\beta} dt dy<\infty ,
    \end{align*}
    because $\beta < \overset{n}{\underset{j=1}{\sum}} (\lambda_j+1/2)+1$.\\

    In particular, we have
    $$\Delta_{\lambda_1, \dots, \lambda_n}^{-1/2}f(x)
      = \int_{(0,\infty)^n} f(y) K_{1/2}^{\lambda_1, \dots, \lambda_n}(x,y) \prod_{j=1}^n y_j^{2\lambda_j} dy, \quad x \in (0,\infty). $$

    \begin{Rm}
      Note that
      if $0 < \beta < \overset{n}{\underset{j=1}{\sum}} (\lambda_j+1/2)$, for every $f \in C_c^\infty((0,\infty)^n)$,
      $$\int_{(0,\infty)^n} \lvert f(y) \rvert \int_0^\infty \prod_{j=1}^n W_t^{\lambda_j}(x_j,y_j) t^{\beta-1} dt \prod_{j=1}^n y_j^{2\lambda_j} dy <\infty,
        \quad x \in (0,\infty)^n,$$
      and we can define the function
      $$\left( \Delta_{\lambda_1, \dots, \lambda_n} \right)^{-\beta} f(x)
        = \int_{(0,\infty)^n}  \frac{ f(y)}{\Gamma(\beta)} \int_0^\infty \prod_{j=1}^n W_t^{\lambda_j}(x_j,y_j) t^{\beta-1} dt \prod_{j=1}^n y_j^{2\lambda_j} dy,
        \quad x \in (0,\infty)^n,$$
      when $f \in C_c^\infty((0,\infty)^n)$. Then, if $0 < \beta < \overset{n}{\underset{j=1}{\sum}} (\lambda_j+1/2)$ and $f \in C_c^\infty((0,\infty)^n)$, for each $x \in (0,\infty)^n,$
      \begin{align*}
        \Delta_{\lambda_1, \dots, \lambda_n}^{-\beta}f(x) - \left( \Delta_{\lambda_1, \dots, \lambda_n} \right)^{-\beta} f(x)
         =&  \frac{-1}{\Gamma(\beta)} \int_{(0,\infty)^n}  f(y) \int_1^\infty \prod_{j=1}^n \frac{t^{-(\lambda_j + 1/2)}}{2^{2\lambda_j}\Gamma(\lambda_j + 1/2)}
          t^{\beta-1} dt  \prod_{j=1}^n y_j^{2\lambda_j} dy
      \end{align*}
      and, for every $i=1, \dots, n$,
      $$\frac{\partial}{\partial x_i} \Delta_{\lambda_1, \dots, \lambda_n}^{-\beta} f(x)
         = \frac{\partial}{\partial x_i} \left( \Delta_{\lambda_1, \dots, \lambda_n } \right)^{-\beta} f(x),$$
      provided that this derivative exists.
    \end{Rm}

    For every $0 < \beta < n/2$, the fractional power $\Delta^{-\beta}$ is defined on $C_c^\infty(\mathbb{R}^n)$ by
    \begin{align*}
      \Delta^{-\beta} f(x)  = &\frac{1}{\Gamma(\beta)} \int_0^\infty t^{\beta-1} \frac{1}{(2\sqrt{\pi})^n} \int_{\mathbb{R}^n} \frac{e^{- \lvert x-y \rvert^2/4t}}{t^{n/2}} f(y) dy dt\\
                            = &\int_{\mathbb{R}^n} f(y) \frac{1}{\Gamma(\beta)(2\sqrt{\pi})^n} \int_0^\infty t^{\beta-1-n/2}e^{- \lvert x-y \rvert^2/4t} dt dy\\
                            = & \frac{\Gamma(n/2-\beta)}{\pi^{n/2}4^\beta \Gamma(\beta)} \int_{\mathbb{R}^n} \frac{f(y)}{\lvert x - y \rvert^{n-2\beta}} dy,
                           \quad x \in \mathbb{R}^n \text{ and } f \in C_c^\infty(\mathbb{R}^n).
    \end{align*}

    Note that all the above integrals are absolutely convergent. In particular, since $n>1$, we have that, for every $f \in C_c^\infty(\mathbb{R}^n)$,
    $$\Delta^{-1/2}f(x) = \frac{1}{2^n \pi^{\frac{n+1}{2}}} \int_{\mathbb{R}^n} f(y) \int_0^\infty \frac{e^{- \lvert x-y \rvert^2/4t}}{t^{\frac{n+1}{2}}} dt dy, \quad x \in \mathbb{R}^n.$$

    A crucial result to prove Theorem~\ref{Riesz} is the following one.

    \begin{Prop}\label{Proposition 4.1}
      Let $f \in C_c^\infty((0,\infty)^n)$. Assume that $\lambda_j>-1/2$, $j=1, \dots, n$. Then, for every $i=1, \dots, n$, it has
      \begin{align}\label{C3}
         \frac{\partial}{\partial x_i}& \left( \Delta_{\lambda_1, \dots, \lambda_n}^{-1/2} f(x) -
             \prod_{j=1}^n x_j^{-\lambda_j} \Delta^{-1/2}\big(\prod_{j=1}^n y_j^{\lambda_j} f\big)(x) \right) \\
        & = \int_{(0,\infty)^n} f(y) \left( R_i^{\lambda_1, \dots, \lambda_n}(x,y) - \mathcal{R}_i^{\lambda_1, \dots, \lambda_n}(x,y)  \right)
            \prod_{j=1}^n y_j^{2\lambda_j} dy, \quad  a.e. \ x \in (0,\infty)^n,\nonumber
      \end{align}
       where
      $$\mathcal{R}_i^{\lambda_1, \dots, \lambda_n}(x,y) = \frac{1}{2^n \pi^{\frac{n+1}{2}}} \frac{\partial}{\partial x_i}
        \left( \int_0^\infty \frac{e^{- \lvert x-y \rvert^2/4t}}{t^{\frac{n+1}{2}}} dt \prod_{j=1}^n (x_j y_j)^{-\lambda_j}\right),
       \ x, y \in \mathbb{R}^n, \ x \neq y,$$
      and
      $$R_i^{\lambda_1, \dots, \lambda_n}(x,y) = \frac{1}{\sqrt{\pi}} \int_0^\infty \frac{\partial}{\partial x_i} W_t^{\lambda_i}(x_i,y_i)
        \prod_{j=1, \ j \neq i}^{n} W_t^{\lambda_j}(x_j,y_j) \frac{dt}{\sqrt{t}}, \ x, y \in (0,\infty)^n, \ x \neq y.$$
      Moreover, the integral in \eqref{C3} is absolutely convergent.
    \end{Prop}

    \begin{proof}
      We are going to prove \eqref{C3} for $i=1$. We denote $K=supp f$. Suppose that $\Omega$ is a compact subset of $(0,\infty)^n$.
      There exist $0<a<1$ and $b>1$, such that for every $y \in K$ and $x \in \Omega$, $\frac{x_j y_j}{t} \geq 1$, when $0<t<a$, and
      $\frac{x_j y_j}{t} \leq 1$, when $b<t<\infty$, for every $j=1,...,n$. We can write, for each $x,y \in (0,\infty)^n$,
      \begin{align}\label{C4}
        K&_{1/2}^{\lambda_1, \dots, \lambda_n}(x,y) - \frac{\overset{n}{\underset{j=1}{\prod}} (x_j y_j)^{-\lambda_j}}{2^n \pi^{\frac{n+1}{2}}}
                \int_0^\infty \frac{e^{- \lvert x-y \rvert^2/4t}}{t^{\frac{n+1}{2}}} dt    \\
       = &   \frac{1}{\sqrt{\pi}}\left\{ \int_0^a  \left(\prod_{j=1}^n W_t^{\lambda_j}(x_j,y_j) - \prod_{j=1}^n (x_jy_j)^{-\lambda_j}
                \frac{e^{-(x_j-y_j)^2/4t}}{2\sqrt{\pi t}}\right)  \frac{dt}{\sqrt{t}}  \right. \nonumber \\
        & + \int_a^1  \left(\prod_{j=1}^n W_t^{\lambda_j}(x_j,y_j) - \prod_{j=1}^n (x_jy_j)^{-\lambda_j}
                \frac{e^{-(x_j-y_j)^2/4t}}{2\sqrt{\pi t}}\right)  \frac{dt}{\sqrt{t}}  \nonumber \\
        &  \left. + \int_1^b \left( \prod_{j=1}^n W_t^{\lambda_j}(x_j,y_j)
                                               - \frac{t^{-\overset{n}{\underset{j=1}{\sum}}(\lambda_j + 1/2)}}{\overset{n}{\underset{j=1}{\prod}} 2^{2\lambda_j}\Gamma(\lambda_j + 1/2)}
                                               - \frac{\overset{n}{\underset{j=1}{\prod}} (x_j y_j)^{-\lambda_j}}{2^n \pi^{\frac{n}{2}}}
                                                 \frac{e^{- \lvert x-y \rvert^2/4t}}{t^{\frac{n}{2}}} \right) \frac{dt}{\sqrt{t}}\right.   \nonumber\\
        &  \left. + \int_b^\infty \left( \prod_{j=1}^n W_t^{\lambda_j}(x_j,y_j)
                                               - \frac{t^{-\overset{n}{\underset{j=1}{\sum}}(\lambda_j + 1/2)}}{\overset{n}{\underset{j=1}{\prod}} 2^{2\lambda_j}\Gamma(\lambda_j + 1/2)}
                                               - \frac{\overset{n}{\underset{j=1}{\prod}} (x_j y_j)^{-\lambda_j}}{2^n \pi^{\frac{n}{2}}}
                                                 \frac{e^{- \lvert x-y \rvert^2/4t}}{t^{\frac{n}{2}}}\right) \frac{dt}{\sqrt{t}}\right\}. \nonumber
      \end{align}
       According to \eqref{E2}, we have for every $x,y \in (0,\infty)^n$ and $t>0$,
      \begin{align*}
        \prod_{j=1}^n W_t^{\lambda_j}(x_j,y_j)
          & = \prod_{j=1}^n \frac{(x_j y_j)^{-\lambda_j}}{\sqrt{2t}} \left( \frac{x_j y_j}{2t} \right)^{1/2}
              I_{\lambda_j -1/2}\left( \frac{x_j y_j}{2t} \right) e^{-(x_j^2 + y_j^2)/4t}\\
          & = \prod_{j=1}^n \frac{(x_j y_j)^{-\lambda_j}}{2 \sqrt{\pi t}} e^{-\frac{(x_j-y_j)^2}{4t}}
              \left( \sum_{k=0}^n (-1)^k [\lambda_j-1/2,k] \left( \frac{t}{x_jy_j} \right)^k + f_n^{\lambda_j-1/2} \left( \frac{x_jy_j}{2t} \right) \right),
      \end{align*}
      where, for every $\nu>-1$, $f_n^\nu$ is a $C^\infty(0,\infty)$-function  such that $f_n^\nu(z)=O\left( z^{-(n+1)}\right)$, as $z \rightarrow \infty$. Then, for each $x,y \in (0,\infty)^n$ and $t>0$,
      \begin{align}\label{34}
        \prod_{j=1}^n W_t^{\lambda_j}(x_j,y_j) = &  \prod_{j=1}^n \frac{(x_j y_j)^{-\lambda_j}}{2 \sqrt{\pi t}} e^{-\frac{(x_j-y_j)^2}{4t}} \\
                                                 & \times  \left( 1 + \sum_{\tiny{\begin{array}{c} k_1, \dots, k_n \in \mathbb{N}\\ k_1, \dots, k_n \leq n \\ (k_1, \dots, k_n) \neq 0 \end{array}}}
                                                 c_{k_1, \dots, k_n} \frac{t^{k_1+ \dots + k_n}}{\overset{n}{\underset{j=1}{\prod}} (x_jy_j)^{k_j}} + g_n(x,y,t)\right), \nonumber
      \end{align}
      where $c_{k_1, \dots, k_n} \in \mathbb{R}$, $k_1, \dots, k_n \in \mathbb{N}$, $k_1, \dots, k_n \leq n$, and $(k_1,\cdots, k_n)\neq 0$, and $g_n \in C^\infty((0,\infty)^n \times (0,\infty)^n \times (0,\infty))$
      and
      $$\lvert g_n(x,y,t) \lvert \leq C t^{n+1}, \quad  t \in (0,a), \ x \in \Omega \text{ and } y \in K.$$
      Let $\nu>-1$. By \eqref{E2} and \eqref{E3} we get, for each $z \in (0,\infty)$ ,
       \begin{align}\label{35}
        \frac{d}{dz}\big( e^{-z}z^{-\nu}&I_{\nu}(z) \big)  =  e^{-z}z^{-\nu-1/2} \left( \sqrt{z} I_{\nu+1} - \sqrt{z}I_\nu(z) \right) \nonumber\\
                                                            & =  \frac{z^{-\nu-1/2}}{\sqrt{2\pi}} \left( \sum_{k=0}^{n+1} (-1)^k ([\nu+1,k]-[\nu,k])
                                                            (2z)^{-k} + f_{n+1}^{\nu+1}(z) - f_{n+1}^\nu(z)\right) \nonumber\\
                                                            & =- \frac{z^{-\nu-1/2}}{\sqrt{2\pi}} \left( \sum_{k=0}^n (-1)^k ([\nu+1,k+1]-[\nu,k+1])
                                                            (2z)^{-k-1} + f_{n+1}^{\nu+1}(z) - f_{n+1}^\nu(z)\right),
      \end{align}
      and
      \begin{align}\label{36}
        \frac{d}{dz}\big( e^{-z}z^{-\nu}&I_{\nu}(z) \big)  =  \frac{d}{dz}\left(  \frac{z^{-\nu-1/2}}{\sqrt{2\pi}}
                                                                  \left(  \sum_{k=0}^n (-1)^k [\nu,k](2z)^{-k} + f_{n}^{\nu}(z) \right)\right) \nonumber \\
                                                            & = - \frac{z^{-\nu-1/2}}{\sqrt{2\pi}}
                                                            \left(  \sum_{k=0}^n (-1)^k [\nu,k](2\nu+2k+1)(2z)^{-k-1} + \frac{\nu+1/2}{z} f_{n}^{\nu}(z) - \frac{d}{dz}f_n^\nu(z) \right).
      \end{align}
      Moreover, $[\nu+1,k+1]-[\nu,k+1] = [\nu,k](2\nu+2k+1)$, for $k \in \mathbb{N}$, $k \geq 1$. Indeed, let $k \geq 1$. We have that
      \begin{align*}
        [\nu&+1, k+1]-[\nu,k+1]  = \frac{1}{2^{2k+2} \Gamma(k+2)} \left\{\overset{k}{\underset{j=0}{\prod}}(4(\nu+1)^2-(2j+1)^2)-\overset{k}{\underset{j=0}{\prod}}(4\nu^2-(2j+1)^2) \right\} \\
               & =\frac{1}{2^{2k+2} \Gamma(k+2)} \left\{ \overset{k}{\underset{j=0}{\prod}}(2\nu+2(j+1)+1)(2\nu-2(j-1)-1) - \overset{k}{\underset{j=0}{\prod}}(2\nu+2j+1)(2\nu-2j-1) \right\} \\
               & =\frac{(2\nu+2k+3)-(2\nu-2k-1)}{4 (k+1) 2^{2k} \Gamma(k+1)} \left\{ \overset{k-1}{\underset{j=0}{\prod}}(2\nu+2j+1)(2\nu-2j-1) \right\} (2\nu+2k+1) \\
               & =[\nu,k](2\nu+2k+1).
      \end{align*}
      Then, from \eqref{35} and \eqref{36} we deduce that
      $$\frac{d}{dz} f_n^\nu = \frac{\nu+1/2}{z}f_n^\nu - f_{n+1}^{\nu+1}+f_{n+1}^{\nu}.$$
      Hence,
      \begin{equation}\label{37}
        \frac{d}{dz} f_n^\nu(z)=O\left( \frac{1}{z^{n+2}} \right), \quad \text{as } z \rightarrow \infty,
      \end{equation}
      and we conclude that
      \begin{equation}\label{38}
        \left \lvert \frac{\partial}{\partial x_1}g_n(x,y,t) \right\rvert \leq C t^{n+1}, \quad t \in (0,a), \ x \in \Omega \text{ and } y \in K.
      \end{equation}
 If $k \in \mathbb{N}$, according to \cite[Lemma 1.1]{ST} we get
      \begin{align}\label{39}
        \int_0^a e^{\lvert x-y \rvert^2/4t} t^{k-\frac{n+1}{2}} dt & \leq \int_0^a e^{\lvert x-y \rvert^2/4t} t^{-\frac{n-1}{2}} dt
                                                                    \leq C \frac{1}{ \lvert x-y \rvert^{n-2}}, \quad x\neq y, \ x,y \in (0,\infty)^n.
      \end{align}

      Suppose that $k_1, \dots, k_n \in \mathbb{N}$ such that $k_1 + \dots + k_n \leq n$ and $k_i \geq 1$, for some $i=1, \dots, n$. From \eqref{39} it deduces that
      the functions
      $$h_{k_1, \dots, k_n}(z)=\int_0^a \frac{e^{-\lvert z \rvert^2/4t}}{t^{\frac{n+1}{2} - \overset{n}{\underset{j=1}{\sum}}k_j}} dt \quad , \quad
        H_{k_1, \dots, k_n}(z)=\int_0^a \frac{e^{-\lvert z \rvert^2/4t}}{t^{\frac{n}{2}+1 - \overset{n}{\underset{j=1}{\sum}}k_j}} dt, \quad z \in \mathbb{R}^n \setminus \{0\}$$
      are in $L^1(\Lambda)$ for every compact subset $\Lambda \subset \mathbb{R}^n$. Moreover,
      $$\frac{\partial}{\partial x_1}h_{k_1, \dots, k_n}(x)=\int_0^a \frac{\partial}{\partial x_1}e^{-\lvert x \rvert^2/4t}t^{-\frac{n+1}{2} + \overset{n}{\underset{j=1}{\sum}}k_j} dt, \quad x \in \mathbb{R}^n \setminus \{0\}.$$
      Since $f \in C_c^\infty((0,\infty)^n)$ by defining $f(y)=0$, $y \in \mathbb{R}^n \setminus (0,\infty)^n$, the function
      $$G_{k_1, \dots, k_n}(x)= \int_{\mathbb{R}^n} \prod_{j=1}^n (x_jy_j)^{-\lambda_j-k_j}f(y)h_{k_1, \dots, k_n}(x-y) \prod_{j=1}^n y_j^{2\lambda_j}dy, \quad x \in (0,\infty)^n,$$
      is derivable with respect to $x_1$ on $(0,\infty)^n$, and
      \begin{align*}
        &\frac{\partial}{\partial x_1}  G_{k_1, \dots, k_n}(x)
         = -(\lambda_1 + k_1) \int_{\mathbb{R}^n}x_1^{-\lambda_1 - k_1 -1}y_1^{-\lambda_1-k_1} \prod_{j=2}^n (x_jy_j)^{-\lambda_j-k_j}f(y)h_{k_1, \dots,k_n}(x-y) \prod_{j=1}^n y_j^{2\lambda_j}dy \\
         & + \int_{\mathbb{R}^n}x_1^{-\lambda_1 - k_1} \prod_{j=2}^n (x_j(x_j-y_j))^{-\lambda_j-k_j}\frac{\partial}{\partial x_1}\left[(x_1-y_1)^{\lambda_1-k_1}f(x-y) \right] h_{k_1, \dots,k_n}(y) \prod_{j=2}^n (x_j-y_j)^{2\lambda_j}dy \\
         = & -(\lambda_1 + k_1) \int_{\mathbb{R}^n}x_1^{-\lambda_1 - k_1 -1}y_1^{-\lambda_1-k_1} \prod_{j=2}^n (x_jy_j)^{-\lambda_j-k_j}f(y)h_{k_1, \dots,k_n}(x-y) \prod_{j=1}^n y_j^{2\lambda_j}dy \\
         &- \int_{\mathbb{R}^n}x_1^{-\lambda_1 - k_1} \prod_{j=2}^n (x_j(x_j-y_j))^{-\lambda_j-k_j}\frac{\partial}{\partial y_1}\left[(x_1-y_1)^{\lambda_1-k_1}f(x-y) \right] h_{k_1, \dots,k_n}(y) \prod_{j=2}^n (x_j-y_j)^{2\lambda_j}dy \\
          = & -(\lambda_1 + k_1) \int_{\mathbb{R}^n}x_1^{-\lambda_1 - k_1 -1}y_1^{-\lambda_1-k_1} \prod_{j=2}^n (x_jy_j)^{-\lambda_j-k_j}f(y)h_{k_1, \dots,k_n}(x-y) \prod_{j=1}^n y_j^{2\lambda_j}dy \\
&-\int_{\mathbb{R}^{n-1}} x_1^{-\lambda_1-k_1}\prod_{j=2}^n \frac{(x_j-y_j)^{2\lambda_j}}{(x_j(x_j-y_j))^{\lambda_j+k_j}}\int_{-\infty}^\infty \frac{\partial}{\partial y_1}\left[(x_1-y_1)^{\lambda_1-k_1}f(x-y)\right] h_{k_1,\ldots,k_n}(y) dy_1\ldots dy_n\\
=& -(\lambda_1+k_1) \int_{\mathbb{R}^n} x_1^{-\lambda-k_1-1}y_1^{-\lambda_1-k_1}\prod_{j=2}^n (x_jy_j)^{-\lambda_j-k_j}f(y) h_{k_1,\ldots,k_n}(x-y) \prod_{j=1}^n y_j^{2\lambda_j}dy\\
& +  \int_{\mathbb{R}^{n-1}}  \lim_{\varepsilon \rightarrow 0} \left\{ -(x_1-y_1)^{\lambda_1-k_1}f(x-y) h_{k_1,\ldots,k_n}(y) \right]_{x_1+\varepsilon}^{\infty}
 -  \left.(x_1-y_1)^{\lambda_1-k_1}f(x-y) h_{k_1,\ldots, k_n}(y) \right]_{-\infty}^{x_1-\varepsilon} \\
 & + \int_{\mathbb{R}\setminus(x_1-\varepsilon,x_1+\varepsilon)} (x_1-y_1)^{\lambda_1-k_1}f(x-y) \frac{\partial}{\partial y_1}\left.h_{k_1,\ldots,k_n}(y) dy_1\right\} x_1^{-\lambda_1-k_1}\prod_{j=2}^n \frac{(x_j-y_j)^{2\lambda_j}}{(x_j(x_j-y_j))^{\lambda_j+k_j}} dy_2 \ldots dy_n\\
 =& -(\lambda_1+k_1)\int_{\mathbb{R}^n} x_1^{-\lambda_1-k_1-1}y_1^{-\lambda_1-k_1}\prod_{j=2}^n (x_jy_j)^{-\lambda_j-k_j}f(y) h_{k_1,\ldots,k_n}(x-y) \prod_{j=1}^n y_j^{2\lambda_j}dy\\
 &+ \int_{\mathbb{R}^n} \prod_{j=1}^n (x_j(x_j-y_j))^{-\lambda_j-k_j}\prod_{j=1}^n(x_j-y_j)^{2\lambda_j}f(x-y) \frac{\partial}{\partial y_1}h_{k_1,\ldots,k_n}(y)dy\\
 =& -(\lambda_1+k_1)\int_{\mathbb{R}^n} x_1^{-\lambda_1-k_1-1}y_1^{-\lambda_1-k_1}\prod_{j=2}^n (x_jy_j)^{-\lambda_j-k_j}f(y) h_{k_1,\ldots,k_n}(x-y) \prod_{j=1}^n y_j^{2\lambda_j}dy\\
   \end{align*}
 \begin{align*}
 &+ \int_{\mathbb{R}^n}\prod_{j=1}^n (x_jy_j)^{-\lambda_j-k_j}f(y) \frac{\partial}{\partial u_1}h_{k_1,\ldots,k_n}(u)_{|u=x-y}\prod_{j=1}^n y_j^{2\lambda_j}dy\\
 =& -(\lambda_1+k_1)\int_{\mathbb{R}^n} x_1^{-\lambda_1-k_1-1}y_1^{-\lambda_1-k_1}\prod_{j=2}^n (x_jy_j)^{-\lambda_j-k_j}f(y) h_{k_1,\ldots,k_n}(x-y) \prod_{j=1}^n y_j^{2\lambda_j}dy\\
 &+ \int_{\mathbb{R}^n}\prod_{j=1}^n (x_jy_j)^{-\lambda_j-k_j}f(y)\frac{\partial}{\partial x_1}h_{k_1,\ldots,k_n}(x-y)\prod_{j=1}^n y_j^{2\lambda_j}dy\\
 =& \int_{\mathbb{R}^n}\frac{\partial}{\partial x_1}\left[\prod_{j=1}^n (x_jy_j)^{-\lambda_j-k_j}h_{k_1,\ldots,k_n}(x-y) \right]f(y) \prod_{j=1}^n y_j^{2\lambda_j}dy, \quad x\in (0,\infty)^n.
\end{align*}
Hence, we obtain, for each $x\in (0,\infty)^n,$
\begin{align}\label{40}
&\frac{\partial}{\partial x_1}\Big( \int_{\mathbb{R}^n} f(y) \int_0^a \prod_{j=1}^n \frac{(x_jy_j)^{-\lambda_j}}{2\sqrt{\pi t}} e^{-\frac{(x_j-y_j)^2}{4t}} \sum_{\tiny{\begin{array}{c} k_1,\ldots,k_n \in \mathbb{N}\\
 k_1,\ldots,k_n \leq n \\ (k_1,\ldots,k_n)\neq 0\end{array}}} c_{k_1,\ldots,k_n}\frac{t^{k_1+\ldots+k_n}}{\overset{n}{\underset{j=1}{\prod}} (x_jy_j)^{k_j}}dt\prod_{j=1}^n y_j^{2\lambda_j}dy\Big) \\
=& \int_{\mathbb{R}^n} f(y) \int_0^a \frac{\partial}{\partial x_1}\Big(\prod_{j=1}^n \frac{(x_jy_j)^{-\lambda_j}}{2\sqrt{\pi t}} e^{-\frac{(x_j-y_j)^2}{4t}} \sum_{\tiny{\begin{array}{c} k_1,\ldots,k_n \in \mathbb{N}\\
 k_1,\ldots,k_n \leq n \\ (k_1,\ldots,k_n)\neq 0\end{array}}} c_{k_1,\ldots,k_n}\frac{t^{k_1+\ldots+k_n}}{\overset{n}{\underset{j=1}{\prod}} (x_jy_j)^{k_j}}\Big)\frac{dt}{\sqrt{t}}\prod_{j=1}^n y_j^{2\lambda_j}dy.  \nonumber
\end{align}

Also, by using \eqref{38} we can see that
\begin{align}\label{41}
\frac{\partial}{\partial x_1}\int_{\mathbb{R}^n}& f(y) \int_0^a \prod_{j=1}^n \frac{(x_jy_j)^{-\lambda_j}}{2\sqrt{\pi t}} e^{-\frac{(x_j-y_j)^2}{4t}} g_n(x,y,t) dt\prod_{j=1}^n y_j^{2\lambda_j} dy\\
= &\int_{\mathbb{R}^n} f(y) \int_0^a \frac{\partial}{\partial
x_1}\left(\prod_{j=1}^n \frac{(x_jy_j)^{-\lambda_j}}{2\sqrt{\pi t}}
e^{-\frac{(x_j-y_j)^2}{4t}} g_n(x,y,t)\right) dt\prod_{j=1}^n
y_j^{2\lambda_j} dy, \quad  x\in (0,\infty)^n. \nonumber
\end{align}
 From \eqref{34}, \eqref{40} and \eqref{41} we deduce that
 \begin{align}\label{42}
&\frac{\partial}{\partial x_1}\int_{\mathbb{R}^n} f(y) \int_0^a \left(\prod_{j=1}^n W_t^{\lambda_j} (x_j,y_j) -\prod_{j=1}^n \frac{(x_jy_j)^{-\lambda_j}}{2\sqrt{\pi t}} e^{-\frac{(x_j-y_j)^2}{4t}}\right) \frac{dt}{\sqrt{t}} \prod_{j=1}^n y_j^{2\lambda_j}dy\\
= &\int_{\mathbb{R}^n}f(y) \int_0^a \frac{\partial}{\partial
x_1}\left(\prod_{j=1}^n W_t^{\lambda_j} (x_j,y_j) -\prod_{j=1}^n
\frac{(x_jy_j)^{-\lambda_j}}{2\sqrt{\pi t}}
e^{-\frac{(x_j-y_j)^2}{4t}}\right) \frac{dt}{\sqrt{t}} \prod_{j=1}^n
y_j^{2\lambda_j}dy, \; x\in (0,\infty)^n. \nonumber
\end{align}
By \eqref{Z} it follows that
\begin{align*}
&\left|\prod_{j=1}^n W_t^{\lambda_j}(x_j,y_j) - \frac{t^{-\overset{n}{\underset{j=1}{\sum}}(\lambda_j + 1/2)}}{\overset{n}{\underset{j=1}{\prod}} 2^{2\lambda_j}\Gamma(\lambda_j+1/2)}- \prod_{j=1}^n \frac{(x_jy_j)^{-\lambda_j}}{2\sqrt{\pi t}} e^{-\frac{(x_j-y_j)^2}{4t}}\right| \frac{1}{\sqrt{t}}\\
\leq & \sum_{i=1}^n \frac{1}{\sqrt{t}}\prod_{j=1}^{i-1} W_t^{\lambda_j}(x_j,y_j)\left| W_t^{\lambda_i}(x_i,y_i) - \frac{t^{-\lambda_i - 1/2}}{ 2^{2\lambda_i}\Gamma(\lambda_i+1/2)}\right| \frac{t^{-\overset{n}{\underset{j=i+1}{\sum}}(\lambda_j + 1/2)}}{\overset{n}{\underset{j=i+1}{\prod}} 2^{2\lambda_j}\Gamma(\lambda_j+1/2)} + \frac{C}{t^{(n+1)/2}}\\
\leq& C\left(  \ \frac{1}{t^{\overset{n}{\underset{j=1}{\sum}}(\lambda_j+1/2)+ 3/2}}+ \frac{1}{t^{(n+1)/2}}\right), \quad b <t <\infty, x\in \Omega \text{ and } y \in K.
\end{align*}
Moreover, according to \eqref{E3} we have that
\begin{align*}
&\frac{\partial}{\partial x_1}\left( \prod_{j=1}^n W_t^{\lambda_j}(x_j,y_j) - \frac{t^{-\overset{n}{\underset{j=1}{\sum}}(\lambda_j + 1/2)}}{\overset{n}{\underset{j=1}{\prod}} 2^{2\lambda_j}\Gamma(\lambda_j+1/2)}- \prod_{j=1}^n \frac{(x_jy_j)^{-\lambda_j}}{2\sqrt{\pi t}} e^{-\frac{(x_j-y_j)^2}{4t}}\right)\\
=& \frac{\overset{n}{\underset{{j=2}}{\prod}} W_t^{\lambda_j} (x_j,y_j) }{(2t)^{\lambda_1+1/2}}\left[ \left( \frac{x_1y_1}{2t}\right)^{-\lambda_1+1/2}I_{\lambda_1+1/2}\left( \frac{x_1y_1}{2t}\right)\frac{y_1}{2t}- \frac{x_1}{2t}\left(\frac{x_1y_1}{2t} \right)^{-\lambda_1+1/2}I_{\lambda_1-1/2}\left( \frac{x_1y_1}{2t}\right)\right]e^{-\frac{x_1^2+y_1^2}{4t}}
\\&+ \lambda_1 \frac{x_1^{-\lambda_1-1}y_1^{-\lambda_1}}{2\sqrt{\pi t}} \prod_{j=2}^n (x_jy_j)^{-\lambda_j}\frac{e^{-\frac{(x_j-y_j)^2}{4t}}}{2\sqrt{\pi t}}+ \prod_{j=1}^n (x_jy_j)^{-\lambda_j} \frac{e^{-\frac{|x-y|^2}{4t}}}{(2\sqrt{\pi t})^n}\frac{(x_1-y_1)}{2t}, \quad t>0, x,y \in (0,\infty)^n.
\end{align*}
Then, \eqref{E1} leads to
\begin{align*}
\Big|\frac{\partial}{\partial x_1}\Big( \prod_{j=1}^n W_t^{\lambda_j}(x_j,y_j) &- \frac{t^{-\overset{n}{\underset{j=1}{\sum}}(\lambda_j + 1/2)}}{\overset{n}{\underset{j=1}{\prod}} 2^{2\lambda_j}\Gamma(\lambda_j+1/2)}-  \frac{e^{-\frac{|x-y|^2}{4t}}}{(2\sqrt{\pi t})^n} \prod_{j=1}^n (x_jy_j)^{-\lambda_j}\Big)\Big|\\
\leq& C\left( \frac{1}{t^{\overset{n}{\underset{j=1}{\sum}}(\lambda_j+1/2)+1}} + \frac{1}{t^{n/2}}\right), \quad t>b, x \in \Omega \text{ and } y\in K.
\end{align*}
Hence, for each $  x \in (0,\infty)^n ,$ we can differentiate under the integral sign obtaining
\begin{align}\label{43}
&\frac{\partial}{\partial x_1}\int_{(0,\infty)^n} f(y)   \int_b^\infty \Big(W_t^{\lambda_j}(x_j,y_j) - \frac{t^{-\overset{n}{\underset{j=1}{\sum}}(\lambda_j + 1/2)}}{\overset{n}{\underset{j=1}{\prod}} 2^{2\lambda_j}\Gamma(\lambda_j+1/2)}- \prod_{j=1}^n (x_jy_j)^{-\lambda_j} \frac{e^{-\frac{|x-y|^2}{4t}}}{(2\sqrt{\pi t})^n}\Big) \frac{dt}{\sqrt{t}} dy\\
=& \int_{(0,\infty)^n} f(y) \prod_{j=1}^n y_j^{2\lambda_j}  \int_b^\infty \frac{\partial}{\partial x_1}\Big(W_t^{\lambda_j}(x_j,y_j) - \frac{t^{-\overset{n}{\underset{j=1}{\sum}}(\lambda_j + 1/2)}}{\overset{n}{\underset{j=1}{\prod}} 2^{2\lambda_j}\Gamma(\lambda_j+1/2)}- \prod_{j=1}^n (x_jy_j)^{-\lambda_j} \frac{e^{-\frac{|x-y|^2}{4t}}}{(2\sqrt{\pi t})^n}\Big) \frac{dt}{\sqrt{t}}\prod_{j=1}^n y_j^{2\lambda_j} dy . \nonumber
\end{align}
Finally, it is not hard to see that
\begin{align}\label{44}
&\frac{\partial}{\partial x_1}\int_{(0,\infty)^n} f(y)  \left( \int_a^b \Big(W_t^{\lambda_j}(x_j,y_j) -  \frac{e^{-\frac{|x-y|^2}{4t}}}{(2\sqrt{\pi t})^n}\prod_{j=1}^n (x_jy_j)^{-\lambda_j}\Big)\frac{dt}{\sqrt{t}}-\int_1^b  \frac{t^{-\overset{n}{\underset{j=1}{\sum}}(\lambda_j + 1/2)}}{\overset{n}{\underset{j=1}{\prod}} 2^{2\lambda_j}\Gamma(\lambda_j+1/2)} \frac{dt}{\sqrt{t}}\right)\prod_{j=1}^n y_j^{2\lambda_j} dy\\
=& \int_{(0,\infty)^n} f(y) \int_a^b \frac{\partial}{\partial x_1}\Big(W_t^{\lambda_j}(x_j,y_j) -  \frac{e^{-\frac{|x-y|^2}{4t}}}{(2\sqrt{\pi t})^n}\prod_{j=1}^n (x_jy_j)^{-\lambda_j}\Big) \frac{dt}{\sqrt{t}}\prod_{j=1}^n y_j^{2\lambda_j} dy, \quad  x \in (0,\infty)^n .  \nonumber
\end{align}
By combining \eqref{C4}, \eqref{42}, \eqref{43} and \eqref{44} we can prove \eqref{C3}. Moreover the estimations that we have established show the absolute convergence of the integral in \eqref{C3}.\\

Thus the proof of Proposition \ref{Proposition 4.1} is finished.
    \end{proof}

    Proposition \ref{Proposition 4.1} allows us to define the Riesz transforms $R_i^{\lambda_1, \ldots, \lambda_n},$ $i=1, \ldots, n,$ on $C_c^\infty ((0,\infty)^n).$

 \begin{Prop}\label{Proposition 4.2}
 Let $f\in C_c^\infty ((0,\infty)^n)$ and $\lambda_j>-1/2,$ $j=1, \ldots, n.$ Then, for every $i=1, \ldots, n,$ the function $\Delta_{\lambda_1, \ldots, \lambda_n}^{-1/2} f$ admits derivative $\frac{\partial}{\partial x_i} \Delta_{\lambda_1, \ldots, \lambda_n}^{-1/2} f$ with respect to $x_i$ on almost all $(0,\infty)^n$, and
 $$\frac{\partial}{\partial x_i} \Delta_{\lambda_1, \ldots, \lambda_n}^{-1/2} f(x)= \lim_{\varepsilon \rightarrow 0^+} \int_{|x-y|>\varepsilon}
    R_i^{\lambda_1,\ldots, \lambda_n} (x,y) f(y) \prod_{j=1}^n y_j^{2\lambda_j} dy, \text{ a.e. } x\in (0,\infty)^n. $$
 \end{Prop}

 \begin{proof}
 Let $i=1, \ldots, n.$ As it is well known, for every $g\in C_c^\infty(\mathbb{R}^n),$ $\Delta^{-1/2}g$ admits derivative $\frac{\partial}{\partial x_i} \Delta^{-1/2} g$ with respect to $x_i$ on almost all $\mathbb{R}^n$ and
 $$ \frac{\partial}{\partial x_i} \Delta^{-1/2} g(x)= \lim_{\varepsilon \rightarrow 0^+} \int_{|x-y|>\varepsilon} g(y) \frac{\partial}{\partial x_i}\int_0^\infty \frac{e^{-\frac{|x-y|^2}{4t}}}{2^n(\pi t)^{(n+1)/2}} dt dy, \text{ a.e. } x\in (0,\infty)^n.$$
 Moreover, for every $x\in (0,\infty)^n,$ it has
 $$\int_{(0,\infty)^n} |f(y)|\prod_{j=1}^n y_j^{\lambda_j}\int_0^\infty \frac{e^{-\frac{|x-y|^2}{4t}}}{t^{\frac{n+1}{2}}}dt dy \leq C \int_{(0,\infty)^n} \frac{|f(y)| \overset{n}{\underset{j=1}{\prod}} y_j^{\lambda_j}}{|x-y|^{n-1}} dy <  \infty.$$
 Then,
 \begin{align*}
    \frac{\partial}{\partial x_i}&\left( \prod_{j=1}^n x_j^{-\lambda_j}\Delta^{-1/2}\big(\prod_{j=1}^n y_j^{\lambda_j} f\big )(x)\right)\\
       &\qquad = \lim_{\varepsilon \rightarrow 0^+} \int_{|x-y|>\varepsilon} f(y) \frac{\partial}{\partial x_i}\left( \prod_{j=1}^n x_j^{-\lambda_j} \int_0^\infty \frac{e^{-\frac{|x-y|^2}{4t}}}{2^n(\pi t)^{(n+1)/2}} dt\right)\prod_{j=1}^n y_j^{\lambda_j}dy, \text{ a.e. } x\in (0,\infty)^n.
 \end{align*}
 Hence, from Proposition \ref{Proposition 4.1} we conclude that $\Delta_{\lambda_1, \ldots, \lambda_n}^{-1/2} f $ admits derivative  $\frac{\partial}{\partial x_i}\Delta_{\lambda_1,\ldots, \lambda_n}^{-1/2}$ with respect to $x_i$ on almost all $(0,\infty)^n$ and
 \begin{align*}
    \frac{\partial}{\partial x_i}\Delta_{\lambda_1,\ldots, \lambda_n}^{-1/2} f(x)= & \frac{\partial}{\partial x_i}
       \left( \Delta_{\lambda_1, \ldots, \lambda_n}^{-1/2} f(x) - \prod_{j=1}^n x_j^{\lambda_j}\Delta^{-1/2}\big(f\prod_{j=1}^n y_j^{\lambda_j}\big)(x)\right)\\
      & + \frac{\partial}{\partial x_i}\left( \prod_{j=1}^n x_j^{\lambda_j}\Delta^{-1/2}\big(f\prod_{j=1}^n y_j^{\lambda_j}\big)(x)\right)\\
    = & \lim_{\varepsilon \rightarrow 0^+} \int_{|x-y|>\varepsilon} f(y) \left( R_i^{\lambda_1,\ldots, \lambda_n}(x,y)-
       \mathcal{R}_i^{\lambda_1,\ldots, \lambda_n}(x,y)\right)\prod_{j=1}^n y_j^{2\lambda_j} dy\\
    & + \lim_{\varepsilon \rightarrow 0^+} \int_{|x-y|>\varepsilon} f(y)
       \mathcal{R}_i^{\lambda_1,\ldots, \lambda_n}(x,y)\prod_{j=1}^n y_j^{2\lambda_j} dy\\
    =& \lim_{\varepsilon \rightarrow 0^+} \int_{|x-y|>\varepsilon} f(y)
       R_i^{\lambda_1,\ldots, \lambda_n}(x,y)\prod_{j=1}^n y_j^{2\lambda_j} dy, \text{ a.e. } x\in(0,\infty)^n.
 \end{align*}
 \end{proof}
 We now prove that, for every $f\in L^p((0,\infty)^n, \overset{n}{\underset{j=1}{\prod}} x_j^{2\lambda_j} dx),$ $1\leq p <\infty,$ exists the limit
 $$ \lim_{\varepsilon \rightarrow 0^+} \int_{|x-y|>\varepsilon} R_i^{\lambda_1, \ldots, \lambda_n}(x,y) f(y) \prod_{j=1}^n y_j^{2\lambda_j}dy, \text{ a.e } x\in (0,\infty)^n,$$
 with $i=1, \ldots, n.$ In order to show this we consider, for every $i=1, \ldots, n,$ the maximal operator $R_{i,*}^{\lambda_1, \ldots,\lambda_n}$ defined by
 $$ R_{i,*}^{\lambda_1, \ldots,\lambda_n}(f)(x) = \sup_{\varepsilon>0}\left| \int_{|x-y|>\varepsilon} R_i^{\lambda_1,\ldots, \lambda_n} (x,y) f(y) \prod_{j=1}^n y_j^{2\lambda_j}dy \right|. $$
 \begin{Prop}\label{Proposition 4.3}
 Let $\lambda_j >-1/2, $ $j=1, \ldots, n,$ and $i=1,\ldots, n.$ The maximal operator $R_{i,*}^{\lambda_1, \ldots, \lambda_n}$ is bounded from
 $L^p((0,\infty)^n, \overset{n}{\underset{j=1}{\prod}} x_j^{2\lambda_j} dx)$ into itself, $1< p<\infty$ and from $ L^1((0,\infty)^n, \overset{n}{\underset{j=1}{\prod}} x_j^{2\lambda_j} dx)$ into $L^{1,\infty}((0,\infty)^n, \overset{n}{\underset{j=1}{\prod}} x_j^{2\lambda_j} dx).$
 \end{Prop}
 \begin{proof}
 We consider $i=1.$ For the other values of $i$ we can proceed analogously. We can write
 \begin{align*}
   & R_1^{\lambda_1, \ldots, \lambda_n}(x,y)= \frac{1}{\sqrt{\pi}}\int_0^\infty \left( \frac{\partial}{\partial x_1}
       W_t^{\lambda_1}(x_1,y_1)-\chi_{\{ x_1/2<y_1<2x_1\}}(y_1)(x_1y_1)^{-\lambda_1}\frac{\partial}{\partial x_1}\mathbb{W}_t(x_1,y_1) \right)\\
    &\times\prod_{j=2}^n W_t^{\lambda_j}(x_j,y_j)\frac{dt}{\sqrt{t}}\\
    &+ \sum_{l=2}^n \frac{1}{\sqrt{\pi}}\int_0^\infty (x_1y_1)^{-\lambda_1}
       \frac{\partial}{\partial x_1}\mathbb{W}_t(x_1,y_1) \prod_{j=2}^{l-1}(x_jy_j)^{-\lambda_j}\mathbb{W}_t(x_j,y_j)\\
    & \times \prod_{j=1}^{l-1}\chi_{\{x_j/2<y_j<2x_j\}}(y_j)\left(W_t^{\lambda_l}(x_l,y_l)-\chi_{\{ x_l/2<y_l<2x_l\}}(y_l)(x_ly_l)^{-\lambda_l}\mathbb{W}_t(x_l,y_l)  \right)
   \prod_{j=l+1}^n W_t^{\lambda_j}(x_j,y_j)\frac{dt}{\sqrt{t}}\\
   & + \frac{1}{\sqrt{\pi}}\int_0^\infty (x_1y_1)^{-\lambda_1}
       \frac{\partial}{\partial x_1}\mathbb{W}_t(x_1,y_1) \prod_{j=2}^{n}(x_jy_j)^{-\lambda_j}\mathbb{W}_t(x_j,y_j)\prod_{j=1}^{n}\chi_{\{x_j/2<y_j<2x_j\}}(y_j)\frac{dt}{\sqrt{t}}\\
    = & \sum_{l=1}^{n+1} R_{1,l}^{\lambda_1, \ldots, \lambda_n}(x,y), \quad x,y \in (0, \infty)^n.
 \end{align*}

 We now study the maximal operator $R_{1,l,*}^{\lambda_1, \ldots, \lambda_n}$ associated with $R_{1,l}^{\lambda_1, \ldots, \lambda_n}$ in the usual way. Firstly we consider $R_{1,1,*}^{\lambda_1, \ldots, \lambda_n}.$ According to \cite[(22)]{MS}, \eqref{E1} and \eqref{E2}, we have
 \begin{equation}\label{X1}
 \Big|\frac{\partial}{\partial u}W_t^\lambda(u,v)\Big|\le C\frac{u+v}{t^{\lambda+3/2}}e^{-(u^2+v^2)/4},\,\,\,t,u,v\in (0,\infty)\,\,and\,\,\frac{uv}{t}\le 1,
 \end{equation}
 and
 \begin{equation}\label{X2}
 \Big|\frac{\partial}{\partial u}W_t^\lambda(u,v)\Big|\le C (uv)^{-\lambda}\frac{e^{-(u-v)^2/8t}}{t}, \,\,\,t,u,v\in (0,\infty)\,\,and\,\,\frac{uv}{t}\ge 1,
 \end{equation}
Inequalities \eqref{X1} and \eqref{X2} lead, for every $\lambda >-1/2$ and $0<v<u/2,$, to
 \begin{equation}\label{C14}
    \left| \frac{\partial}{\partial u}W_t^\lambda (u,v)\right|\leq C \frac{ue^{-u^2/32t}}{t^{\lambda+3/2}} \leq C \frac{e^{-u^2/40t}}{t^{\lambda+1}}, \quad t >0
 \end{equation}
 Suppose that $1\le l\leq m\leq n.$ According to \eqref{A0}, \eqref{A3}, \eqref{A6} and \eqref{C14}, we get
 \begin{align*}
   \int_0^{\frac{x_1}{2}}\ldots \int_0^{\frac{x_l}{2}} &\int_{\frac{x_{l+1}}{2}}^{2x_{l+1}}\ldots \int_{\frac{x_m}{2}}^{2x_m}\int_{2x_{m+1}}^\infty
      \ldots \int_{2x_n}^\infty |R_{1,1}^{\lambda_1, \ldots, \lambda_n}(x,y)| |f(y)|\prod_{j=1}^n y_j^{2\lambda_j}dy\\
   \leq & C \int_0^{\frac{x_1}{2}}\ldots \int_0^{\frac{x_l}{2}} \int_{\frac{x_{l+1}}{2}}^{2x_{l+1}}\ldots
      \int_{\frac{x_m}{2}}^{2x_m}\int_{2x_{m+1}}^\infty \ldots \int_{2x_n}^\infty \prod_{j=m+1}^n \frac{1}{y_j^{2\lambda_j+1}}|f(y)|\\
   & \times \int_0^\infty \frac{e^{- (\overset{l}{\underset{j=1}{\sum}} x_j^2)/40t}}{t^{\overset{l}{\underset{j=1}{\sum}}(\lambda_j+ 1/2)+1}}
      \prod_{j=l+1}^m \left( \frac{1}{x_j^{2\lambda_j+1}}+ \frac{(x_jy_j)^{-\lambda_j}}{\sqrt{t}} e^{-(x_j-y_j)^2/4t}\right)dt \prod_{j=1}^n y_j^{2\lambda_j}dy
 \end{align*}
 Since, for every $k\in \mathbb{N},$ $H_\infty^k$ and $H_{loc}^{\alpha_1, \ldots, \alpha_k}$ are bounded from $L^1((0,\infty)^k, \overset{k}{\underset{j=1}{\prod}} x_j^{2\alpha_j}dx)$ into itself, provided that $\alpha_j>-1/2,$ $j=1,\ldots, k,$ the operator
 $$ S_{l,m}^{\lambda_1,\ldots, \lambda_n}(f)(x)= \int_0^{\frac{x_1}{2}}\ldots \int_0^{\frac{x_l}{2}} \int_{\frac{x_{l+1}}{2}}^{2x_{l+1}}\ldots \int_{\frac{x_m}{2}}^{2x_m}\int_{2x_{m+1}}^\infty \ldots \int_{2x_n}^\infty |R_{1,1}^{\lambda_1, \ldots, \lambda_n}(x,y)| |f(y)|\prod_{j=1}^n y_j^{2\lambda_j}dy$$
 is bounded from $L^1((0,\infty)^n, \overset{n}{\underset{j=1}{\prod}} x_j^{2\lambda_j}dx)$ into $L^{1,\infty}((0,\infty)^n, \overset{n}{\underset{j=1}{\prod}} x_j^{2\lambda_j}dx),$ provided that, for every $r,k\in \mathbb{N}$, $1<r\leq k$ and $\alpha_j>-1/2,$ $j=1,\ldots,k,$ the operator
 \begin{align*}
    \mathcal{S}_r^{\alpha_1,\ldots,\alpha_k}(g)(x) =& \int_0^{\frac{x_1}{2}}\ldots \int_0^{\frac{x_r}{2}} \int_{\frac{x_{r+1}}{2}}^{2x_{r+1}}
       \ldots \int_{\frac{x_k}{2}}^{2x_k}|g(y)|\\
       & \times \int_0^\infty\prod_{j=r+1}^k (x_jy_j)^{-\alpha_j}\frac{e^{-(\overset{r}{\underset{j=1}{\sum}}x_j^2 +\overset{k}{\underset{j=r+1}{\sum}}(x_j-y_j)^2)/40t}}
       {t^{\overset{r}{\underset{j=1}{\sum}}(\alpha_j+1/2)+1+(k-r)/2}}dt \prod_{j=1}^ky_j^{2\alpha_j} dy\\
    = & c_{\alpha_1, \ldots, \alpha_k}^{k,r}\int_0^{\frac{x_1}{2}}\ldots \int_0^{\frac{x_r}{2}} \int_{\frac{x_{r+1}}{2}}^{2x_{r+1}}
       \ldots \int_{\frac{x_k}{2}}^{2x_k}\prod_{j=r+1}^k (x_jy_j)^{-\alpha_j} |g(y)| \\
    & \times \frac{1}{\big(\overset{r}{\underset{j=1}{\sum}}x_j^2+\overset{k}{\underset{j=r+1}{\sum}}
       (x_j-y_j)^2\big)^{\overset{r}{\underset{j=1}{\sum}}(\alpha_j+1/2)+(k-r)/2}}\prod_{j=1}^ky_j^{2\alpha_j}dy,
 \end{align*}
 where $ c_{\alpha_1, \ldots, \alpha_k}^{k,r}$ is a constant, is bounded from $L^1((0,\infty)^k, \overset{k}{\underset{j=1}{\prod}} x_j^{2\alpha_j}dx)$ into $L^{1,\infty}((0,\infty)^k, \overset{k}{\underset{j=1}{\prod}} x_j^{2\alpha_j}$ $dx).$
  This last property is true because the operators $\mathcal{H}_{r,k}^{\alpha_1,\ldots,\alpha_k}$ and $L_{\alpha_1,\ldots,\alpha_k}$ are bounded from
  $L^1((0,\infty)^k, \overset{k}{\underset{j=1}{\prod}} x_j^{2\alpha_j}dx)$ into $L^{1,\infty}((0,\infty)^k, \overset{k}{\underset{j=1}{\prod}} x_j^{2\alpha_j}dx),$
  when $\alpha_j>-1/2,$ $j=1,\ldots,k,$ and $1<r\leq k,$ $r,k\in \mathbb{N}.$\\

 On the other hand, as it was proved in \cite[Lemma 5]{BHNV},
 $$ \int_0^\infty \left| \frac{\partial}{\partial x_1}W_t^{\lambda_1}(x_1,y_1)\right|\frac{dt}{\sqrt{t}}\leq C \frac{1}{x_1^{2\lambda_1 +1}}, \quad 0<y_1 <x_1/2.$$
 Then, by writing $\bar{y}=(y_2,\ldots, y_n),$ it has, for every $x\in(0,\infty)^n,$
 \begin{align*}
    S_{l,m}^{\lambda_1,\ldots, \lambda_n}(f)(x)\leq & C \int_0^{\frac{x_1}{2}}\int_0^\infty \left|\frac{\partial}{\partial x_1} W_t^{\lambda_1}(x_1,y_1) \right|  \\
    &\times \left( \sup_{t>0} \int_{(0,\infty)^{n-1}} \prod_{j=2}^n W_t^{\lambda_j}(x_j,y_j) |f(y_1,\bar{y})|\prod_{j=2}^n y_j^{2\lambda_j}d\bar{y}\right) dt y_1^{2\lambda_1}dy_1\\
    \leq & \frac{C}{x_1^{2\lambda_1+1}}\int_0^{\frac{x_1}{2}} \sup_{t>0} \int_{(0,\infty)^{n-1}} \prod_{j=2}^n W_t^{\lambda_j}(x_j,y_j) |f(y_1,\bar{y})|\prod_{j=2}^n y_j^{2\lambda_j}d\bar{y} y_1^{2\lambda_1}dy_1.
 \end{align*}
 Since $L_{\lambda_1}$ is bounded from $L^p((0,\infty),x^{2\lambda_1}dx)$ into itself and $W_*^{\lambda_2, \ldots, \lambda_n}$ is bounded from
 $L^p((0,\infty)^n,$ $ \overset{n}{\underset{j=2}{\prod}} x_j^{2\alpha_j}d\overline{x})$ into itself, $S_{l,m}^{\lambda_1,\ldots, \lambda_n}$ is bounded from
 $L^p((0,\infty)^n, \overset{n}{\underset{j=1}{\prod}} x_j^{2\alpha_j}dx)$ into itself, for every $1<p<\infty.$\\

 By using \eqref{X1} and \eqref{X2}, for every $\lambda>-1/2$ and $2u<v<\infty,$ we obtain
 \begin{equation}\label{C15}
    \left| \frac{\partial}{\partial u}W_t^\lambda(u,v)\right|\leq \frac{C}{t^{\lambda+1}}e^{-v^2/40t}, \quad t>0.
 \end{equation}

 Suppose that $1\le l\leq m \leq n.$ By \eqref{A0}, \eqref{A3}, \eqref{A6} and \eqref{C15}, for each $ x=(x_1, \dots, x_n)\in(0, \infty)^n$, it follows that,
 \begin{align*}
    \mathcal{S}_{l,m}^{\lambda_1,\ldots, \lambda_n}&(f)(x)= \int_{2x_1}^\infty  \ldots \int_{2x_l}^\infty \int_{\frac{x_{l+1}}{2}}^{2x_{l+1}}\ldots
    \int_{\frac{x_m}{2}}^{2x_m}\int_{0}^{\frac{x_{m+1}}{2}} \ldots \int_0^{\frac{x_n}{2}} |R_{1,1}^{\lambda_1,\ldots,\lambda_n}(x,y)|
     |f(y)| \prod_{j=1}^n y_j^{2\lambda_j}dy
 \end{align*}
 \begin{align*}
    \leq & C \int_{2x_1}^\infty  \ldots \int_{2x_l}^\infty \int_{\frac{x_{l+1}}{2}}^{2x_{l+1}}\ldots \int_{\frac{x_m}{2}}^{2x_m}\int_{0}^{\frac{x_{m+1}}{2}}
    \ldots \int_0^{\frac{x_n}{2}} |f(y)|\int_0^\infty \frac{e^{- (\overset{l}{\underset{j=1}{\sum}}y_j^2+ \overset{n}{\underset{j=m+1}{\sum}}x_j^2)/40t}}{t^{\overset{l}{\underset{j=1}{\sum}}
    (\lambda_j + 1/2) + \overset{n}{\underset{j=m+1}{\sum}} (\lambda_j+1/2)+1}}\\
    & \times \prod_{j=l+1}^m \left[ (x_jy_j)^{-\lambda_j}\frac{e^{-(x_j-y_j)^2/4t}}{\sqrt{t}} +
    \frac{1}{x_j^{2\lambda_j+1}} \right] dt \prod_{j=1}^n y_j^{2\lambda_j} dy.
 \end{align*}

 Since the operator $H_{loc}^{\alpha_1,\ldots,\alpha_k}$ is bounded from $L^1((0,\infty)^k, \overset{k}{\underset{j=1}{\prod}} x_j^{2\alpha_j}dx)$ into itself, provided that $\alpha_j>-1/2,$ $j=1,\ldots, k,$ in order that to see that the operator $\mathcal{S}_{l,m}^{\lambda_1,\ldots, \lambda_n}$ is bounded from $L^1((0,\infty)^n, \overset{n}{\underset{j=1}{\prod}} x_j^{2\lambda_j}dx)$ into $L^{1,\infty}((0,\infty)^n, \overset{n}{\underset{j=1}{\prod}} x_j^{2\lambda_j}dx)$ it is sufficient to show that, for every $\alpha_j>-1/2,$ $j=1,\ldots,k,$ and $1<s\leq r\leq k,$ $s,r,k \in \mathbb{N},$ the operator
 \begin{align*}
    T&_{s,r}^{\alpha_1,\ldots, \alpha_k}(g)(x)= \int_{2x_1}^\infty \int_{2x_2}^\infty \ldots \int_{2x_s}^\infty \int_{\frac{x_{s+1}}{2}}^{2x_{s+1}}\ldots \int_{\frac{x_r}{2}}^{2x_r}\int_{0}^{\frac{x_{r+1}}{2}} \ldots \int_0^{\frac{x_k}{2}} |g(y)|\prod_{j=1}^k y_j^{2\alpha_j}\\
    & \times \frac{\overset{r}{\underset{j=s+1}{\prod}}(x_jy_j)^{-\alpha_j}dy}{(\overset{s}{\underset{j=1}{\sum}}y_j^2+ \overset{r}{\underset{j=s+1}{\sum}}(x_j-y_j)^2+ \overset{k}{\underset{j=r+1}{\sum}} x_j^2)^{\overset{s}{\underset{j=1}{\sum}} (\lambda_j + 1/2) + \overset{k}{\underset{j=r+1}{\sum}} (\lambda_j+1/2)+(r-s)/2}}, \ x\in (0,\infty)^k,
 \end{align*}
 is bounded from $L^1((0,\infty)^k, \overset{k}{\underset{j=1}{\prod}} x_j^{2\alpha_j}dx)$ into $L^{1,\infty}((0,\infty)^k, \overset{k}{\underset{j=1}{\prod}} x_j^{2\alpha_j}dx).$ Let $\alpha_j>-1/2,$ $j=1,\ldots,k,$ and $0<s\leq r\leq k.$
 By proceeding as in the proof of \cite[Case 3]{NS} we can prove that the operator $T_{s,r}^{\alpha_1,\ldots,\alpha_k}$ is bounded from $L^1((0,\infty)^k, \overset{k}{\underset{j=1}{\prod}} x_j^{2\alpha_j}dx)$ into $L^{1,\infty}((0,\infty)^k, \overset{k}{\underset{j=1}{\prod}} x_j^{2\alpha_j}dx),$ when $s<r.$
 Also, the operator $T_{s,s}^{\alpha_1,\ldots,\alpha_k}$ is bounded from $L^1((0,\infty)^k,$ $ \overset{k}{\underset{j=1}{\prod}} x_j^{2\alpha_j}dx)$ into $L^{1,\infty}((0,\infty)^k, \overset{k}{\underset{j=1}{\prod}} x_j^{2\alpha_j}dx),$ because
 \begin{align*}
    T_{s,s}(g)(x)\leq\frac{C}{(\overset{k}{\underset{j=s+1}{\sum}} x_j^2)^{\overset{k}{\underset{j=s+1}{\sum}}(\lambda_j+1/2)}}
    \int_0^{\frac{x_{s+1}}{2}}\ldots \int_0^{\frac{x_{k+1}}{2}}\int_{2x_1}^\infty \ldots \int_{2x_s}^\infty \frac{|g(y)|}{y_1...y_s}\prod_{j=s+1}^k y_j^{2\lambda_j} dy, \ x\in (0,\infty)^k,
 \end{align*}
 and $H_\infty^s$ is bounded from $L^1((0,\infty)^s, \overset{s}{\underset{j=1}{\prod}} x_j^{2\lambda_j}dx)$ into itself and
 $L_{\lambda_{s+1},\ldots,\lambda_k}$ is bounded from $L^1((0,\infty)^{k-s}, \overset{k}{\underset{j=s+1}{\prod}} x_j^{2\lambda_j}dx)$ into
 $L^{1,\infty}((0,\infty)^{k-s}, \overset{k}{\underset{j=s+1}{\prod}} x_j^{2\lambda_j}dx).$ Then, the operator $\mathcal{S}_{l,m}^{\lambda_1,\ldots,\lambda_n}$
 is bounded from $L^1((0,\infty)^n, \overset{n}{\underset{j=1}{\prod}} x_j^{2\lambda_j}dx)$ into
 $L^{1,\infty}((0,\infty)^n, \overset{n}{\underset{j=1}{\prod}} x_j^{2\lambda_j}dx).$\\

 Moreover, \cite[Lemma 5]{BHNV},
 $$\int_0^\infty \left|\frac{\partial}{\partial x_1} W_t^{\lambda_1}(x_1,y_1) \right|\frac{dt}{\sqrt{t}}
 \leq C \frac{x_1}{y_1^{2\lambda_1+2}}, \quad 2x_1 <y_1<\infty. $$
 Then, for each $x \in(0,\infty)^n$ , putting $\bar{y}=(y_2,\ldots,y_n),$
 $$ \mathcal{S}_{l,m}^{\lambda_1,\ldots,\lambda_n}(f)(x)\leq C\int_{2x_1}^\infty \frac{1}{y_1}\left(\sup_{t>0}\int_{(0,\infty)^{n-1}} \prod_{j=2}^n
  W_t^{\lambda_j}(x_j,y_j) |f(y_1,\bar{y})| \prod_{j=2}^n y_j^{2\lambda_j}d\bar{y}\right)dy_1.$$
 Since $H_\infty^1$ is bounded from $L^p((0,\infty),x^{2\lambda_1}dx)$ into itself and $W_*^{\lambda_2,\ldots,\lambda_n}$ is bounded from
 $L^p($ $(0,\infty)^{n-1}, \overset{n}{\underset{j=2}{\prod}} x_j^{2\lambda_j}dx)$ into itself, $\mathcal{S}_{l,m}^{\lambda_1,\ldots,\lambda_n}$ is
 bounded from $L^p((0,\infty)^n, \overset{n}{\underset{j=1}{\prod}} x_j^{2\lambda_j}dx)$ into itself, for every $1<p<\infty.$\\

 Finally, as it was proved in the proof of \cite[Lemma 3]{BHNV}, we have that
 $$
 \Big|\frac{\partial}{\partial u}W_t^\lambda(u,v)+(uv)^{-\lambda}\frac{u-v}{4\sqrt{\pi}t^{3/2}}e^{-(u-v)^2/4t}\Big|\le C\frac{e^{-(u-v)^2/4t}}{\sqrt{t}}(uv)^{-\lambda-1/2},\,\,\,uv\ge t.
 $$
 Then, by using \eqref{A0} and \eqref{A3}, we obtain that, for every $1\le l\leq m\leq n,$
 \begin{align*}
    \mathbb{S}&_{l,m}^{\lambda_1, \ldots,\lambda_n}(f)(x)=\int_{\frac{x_1}{2}}^{2x_1}\ldots \int_{\frac{x_l}{2}}^{2x_l}\int_0^{\frac{x_{l+1}}{2}} \ldots \int_0^{\frac{x_{m}}{2}}\int_{2x_{m+1}}^\infty \ldots \int_{2x_{n}}^\infty
 |R_{1,1}^{\lambda_1,\ldots,\lambda_n}(x,y)||f(y)| \prod_{j=1}^n y_j^{2\lambda_j}dy\\
    \leq& C \int_{\frac{x_1}{2}}^{2x_1}\ldots \int_{\frac{x_l}{2}}^{2x_l}\int_0^{\frac{x_{l+1}}{2}} \ldots \int_0^{\frac{x_{m}}{2}}\int_{2x_{m+1}}^\infty \ldots \int_{2x_{n}}^\infty |f(y)|
    \left[ \int_0^{x_1y_1}\frac{e^{-(x_1-y_1)^2/4t}}{t}(x_1y_1)^{-\lambda_1-1/2} \right.\\
    & \times \prod_{j=2}^n W_t^{\lambda_j}(x_j,y_j)dt
    + \left. \int_{x_1y_1}^\infty\left( \frac{e^{-x_1^2/8t}}{t^{\lambda_1+3/2}} + (x_1y_1)^{-\lambda_1} \frac{e^{-(x_1-y_1)^2/8t}}{t^{3/2}} \right)\prod_{j=2}^n W_t^{\lambda_j}(x_j,y_j)dt\right] \prod_{j=1}^n y_j^{2\lambda_j} dy\\
    \leq& C\left[ \int_0^{2x_1^2}\frac{dt}{\sqrt{t}}\sup_{t>0}\int_{\frac{x_1}{2}}^{2x_1}\int_{(0,\infty)^{n-1}} \frac{e^{-(x_1-y_1)^2/4t}}{\sqrt{t}}(x_1y_1)^{-\lambda_1-1/2}
    \prod_{j=2}^n W_t^{\lambda_j}(x_j,y_j)|f(y)| \prod_{j=1}^ny_j^{2\lambda_j}dy \right.\\
    & + \left. \left( \int_{\frac{x_1^2}{2}}^\infty \frac{dt}{t^{\lambda_1+3/2}} + x_1^{-2\lambda_1}\int_{\frac{x_1^2}{2}}^\infty \frac{dt}{t^{3/2}}\right) \sup_{t>0} \int_{\frac{x_1}{2}}^{2x_1} \int_{(0,\infty)^{n-1}} \prod_{j=2}^n W_t^{\lambda_j}(x_j,y_j)|f(y)| \prod_{j=1}^ny_j^{2\lambda_j}dy\right]\\
    \leq & C\left[ \sup_{t>0}\int_{\frac{x_1}{2}}^{2x_1}\int_{(0,\infty)^{n-1}} \frac{e^{-(x_1-y_1)^2/4t}}{\sqrt{t}}(x_1y_1)^{-\lambda_1} \prod_{j=2}^n W_t^{\lambda_j}(x_j,y_j)|f(y)| \prod_{j=1}^ny_j^{2\lambda_j}dy \right.\\
    & + \left. \sup_{t>0}\int_{(0,\infty)^{n-1}} \prod_{j=2}^n W_t^{\lambda_j}(x_j,y_j) \frac{1}{x_1^{2\lambda_1+1}}\int_{\frac{x_1}{2}}^{2x_1}f(y) y_1^{2\lambda_1}dy_1\prod_{j=2}^ny_j^{2\lambda_j}dy \right].
 \end{align*}
 Since the operator $H_{loc}^{\lambda_1}$ is bounded from $L^1((0,\infty),x_1^{2\lambda_1+1}dx_1)$ into itself, the maximal operator
 $W_*^{\lambda_2, \ldots, \lambda_n}$ is bounded from $L^1((0,\infty)^{n-1},\overset{n}{\underset{j=2}{\prod}} x_j^{2\lambda_j} dx)$ into
 $L^{1,\infty}((0,\infty)^{n-1},\overset{n}{\underset{j=2}{\prod}} x_j^{2\lambda_j} dx)$ (Theorem \ref{maximal}), and, as the arguments developed in the proof of Theorem \ref{maximal} show, the maximal operator
 $$ \mathcal{W}_*^{\lambda_1,\ldots,\lambda_n}(f)(x)= \sup_{t>0} \int_{\frac{x_1}{2}}^{2x_1}\int_{(0,\infty)^{n-1}} \frac{e^{-(x_1-y_1)^2/4t}}{\sqrt{t}}
 (x_1y_1)^{-\lambda_1}\prod_{j=2}^n W_t^{\lambda_j}(x_j,y_j) f(y) \prod_{j=1}^n y_j^{2\lambda_j} dy,$$
 is bounded from $L^1((0,\infty)^{n},\overset{n}{\underset{j=1}{\prod}} x_j^{2\lambda_j} dx)$ into
 $L^{1,\infty}((0,\infty)^{n},\overset{n}{\underset{j=1}{\prod}} x_j^{2\lambda_j} dx),$ for each $x\in(0,\infty)^n.$
 We conclude that the operator $\mathbb{S}_{l,m}^{\lambda_1, \ldots, \lambda_n}$ is bounded from
 $L^1((0,\infty)^{n},\overset{n}{\underset{j=1}{\prod}} x_j^{2\lambda_j} dx)$ into
 $L^{1,\infty}((0,\infty)^{n},$ $\overset{n}{\underset{j=1}{\prod}} x_j^{2\lambda_j} dx).$\\

 Moreover, since the operator $H_{loc}^{\lambda_1},$ $W_*^{\lambda_2,\ldots,\lambda_n}$ and $W_*^{\lambda_1, \ldots, \lambda_n}$ are bounded from
 $L^p((0,\infty),$ $x^{2\lambda_1}dx)$ into itself, $L^p((0,\infty)^{n-1},\overset{n}{\underset{j=2}{\prod}} x_j^{2\lambda_j} dx)$ into itself, and
 $L^p((0,\infty)^{n},\overset{n}{\underset{j=1}{\prod}} x^{2\lambda_j} dx)$ into itself, respectively, $\mathbb{S}_{l,m}^{\lambda_1,\ldots,\lambda_n}$
 is bounded from $L^p((0,\infty)^{n},\overset{n}{\underset{j=1}{\prod}} x_j^{2\lambda_j} dx)$ into itself, for every $1<p<\infty.$\\

 Hence, we have proved that the operator $R_{1,1}^{\lambda_1,\ldots,\lambda_n}$ defined by
 $$ R_{1,1}^{\lambda_1,\ldots,\lambda_n}(f)(x)= \int_{(0,\infty)^n}|R_{1,1}^{\lambda_1,\ldots,\lambda_n}(x,y)| f(y)\prod_{j=1}^ny_j^{2\lambda_j} dy, \quad x\in (0,\infty)^n,$$
 is bounded from $L^1((0,\infty)^{n},\overset{n}{\underset{j=1}{\prod}} x_j^{2\lambda_j} dx)$ into $L^{1,\infty}((0,\infty)^{n},\overset{n}{\underset{j=1}{\prod}} x_j^{2\lambda_j} dx)$ and from
 $L^p((0,\infty)^{n},$ $\overset{n}{\underset{j=1}{\prod}} x_j^{2\lambda_j} dx)$ into itself, for every $1<p<\infty.$ Since
 \begin{align*}
    R_{1,1,*}^{\lambda_1,\ldots,\lambda_n}(f)(x)=&\sup_{\varepsilon>0}\left| \int_{|x-y|>\varepsilon} R_{1,1}^{\lambda_1,\ldots,\lambda_n}(x,y)f(y) \prod_{j=1}^n y_j^{2\lambda_j}dy\right|\\
    \leq & R_{1,1}^{\lambda_1,\ldots,\lambda_n}(f)(x), \quad x\in (0,\infty)^n,
 \end{align*}
 $R_{1,1,*}^{\lambda_1,\ldots,\lambda_n}$ is bounded from $L^1((0,\infty)^{n},\overset{n}{\underset{j=1}{\prod}} x_j^{2\lambda_j} dx)$ into
 $L^{1,\infty}((0,\infty)^{n},\overset{n}{\underset{j=1}{\prod}} x_j^{2\lambda_j} dx)$ and from
 $L^p($ $(0,\infty)^{n},\overset{n}{\underset{j=1}{\prod}} x_j^{2\lambda_j} dx)$ into itself, for every $1<p<\infty.$\\

 In order to study the maximal operators $R_{1,l,*}^{\lambda_1,\ldots,\lambda_n},$ $l=2,\ldots, n,$ we can proceed as for
 $R_{1,1,*}^{\lambda_1,\ldots,\lambda_n}$ by taking into account \eqref{B15}.\\

 We now consider the maximal operator
 $$ R_{1,n+1,*}^{\lambda_1,\ldots,\lambda_n}(f)(x)= \sup_{\varepsilon>0}\left| \int_{|x-y|>\varepsilon} R_{1,n+1}^{\lambda_1,\ldots,\lambda_n} (x,y)f(y)\prod_{j=1}^n y_j^{2\lambda_j}dy \right|, \quad x\in (0,\infty)^n.$$
 We use now a procedure similar than the one developed in the proof of Claim \ref{Claim1} in Section 3.\\

 For every $m=(m_1,\ldots,m_n)\in \mathbb{Z}^n$ we define the operator $\mathcal{Y}_{m}^{\lambda_1,\ldots,\lambda_n}$ by
 \begin{align*}
    \mathcal{Y}_{m}^{\lambda_1,\ldots,\lambda_n}(f)(x)= &\sup_{\varepsilon>0}\left| \chi_{\overset{n}{\underset{j=1}{\prod}}
    [2^{m_j}, 2^{m_j+1})} (x) \prod_{j=1}^n x_j^{-\lambda_j}\frac{1}{\sqrt{\pi}} \int_{|x-y|>\varepsilon} f_m(y)\cdot \right.\\
    & \cdot\left. \int_0^\infty \frac{\partial}{\partial x_1}\mathbb{W}_t (x_1,y_1) \prod_{j=2}^n \mathbb{W}_t(x_j,y_j) \frac{dt}{\sqrt{t}}dy \right|, \quad x\in (0,\infty)^n,
 \end{align*}
 where $f_m(y)= \chi_{\overset{n}{\underset{j=1}{\prod}} [2^{m_j-1}, 2^{m_j+2})} (x) f(y) \overset{n}{\underset{j=1}{\prod}} y_j^{\lambda_j},$ $y\in (0,\infty)^n.$\\

 Let $\gamma>0.$ We have that
 \begin{align*}
    m_{\lambda_1,\ldots,\lambda_n}&(\{x\in(0,\infty)^n : R_{1,n+1,*}^{\lambda_1,\ldots,\lambda_n} (f)(x)>\gamma \})\\
    \leq& m_{\lambda_1,\ldots,\lambda_n}(\{x\in(0,\infty)^n : | R_{1,n+1,*}^{\lambda_1,\ldots,\lambda_n} (f)(x)-\sum_{m\in \mathbb{Z}^n} \mathcal{Y}_m^{\lambda_1,\ldots,\lambda_n}(f)(x)| >\gamma/2 \})\\
    & + \sum_{m\in \mathbb{Z}^n} m_{\lambda_1, \ldots, \lambda_n}(\{ x\in \prod_{j=1}^n[2^{m_j},2^{m_j+1}) :\mathcal{Y}_m^{\lambda_1,\ldots,\lambda_n}(f)(x) >\gamma/2 \}).
 \end{align*}

 We denotes by $R_{1,*}$ the maximal operator
 $$ R_{1,*}(g)(x) = \sup_{\varepsilon>0} \left| \int_{|x-y|>\varepsilon} \int_0^\infty \frac{\partial}{\partial x_1} \mathbb{W}_t(x_1,y_1) \prod_{j=2}^n
 \mathbb{W}_t(x_j,y_j)\frac{dt}{\sqrt{t}}g(y) dy \right|,  x\in \mathbb{R}^n.$$

 It is well known that $R_{1,*}$ is a bounded operator from $L^1(\mathbb{R}^n,dx)$ into $L^{1,\infty}(\mathbb{R}^n,dx)$ and from $L^p(\mathbb{R}^n,dx)$
 into itself, for every $1<p<\infty$.\\

 Then, for every $m=(m_1,\ldots,m_n)\in \mathbb{Z}^n,$ it gets
 \begin{align*}
    m_{\lambda_1,\ldots,\lambda_n}&(\{ x\in \prod_{j=1}^n[2^{m_j},2^{m_j+1}) :\mathcal{Y}_m^{\lambda_1,\ldots,\lambda_n}(f)(x) >\gamma/2 \})\\
    \leq & C 2^{2\overset{n}{\underset{j=1}{\sum}} \lambda_jm_j} m_0^{(n)} (\{x\in \overset{n}{\underset{j=1}{\prod}} [2^{m_j},2^{m_j+1}): R_{1,*}(f_m)({x})>M\gamma 2^{\overset{n}{\underset{j=1}{\sum}} \lambda_jm_j}\})\\
    \leq & \frac{C}{\gamma} 2^{\overset{n}{\underset{j=1}{\sum}} \lambda_jm_j}\|f_m\|_{L^1(\mathbb{R}^n,dx)}\\
    \leq & \frac{C}{\gamma} \int_{\overset{n}{\underset{j=1}{\prod}} [2^{m_j-1},2^{m_j+2})} |f(y)| \prod_{j=1}^n y_j^{2\lambda_j}dy.
 \end{align*}
 Here $M$ and $C$ denotes suitable positive constants that do not depend on $m\in \mathbb{Z}^n.$ By $m_0^{(n)}$ we represent the Lebesgue measure in $\mathbb{R}^n.$
 Hence,
 \begin{align}\label{C16}
    \sum_{m\in \mathbb{Z}^n} m_{\lambda_1,\ldots,\lambda_n}& (\{ x\in \prod_{j=1}^n[2^{m_j},2^{m_j+1}) :\mathcal{Y}_m^{\lambda_1,\ldots,\lambda_n}(f)(x) >\gamma/2 \})\\
    \leq & \frac{C}{\gamma}\|f\|_{L^1((0,\infty)^n, \overset{n}{\underset{j=1}{\prod}} x_j^{2\lambda_j})}. \nonumber
 \end{align}
 Also, if $1<p<\infty,$ we can write
 \begin{align}\label{C17}
    \|\sum_{m\in \mathbb{Z}^n} \mathcal{Y}_{m}^{\lambda_1,\ldots,\lambda_n}(f)\|_p^p= &\sum_{m\in \mathbb{Z}^n} \int_{\overset{n}{\underset{j=1}{\prod}} [2^{m_j},2^{m_j+1})} |\mathcal{Y}_m^{\lambda_1,\ldots,\lambda_n} (f)(x)|^p \prod_{j=1}^n x_j^{2\lambda_j}dx\nonumber\\
    \leq & C\sum_{m\in \mathbb{Z}^n} 2^{(2-p)\overset{n}{\underset{j=1}{\sum}} m_j\lambda_j}\int_{\mathbb{R}^n} |R_{1,*}(f_m)(x)|^p dx\nonumber\\
    \leq & C\sum_{m\in \mathbb{Z}^n} 2^{(2-p)\overset{n}{\underset{j=1}{\sum}} m_j\lambda_j} \int_{\overset{n}{\underset{j=1}{\prod}} [2^{m_j-1},2^{m_j+2})} |f(y)|^p \prod_{j=1}^n y_j^{p\lambda_j}dy\nonumber\\
    \leq & C\sum_{m\in \mathbb{Z}^n}  \int_{\overset{n}{\underset{j=1}{\prod}} [2^{m_j-1},2^{m_j+2})} |f(y)|^p \prod_{j=1}^n y_j^{2\lambda_j}dy\nonumber\\
    \leq & C\|f\|_p^p.
 \end{align}
 On the other hand, we have that
 \begin{align*}
    m_{\lambda_1,\ldots,\lambda_n}&(\{x\in(0,\infty)^n : | R_{1,n+1,*}^{\lambda_1,\ldots,\lambda_n} (f)(x)-\sum_{m\in \mathbb{Z}^n} \mathcal{Y}_m^{\lambda_1,\ldots,\lambda_n}(f)(x) |>\gamma/2 \})\\
    \leq &  \sum_{m\in \mathbb{Z}^n} m_{\lambda_1,\ldots,\lambda_n} (\{x\in\prod_{j=1}^n[2^{m_j},2^{m_j+1}) : | R_{1,n+1,*}^{\lambda_1,\ldots,\lambda_n} (f)(x)- \mathcal{Y}_m^{\lambda_1,\ldots,\lambda_n}(f)(x)| >\gamma/2 \}).
 \end{align*}
 Let $m=(m_1,\ldots,m_n)\in \mathbb{Z}^n.$ We define, for every $x\in \underset{j\in \mathbb{N}}{\prod}[2^{m_j},2^{m_j+1}),$ the sets $B_l(x)$ and $C_l(x),$ $l=1,\ldots,n,$ and $A(x)$ as in the proof of Theorem \ref{gfunction}. Then,
 \begin{align*}
    |R_{1,n+1,*}^{\lambda_1,\ldots,\lambda_n}& (f)(x)- \mathcal{Y}_m^{\lambda_1,\ldots,\lambda_n}(f)(x)|\\
    \leq & \frac{1}{\sqrt{\pi}}\sum_{l=1}^n\left( \prod_{j=1}^n x_j^{-\lambda_j} \int_{B_l(x)} |f(y)| \int_0^\infty \Big|\frac{\partial}{\partial x_1} \mathbb{W}_t(x_1,y_1)\Big|\prod_{j=2}^n \mathbb{W}_t(x_2,y_2)\frac{dt}{\sqrt{t}}\prod_{j=1}^n y_j^{\lambda_j}dy\right.\\
    & \left. + \prod_{j=1}^n x_j^{-\lambda_j} \int_{C_l(x)} |f(y)| \int_0^\infty \Big|\frac{\partial}{\partial x_1} \mathbb{W}_t(x_1,y_1)\Big|\prod_{j=2}^n
    \mathbb{W}_t(x_2,y_2)\frac{dt}{\sqrt{t}}\prod_{j=1}^n y_j^{\lambda_j}dy\right)
\end{align*}
\begin{align*}
    \leq & C \sum_{l=1}^n\left(  \int_{B_l(x)} |f(y)| \int_0^\infty \frac{e^{-|x-y|^2/8t}}{t^{n/2+1}} dt dy + \int_{C_l(x)} |f(y)| \int_0^\infty \frac{e^{-|x-y|^2/8t}}{t^{n/2+1}} dt dy\right)\\
    \leq & C \sum_{l=1}^n\left(  \int_{B_l(x)} \frac{|f(y)|}{|x-y|^n}dy +  \int_{C_l(x)} \frac{|f(y)|}{|x-y|^n}dy \right), \quad x\in \prod_{j=1}^n [2^{m_j}, 2^{m_j+1}).
 \end{align*}
 By denoting $\bar{y}=(y_2,\ldots,y_n) \in \mathbb{R}^{n-1},$ and $\displaystyle \tilde{f}_m(\bar{y}) = \chi_{\overset{n}{\underset{j=2}{\prod}} [2^{m_j-1}, 2^{m_j+2})} (\bar{y})\int_{2^{m_1-1}}^{2^{m_1+2}} |f(y_1,\bar{y})|dy_1,$ we obtain
 $$ \int_{B_1(x)}\frac{|f(y)|}{|x-y|^n}dy \leq C\frac{1}{2^{m_1}} M_{n-1}(\tilde{f}_m)(\bar{x}),$$
 where $\bar{x}=(x_2, \ldots,x_n)$ when $x=(x_1,x_2,\ldots,x_n),$ and $M_{n-1}(\tilde{f}_m)(\bar{x})$ denotes the Hardy Littlewood maximal function in $\mathbb{R}^{n-1}.$\\

 Then, for every $\beta>0,$
 \begin{align}\label{C18}
    m_{\lambda_1,\ldots,\lambda_n}& \Big(\Big\{x\in\prod_{j=1}^n[2^{m_j},2^{m_j+1}) :\int_{B_1(x)}\frac{f(y)}{|x-y|^n}dy>\beta\Big\}\Big)\nonumber\\
    \leq & m_{\lambda_1,\ldots,\lambda_n} \Big(\Big\{x\in\prod_{j=1}^n[2^{m_j},2^{m_j+1}) :\frac{C}{2^{m_1}} M_{n-1}(\tilde{f}_m)(\bar{x})>\beta\Big\}\Big)\nonumber\\
    \leq & M 2^{m_1}2^{2\overset{n}{\underset{j=1}{\sum}} \lambda_jm_j} m_0^{(n-1)} \Big(\Big\{x\in\mathbb{R}^{n-1} : M_{n-1}(\tilde{f}_m)(\bar{x})>\frac{\beta2^{m_1}}{C}\Big\}\Big)\nonumber\\
    \leq & \frac{M}{\beta} \int_{\overset{n}{\underset{j=1}{\prod}}[2^{m_j-1},2^{m_j+2})} |f(y)|\prod_{j=1}^ny_j^{2\lambda_j}dy,
 \end{align}
 where $M>0$ does not depend on $m\in \mathbb{Z}^n.$
 Also, if $1<p<\infty,$ by using Jensen's inequality and maximal Theorem we have
 \begin{align}\label{C19}
    \int_{\overset{n}{\underset{j=1}{\prod}}[2^{m_j},2^{m_j+1})} \left| \int_{B_1(x)} \frac{|f(y)|}{|x-y|^n}\right|^p \prod_{j=1}^n x_j^{2\lambda_j}dx \leq &C2^{2\overset{n}{\underset{j=1}{\sum}} \lambda_jm_j -m_1(p-1)}\int_{\mathbb{R}^{n-1}} |M_{n-1}(\tilde{f}_m) (\bar{x})|^p dx \nonumber \\
   \leq &  C \int_{\overset{n}{\underset{j=1}{\prod}}[2^{m_j-1},2^{m_j+2})}  |f(y)|^p \prod_{j=1}^n y_j^{2\lambda_j}dy,
 \end{align}
 where $C>0$ does not depend on $m\in \mathbb{Z}^n.$\\

 Estimates similar than \eqref{C18} and \eqref{C19} can be obtain for $B_l(x),$ $l=2,\ldots, n,$ and $C_l(x),$ $l=1,\ldots, n.$ Putting together all the estimates we get
 \begin{align*}
    m_{\lambda_1,\ldots,\lambda_n} (\{x\in\prod_{j=1}^n[2^{m_j},2^{m_j+1}) &: | R_{1,n+1,*}^{\lambda_1,\ldots,\lambda_n} (f)(x)- \mathcal{Y}_m^{\lambda_1,\ldots,\lambda_n}(f)(x)| >\gamma/2 \})\\
    \leq & \frac{C}{\gamma} \int_{\overset{n}{\underset{j=1}{\prod}}[2^{m_j-1},2^{m_j+2})} |f(y)|\prod_{j=1}^n y_j^{2\lambda_j}dy
 \end{align*}
 and, for every $1<p<\infty,$
 \begin{align*}
    \int_{\overset{n}{\underset{j=1}{\prod}}[2^{m_j},2^{m_j+1})} |R_{1,n+1,*}&^{\lambda_1,\ldots,\lambda_n} (f)(x)- \mathcal{Y}_m^{\lambda_1,\ldots,\lambda_n}(f)(x)|^p \prod_{j=1}^n x_j^{2\lambda_j}dx \\
    & \leq  C \int_{\overset{n}{\underset{j=1}{\prod}}[2^{m_j-1},2^{m_j+1})} |f(y)|^p \prod_{j=1}^n y_j^{2\lambda_j}dy.
 \end{align*}
 Hence, by summing in $m\in \mathbb{Z}^n,$ it follows
 \begin{align}\label{C20}
    m_{\lambda_1,\ldots,\lambda_n} (\{x\in (0,\infty)^n &: | R_{1,n+1,*}^{\lambda_1,\ldots,\lambda_n} (f)(x)-\sum_{m\in \mathbb{Z}^n} \mathcal{Y}_m^{\lambda_1,\ldots,\lambda_n}(f)(x)| >\gamma/2 \})\nonumber\\
    \leq & \frac{C}{\gamma}\|f\|_{L^1((0,\infty)^n, \overset{n}{\underset{j=1}{\prod}} x_j^{2\lambda_j})},
 \end{align}
 and, for every $1<p<\infty,$
 \begin{equation}\label{C21}
    \| R_{1,n+1,*}^{\lambda_1,\ldots,\lambda_n} (f)(x)-\sum_{m\in \mathbb{Z}^n} \mathcal{Y}_m^{\lambda_1,\ldots,\lambda_n}(f)(x)\|_p \leq C\|f\|_p.
 \end{equation}
 From \eqref{C16} and \eqref{C20} we deduce that $R_{1,n+1,*}^{\lambda_1, \ldots,\lambda_n}$ is bounded from
 $L^1((0,\infty)^n, \overset{n}{\underset{j=1}{\prod}}x_j^{2\lambda_j}dx)$ into $L^{1,\infty}((0,\infty)^n, \overset{n}{\underset{j=1}{\prod}}x_j^{2\lambda_j}dx),$
 and from \eqref{C17} and \eqref{C21} it infers that $R_{1,n+1,*}^{\lambda_1, \ldots,\lambda_n}$ is bounded from
 $L^p((0,\infty)^n, \overset{n}{\underset{j=1}{\prod}}x_j^{2\lambda_j}dx)$ into itself, for every $1<p<\infty.$\\

 Thus the proof of Proposition \ref{Proposition 4.3} is finished.
 \end{proof}

 According to Proposition \ref{Proposition 4.2} and \ref{Proposition 4.3} standard arguments allow us to conclude that, for every $f\in L^p((0,\infty)^n, \overset{n}{\underset{j=1}{\prod}} x_j^{2\lambda_j}dx),$ $1\leq p<\infty,$ and $i=1,\ldots,n,$ there exists the limit
 $$ \lim_{\varepsilon\rightarrow 0^+} \int_{|x-y|>\varepsilon} R_i^{\lambda_1,\ldots,\lambda_n}(x,y) f(y) dy, \quad \text{ a. e. } x\in (0,\infty)^n.$$

 We define, for every $f\in L^p((0,\infty)^n, \overset{n}{\underset{j=1}{\prod}} x_j^{2\lambda_j}dx),$ $1\leq p<\infty,$ and $i=1,\ldots,n,$ the Riesz transform $R_i^{\lambda_1,\ldots,\lambda_n}(f) $ of $f$ by
 $$R_i^{\lambda_1,\ldots,\lambda_n}(f)(x)= \lim_{\varepsilon\rightarrow 0^+} \int_{|x-y|>\varepsilon} R_i^{\lambda_1,\ldots,\lambda_n}(x,y) f(y) dy, \quad \text{ a.e. } x\in (0,\infty)^n.$$
 Note that, by Proposition \ref{Proposition 4.2}, for every $i=1,\ldots,n$ this definition extends the initial definition of Riesz transform
 $ R_i^{\lambda_1,\ldots,\lambda_n}$ from $C_c^\infty((0,\infty)^n)$ to $L^p((0,\infty)^n, \overset{n}{\underset{j=1}{\prod}} x_j^{2\lambda_j}dx),$ $1\leq p<\infty.$\\

 Finally, from Proposition \ref{Proposition 4.3} we infer the desired $L^p$-boundedness properties for the Riesz transform
 $R_i^{\lambda_1,\ldots,\lambda_n},$ $i=1,\ldots,n,$ and the proof of Theorem \ref{Riesz} is complete.

\end{document}